\documentclass[a4paper,english,reqno,10pt]{amsart}
\pdfoutput=1
\usepackage{amsfonts,amsmath,amssymb,color,amsthm}
\usepackage{enumerate}
\usepackage{color}
\usepackage[T1]{fontenc}
\usepackage[utf8]{inputenc}
\usepackage[english]{babel}

\usepackage[all]{xy}
\usepackage{rotating}
\usepackage[mathscr]{eucal} 
\usepackage{mathabx}%?
\usepackage{multicol}

\usepackage[tracking=true]{microtype}%

\usepackage{mathtools}

\usepackage{caption}

\usepackage[
   style=alphabetic,
   doi=false,
   url=false,
   isbn=false,
   clearlang=true,
   backend=bibtex
   ]{biblatex}
\newtheorem{theorem}{Theorem}

\newtheorem{lemma}{Lemma}[section]
\newtheorem{thm}[lemma]{Theorem}
\newtheorem{corollary}[lemma]{Corollary}
\newtheorem{prop}[lemma]{Proposition}

\def\acts{\mathrel{\reflectbox{$\righttoleftarrow$}}}

\theoremstyle{definition}
\newtheorem{definition}[lemma]{Definition}

\theoremstyle{remark}
\newtheorem{remark}[lemma]{Remark}
\newtheorem{example}[lemma]{Example}

\newcommand\Diff{{\mathrm{Diff}}}
\newcommand\Spec{{\mathrm{Spec\ }}}

\newcommand{\Z}{{\ensuremath{\mathbb{Z}}}}
\newcommand{\Q}{{\ensuremath{\mathbb{Q}}}}
\newcommand{\R}{{\ensuremath{\mathbb{R}}}}
\newcommand{\C}{{\ensuremath{\mathbb{C}}}}
\newcommand{\Pp}{{\ensuremath{\mathbb{P}}}}

\newcommand{\F}{{\ensuremath{\mathbb{F}}}}
\newcommand{\A}{{\ensuremath{\mathbb{A}}}}

\newcommand{\f}{{\ensuremath{\mathscr{F}}}}

\newcommand{\Bir}{{\mathrm{Bir}}}

\newcommand{\PGL}{{\mathrm{PGL}}}
\newcommand{\Ima}{{\mathrm{Im}}}
\newcommand{\coker}{{\mathrm{coker}}}
\newcommand\GL{{\mathrm{GL}}}
\newcommand\Cc{{\mathrm{C}}}
\newcommand\Fix{{\mathrm{Fix}}}
\newcommand\mfp{{\mathfrak p}}
\newcommand{\rk}{{\mathrm{rk}}}

\newcommand{\Aut}{{\mathrm{Aut}}}
\newcommand{\Pic}{{\mathrm{Pic}}}

\title{Prime order birational diffeomorphisms of the sphere}
\author{Maria Fernanda Robayo}
\subjclass[2010]{14E07; 14P25; 14J26; 53A05}
\begin{document}
\maketitle
\centerline{\today}
\begin{abstract}
The aim of this paper is to give the classification of conjugacy classes of elements of prime order in the group of birational diffeomorphisms of the two-dimensional real sphere. Parametrisations of conjugacy classes by moduli spaces are presented.
\end{abstract}\section{Introduction}
Let $\Pp^n_\R$ denote the projective $n$-space as a scheme over~$\R$. A real projective variety $X\subset\Pp^n_\R$ is a scheme over $\R$ which may be thought of as a pair $(X_\C,\sigma)$, where $X_\C$ is its complexification, i.e. $X_\C:=X\times_{\Spec\R}\Spec\C$, and $\sigma$ is an anti-holomorphic involution on $X_\C$. Let $X(\C)$ denote the set of complex points of $X$ and $X(\R):=X(\C)^\sigma$ (the invariant points under~$\sigma$) the real part of $X$. Supposing that $X$ is smooth and $X(\R)$ is nonempty, we can endow $X(\R)$ with the Euclidian topology and obtain a manifold of real dimension $m=\dim_{\C} X_\C$ over $\R$.

There are then two kinds of regular morphisms between real algebraic varieties $X$, $Y$ studied in the literature (see for example the introductions of \cite{Kol01a} and \cite{BM11}):
\begin{enumerate}
\item A regular morphism $X\to Y$ is a rational map defined at all complex points. The corresponding category is the one of schemes defined over $\R$, together with regular morphisms of schemes. The group of automorphisms is denoted by $\Aut(X)$, which is in general quite small: The connected component of the identity is an algebraic group of finite dimension.
\item The second notion of regular morphisms consists of taking rational maps $X\dasharrow Y$ that are defined only at all real points of $X$, such maps will be called morphisms $X(\R)\to Y(\R)$. This gives another category, with more morphisms where the objects are $X(\R)$. The corresponding group of automorphisms will be denoted by $\Aut(X(\R))$ and is the same as the set of birational diffeomorphisms of the algebraic variety considered.
\end{enumerate}

In most real algebraic geometry texts, the second category, much richer, is in fact studied.

In \cite{BiHu07}, I.~Biswas and J.~Huisman showed that if $X$ and $Y$ are two rational real compact surfaces, then $X(\R)$ and $Y(\R)$ are diffeomorphic if and only if $X(\R)$ and $Y(\R)$ are isomorphic (which corresponds to saying that there is a birational diffeomorphism between $X$ and $Y$). The proof of this result was simplified by J.~Huisman and F.~Mangolte in \cite{HM09}, by proving first that $\Aut(X(\R))$ acts $n$-transitively on $X(\R)$ for each $n$. The same question for geometrically rational surfaces (i.e. rational over $\C$) were then studied in \cite{BM11} by J.~Blanc and F.~Mangolte.

The group $\Aut(X(\R))$ is really larger than $\Aut(X)$ in general. In particular, J.~Koll\'ar and F.~Mangolte showed in \cite{KM09} that $\Aut(X(\R))$ is dense in $\Diff(X(\R))$ if $X$ is a smooth real compact rational surface.

Some other information on the group $\Aut(X(\R))$ can be given by looking at its elements of finite order. In particular, in this text we are interested in elements of prime order of $\Aut(S(\R))$ up to conjugacy, where $S(\R)$ is the standard two-dimensional sphere (see Section~\ref{Ch:results}). The group $\Aut(S(\R))$ is contained in the group $\Bir(S)$ of real birational transformations of the sphere, which is isomorphic to the real Cremona group $\Bir(\Pp^2_\R)$. This latter group is, of course, contained in  the complex Cremona group $\Bir(\Pp^2_\C)$. The problem of classification of conjugacy classes of elements of finite order in $\Bir(\Pp^2_\C)$ (which contains the groups $\Bir(X)$ described before) have been of interest for a lot of mathematicians. The first classification was the one of E.~Bertini (\cite{Bertini}), who studied involutions. The decomposition into three types of maps, namely Bertini involutions, Geiser involutions, and Jonqui\`eres involutions, was correct but there is some redundancy because the curves of fixed points were not considered.
 A modern and complete proof was obtained by L.~Bayle and A.~Beauville in \cite{BaBe00}, using the tools of the minimal model program developed in dimension $2$ by Yu.~Manin (\cite{Ma67}) and V.I.~Iskovskikh (\cite{Isk80}). They obtain parametrisations of the conjugacy classes by the associated fixed curves.
T.~de Fernex generalised the classification in \cite{dFe04} for elements of prime order (except for one case, done in \cite{BeB04} by A. Beauville and J. Blanc). See also \cite{bib:Zha} for another approach to the same question. The precise classification of elements of finite order was then obtained in \cite{BlaCyclic} by J.~Blanc, using the description of finite groups of I.~Dolgachev and V.I.~Iskovskikh \cite{DoIsk12}\footnote{Also after \cite{DoIsk12}, there are still open questions on finite subgroups of $\Bir(\Pp^2_\C)$ left, some of them answered in the recent paper \cite{Tsygan}.}. Again, the parametrisations are given by fixed curves (of powers of elements), but also by actions of the elements on the curves.

In this text, we obtain the results for the analogous problem of classification for elements of prime order in the group $\Aut(S(\R))$. The classification is summarised in Section~\ref{Ch:results} (Theorem~\ref{MainThm}), which states that there are eight different families of conjugacy classes, some with only one element and others with infinitely many elements. The second main result is concerning the parametrisation of the conjugacy classes in each family (Theorem~\ref{MainThm2}). As $\Aut(S(\R))\subset\Bir(\Pp^2_\C)$, it is possible to compare the classification of the birational diffeomorphisms with the complex case i.e. birational transformations of the complex plane. For instance, there are three families of involutions on $\Bir(\Pp^2_\C)$: Bertini, Geiser, and de Jonqui\`eres. Bertini involutions do not occur in the group $\Aut(S(\R))$ because they would come from an automorphism of a Del Pezzo surface of degree $1$ after blowing up at least one real point of $S$, which would damage the geometry of the real points; see Proposition~\ref{p} in Section~\ref{Ch:pairs}. The Geiser involution of $\Aut(S(\R))$ corresponds to real quartics with one oval. Moreover, the group $\Aut(S(\R))$ contains distinct families of conjugacy classes of involutions of de Jonqui\`eres type, which are all conjugate in $\Bir(\Pp^2_\C)$, in particular, one family, containing uncountable many elements non conjugate to each other, corresponds to only one conjugacy class in $\Bir(\Pp^2_\C)$.

This text is organised as follows. Section~\ref{Ch:results} contains the compilation of the results of this text presented in two main statements and examples of birational diffeomorphisms of the sphere. In Section~\ref{Ch:pairs}, it is shown why the study of conjugacy classes of elements of finite order of the group of birational diffeomorphisms corresponds to the study of pairs $(X,g)$ consisting of a smooth rational projective surface $X$ and $g$ an automorphism of $X$. More precisely, there are two cases to focus on, say, when $X$ is a Del Pezzo surface whose real Picard group invariant by $g$ is isomorphic to $\Z$, and when $X$ admits a conic bundle structure and the real Picard group invariant by $g$ has rank $2$. This is a result given by V.I.~Iskovskikh (\cite{Isk80}) and in this section, it is given more specifically what pairs are obtained for the sphere (Proposition~\ref{MinMod}). In particular, since the sphere admits a structure of conic bundle given by the projection to one of the affine coordinates, Proposition~\ref{MinMod} gives that the morphism of the conic bundle structure for a pair $(X,g)$, when $X$ admits one, factors through that projection of the sphere. Section~\ref{DPcase} is devoted to the study of pairs $(X,g)$ when $X$ is a Del Pezzo surface, including the case of the sphere itself. Special automorphisms of Del Pezzo surfaces of degree $2$ and $4$ such as Geiser involution and automorphisms $\alpha_1$, $\alpha_2$ that are studied in Subsections~\ref{dP2} and~\ref{dP4} bring on two different families of conjugacy classes on the sphere. In Subsection~\ref{dP8}, the conjugacy classes of the group of automorphisms of the sphere are investigated (Proposition~\ref{ConjAut(S)}).

Section~\ref{Ch:ConicCase} is dedicated to the study of the birational diffeomorphisms that are compatible with the conic bundle structure of the sphere, which is a $\Pp^1$-fibration not locally trivial. It is natural to understand the action of a birational map on the basis of the fibration and that is done in the first Subsection. When the action on $\Pp^1$ is trivial, it is shown in Subsection~\ref{Sec:AlgDescription} that the complex model of the sphere is birational to $\A^2_\C$, which allows to give an explicit algebraic description of the birational transformations of the sphere and in the following subsection for birational diffeomorphisms. In Subsection~\ref{Sec:InvBir(S/pi)}, it is proved that two birational maps of the sphere compatible with the fibration and acting trivially on the basis of it are conjugate in the group of birational maps of the sphere, if and only if there exist a birational map between the curves of fixed points of these two maps, which is defined over $\R$. This result is also proved for the group of birational diffeomorphisms in the following subsection. In addition, a geometrical characterisation of the birational diffeomorphisms of order $2$ is given according to the orientation when restricted to $S(\R)$. More precisely, it is proved that there is a one-to-one correspondence between the conjugacy classes of orientation-preserving birational diffeomorphisms of the sphere compatible with the fibration and acting trivially on the basis and smooth real projective curves with not real point, which are a $2$-$1$ covering of $\Pp^1$ up to isomorphism. For the case of orientation-reversing, they are in correspondence with smooth real projective curves with one oval, which are a $2$-$1$ covering of $\Pp^1$ up to isomorphism. In Subsections~\ref{Sec:Bir(S/pi)Order>2} and~\ref{Sec:Aut(S(R)/pi)Order>2}, for birational maps and for birational diffeomorphisms of the sphere of order larger than two which are compatible with the fibration and acting trivially on the basis, it is shown than they are conjugate to rotations of the sphere. The last subsection is concerning birational maps and birational diffeomorphisms of order two compatible with the fibration and with non-trivial action on the basis. It is constructed a bijection between conjugacy classes of birational involutions as before and classes on a second cohomology group that is isomorphic to $\oplus_{b\in\R_{>0}}\Z/2\Z$. Since the representative of these classes in the group of birational maps of the sphere are particularly birational diffeomorphisms, this implies that there are uncountable many conjugacy classes of birational diffeomorphisms of order two with a non-trivial action on the basis.

In Section~\ref{Ch:connection}, the problem that two pairs $(X,g)$, $(X',g')$ may rise the same conjugacy class in $\Aut(S(\R))$ is examined. In Subsection~\ref{Sec:Proofs}, Theorem~\ref{MainThm} and~\ref{MainThm2} are proved by putting together all results obtained in Sections~\ref{Ch:pairs}, \ref{DPcase}, \ref{Ch:ConicCase}, and \ref{Ch:connection}.
\subsection{Acknowledgements}
This article contains the results of my PhD thesis. I thank my advisor J\'er\'emy Blanc for his help and support during the whole time of my PhD. I am also grateful to Fr\'ed\'eric Mangolte who was the referee of my thesis and made remarks on this text.
\section{Results}\label{Ch:results}
In this section, we state the classification of conjugacy classes of elements of prime order in the group of birational diffeomorphisms of the sphere and also the moduli spaces associated to each conjugacy class (Theorem~\ref{MainThm} and Theorem~\ref{MainThm2} below). It is required first to present some definitions and give some examples that will appear in the classification. 

We denote by $S$ the real projective algebraic surface in $\Pp^3_\R$ defined by the equation $w^2=x^2+y^2+z^2$. Let $\sigma$ denote the standard antiholomorphic involution in $\Pp^3_\C$, $\sigma\colon(w:x:y:z)\mapsto(\bar w:\bar x:\bar y:\bar z)$. Let $S(\R)$ denote the real part of $S$. Note that $S(\R)$ is contained in the affine space where $w=1$ and corresponds to the standard two-dimensional sphere of equation $x^2+y^2+z^2=1$. The following two groups are of our interest, the first one is the group of birational transformations of the sphere and is isomorphic to the real Cremona group, and the second one is the group of birational diffeomorphisms of the sphere.
\begin{equation*}
\begin{array}{lll}
\Bir(S)&:=\{f: S\dasharrow S\ |\ &f\text{\ is birational}\},\\
\\
\Aut(S(\R))&:=\{f:S\dasharrow S\ |\ &f\text{\ is birational\ and $f,f^{-1}$ are defined}\\
 & &\text{at every real} \text{\ point of\ } S\}.
\end{array}
\end{equation*}
\begin{remark}
$\Bir(S)$, $\Aut(S(\R))$ are groups and $\Aut(S(\R))\subset \Bir(S)$.
\end{remark}
Our goal is to classify the conjugacy classes of elements of $\Aut(S(\R))$ of prime order.
\begin{remark}\label{varphi}
\begin{enumerate}[(i)]
\item Forgetting the real structure given by $\sigma$, the surface $S_\C$ is isomorphic to $\Pp^1_\C\times\Pp^1_\C$. Indeed,
\[S_{\C}=\{(w:x:y:z)\in\Pp^3_\C\ |\ (w+z)(w-z)=(y+{\bf i}x)(y-{\bf i}x)\},\] and the isomorphism is given by
\begin{equation}\begin{array}{cccc}
\varphi:& S_\C & \longrightarrow & \Pp^1_\C\times\Pp^1_\C\\
& (w:x:y:z) & \longmapsto & \ \ \ ((w+z:y+{\bf i}x),(w+z:y-{\bf i}x))\\
& & & =  ((y-{\bf i}x:w-z),(y+{\bf i}x:w-z)),
\end{array}\end{equation}
whose inverse is given by
\[\begin{array}{cccc}
\varphi^{-1}:& \Pp^1_\C\times\Pp^1_\C & \longrightarrow & S_\C\\
& ((r:s)(u:v)) & \longmapsto & (ru+sv:{\bf i}(rv-su):rv+su:ru-sv)
\end{array}\]
\item $\Pic(S)=\Z$, $\Pic(S_{\C})=\Z \oplus\Z$.
\end{enumerate}
\end{remark}
We denote by $\pi$ the projection $\pi\colon S\dasharrow \Pp^1$ given  by $\pi(w:x:y:z)=(w:z)$. Notice that every fibre of $\pi$ is rational except for $\pi^{-1}(1:1)$ and $\pi^{-1}(1:-1)$, which are the union of the lines $w=z$, $x=\pm{\bf i}y$, and $w=-z$, $x=\pm{\bf i}y$, respectively. 

Let us fix some notation for groups associated to the pair $(S,\pi)$,
\begin{align*}
\Bir(S,\pi):=&\{g \in \Bir(S)\ |\ \exists\alpha\in\Aut(\Pp^1) \text{\ such that } \alpha\pi=\pi g\},\\
\Aut(S(\R),\pi):=&\{g \in \Aut(S(\R))\ |\ \exists\alpha\in\Aut(\Pp^1) \text{\ such that } \alpha\pi=\pi g\}.
\end{align*}
Note that $\Aut(S(\R),\pi)\subset\Bir(S,\pi)$, more precisely $\Aut(S(\R),\pi)=\Bir(S,\pi)\cap\Aut(S(\R))$. The group $\Aut(S(\R),\pi)$ is the group of birational diffeomorphisms that preserve the fibration. 

There is a natural map $\Phi$ sending any $g\in\Bir(S,\pi)$ to the associated action on the basis $\Phi(g)=\alpha\in\Aut(\Pp^1)$ so that the following diagram commutes:
\[\xymatrix@R-0.8pc{
S\  \ar[d]_\pi \ar@{-->}[r]^{g} & \ S \ar[d]_\pi  \\
\Pp^1\ \ar[r]^\alpha_\simeq & \ \Pp^1
}\]
Hence we get the exact sequence:
\begin{equation}\label{seqCB}
1\rightarrow\Bir(S/\pi)\rightarrow\Bir(S,\pi)\xrightarrow\Phi\Aut(\Pp^1),
\end{equation}
where we have denoted by $\Bir(S/\pi)$ the group:
\[\Bir(S/\pi):=\{g\in\Bir(S,\pi)\ |\ \pi=\pi g\}.\]

One can see the group of birational diffeomorphisms that acts trivially on the basis of the fibration as a subgroup of $\Bir(S/\pi)$, more precisely,
\[\Aut(S(\R)/\pi)=\{g \in \Aut(S(\R),\pi)\ |\ \pi=\pi g\}.\]
This latter subgroup has a special description given by the exact sequence
\[1\rightarrow \Aut^+(S(\R)/\pi)\rightarrow\Aut(S(\R)/\pi)\xrightarrow{o}\Z/2\Z\rightarrow1\]
where $\Aut^+(S(\R)/\pi)$ denotes the orientation preserving birational diffeomorphisms of $S$ and the map $\Aut(S(\R)/\pi)\xrightarrow{o}\Z/2\Z$ admits a section $s\colon\Z/2\Z\to\Aut(S(\R)/\pi)$ mapping $-1$ into $\tau$ where $\tau$ is a reflection, say, $\tau\colon S\to S$, $(x,y,z)\mapsto(x,-y,z)$ in the chart $w=1$. Then
\begin{equation}\label{diffS}
\Aut(S(\R)/\pi)\cong\Aut^+(S(\R)/\pi)\rtimes\langle\tau\rangle.
\end{equation}

Before stating the main results, let us describe some examples.

\begin{example}\textit{Geiser involution of the sphere}\\
The blow-up $\zeta \colon X\to S$ of three pairs of conjugate imaginary points in $S(\C)$ is a real Del Pezzo surface $X$ of degree $2$, with $X(\R)$ isomorphic to $S(\R)$. The linear system of the anticanonical class of $X$ yields double covering of $\Pp^2$ ramified over a smooth real quartic with one oval. The Geiser involution $\nu$ on $X$ is the involution which exchanges the two points of any fibre. The birational map $\zeta\nu\zeta^{-1}$ on $S$ is a birational diffeomorphism of $S$ of order $2$ that fixes pointwise a non-hyperelliptic curve of genus $3$ with one oval. The birational diffeomorphism obtained will be called \emph{Geiser involution of the sphere}.
\end{example}
\begin{example}\label{alpha1alpha2}
The blow-up $\varepsilon \colon X\to S$ of two pairs of conjugate imaginary points in $S(\C)$ is a real Del Pezzo surface $X$ of degree $4$ (see Subsection~\ref{dP4}), with $X(\R)$ isomorphic to $S(\R)$. In this case, the anticanonical divisor of $X$ is very ample and then the linear system of $|-K_X|$ gives an embedding into $\Pp^4$ as an intersection of two quadrics. In the coordinates $(y_1:y_2:y_3:y_4:y_5)$ of $\Pp^4$, $X$ is given by the intersection of
\begin{align*}
Q_1\colon&(\mu-\mu\overline{\mu}+\overline{\mu})y_1^2-2y_1y_2+y_2^2+(1-\overline{\mu}+\mu\overline{\mu}-\mu)y_3^2+y_4^2=0,\\
Q_2\colon&\mu\overline{\mu}y_1^2-2\mu\overline{\mu}y_1y_2+(\mu-1+\overline{\mu})y_2^2+\mu\overline{\mu}y_4^2+(1-\overline{\mu}+\mu\overline{\mu}-\mu)y_5^2=0,
\end{align*} for some $\mu\in\C\setminus\{0,\pm1\}$ (see Proposition~\ref{kerEq} in Subsection~\ref{dP4}).

The automorphisms $\alpha_1$, $\alpha_2$ on $X$ defined by
\begin{align*}
\alpha_1\colon&(y_1:y_2:y_3:y_4:y_5)\mapsto (y_1:y_2:y_3:y_4:-y_5),\\
\alpha_2\colon&(y_1:y_2:y_3:y_4:y_5)\mapsto (y_1:y_2:-y_3:y_4:y_5)
\end{align*}
yield the birational diffeomorphisms $\varepsilon\alpha_1\varepsilon^{-1}$, $\varepsilon\alpha_2\varepsilon^{-1}$ on $S$ of order 2 that by abuse of notation we denote again $\alpha_1$ and $\alpha_2$. Each fixes pointwise an elliptic curve.
 \end{example}
\begin{example}\label{rotation1}
Let $\theta\in [0,2\pi)$. The rotation $r_\theta\in\Aut(S)$ is given by
\[r_\theta\colon(w:x:y:z)\mapsto(w:x\cos\theta-y\sin\theta:x\sin\theta+y\cos\theta:z).\]
This is a rotation that fixes the $z$-axis and preserves the fibration $\pi$.
\end{example}
\begin{example}\label{reflection1}
The reflection $\upsilon$ is given by the map
\[\upsilon\colon(w:x:y:z)\mapsto(w:-x:y:z).\]
This is a reflection that preserves the fibration $\pi$ and fixes a conic.
\end{example}
\begin{example}\label{antipodal1}
The antipodal involution of the sphere $\tilde a$ is given by \[\tilde a\colon(w:x:y:z)\mapsto(-w:x:y:z).\] This involution has no real fixed points. 
\end{example}

With these examples, we are ready to present the main two theorems of this text. The first one tell us that there are eight families of conjugacy classes (some with only one element, some with infinitely many) and the second, the moduli space associated to each family. These two results are proved in Section~\ref{Ch:connection} using all results obtained in Sections~\ref{DPcase} - \ref{Ch:connection}.

\begin{theorem}\label{MainThm}
Every element of prime order of $\Aut(S(\R))$ is conjugate to an element of one of the following families:
\begin{enumerate}[$(1)$]
\item A Geiser involution.
\item An involution $\alpha_1$ or $\alpha_2$ given in Example~$\ref{alpha1alpha2}$. 
\item A rotation $r_\theta$ of prime order given in Example~$\ref{rotation1}$.
\item The reflection $\upsilon$ given in Example~$\ref{reflection1}$.
\item The antipodal involution $\tilde a$ given in Example~$\ref{antipodal1}$.
\item An involution in $\Aut^+(S(\R)/\pi)$ acting on the fibres of~$\pi$ by maps conjugate to rotations of order $2$, and whose set of fixed points on $S(\C)$ is a hyperelliptic curve of genus $\geq 1$ with no real points, plus the two isolated points north and south poles, $P_N$ and $P_S$.
\item An involution in $\Aut(S(\R)/\pi)\setminus\Aut^+(S(\R)/\pi)$, acting on the fibres of $\pi$ by maps conjugate to reflections, and whose set of fixed points on $S(\C)$ is a hyperelliptic curve of genus $\geq 1$ whose set of real points consists of one oval, passing through $P_N$ and $P_S$.
\item An involution in $\Aut(S(\R),\pi)\setminus \Aut(S(\R)/\pi)$ acting by $z\to -z$ on the basis which is not conjugate to $(w:x:y:z)\mapsto(w:\pm x:\pm y:-z)$.
\end{enumerate}
\end{theorem}
\begin{theorem}\label{MainThm2}
The eight families presented in Theorem \ref{MainThm} correspond to distinct sets of conjugacy classes, parametrised respectively by\begin{enumerate}[$(1)$]
\item Isomorphism classes of smooth non-hyperelliptic real projective curves of genus $3$ with one oval.
\item Isomorphism classes of pairs $(X,g)$, where $X$ is a Del Pezzo surface of degree $4$ with $X(\R)\simeq S(\R)$ and $g$ is an automorphism of order $2$ that does not preserve any real conic bundle.
\item Angles of rotations, up to sign.
\item One point (only one conjugacy class).
\item One point (only one conjugacy class).
\item Smooth real projective hyperelliptic curves $\Gamma$ of genus $\geq 1$ with no real point, together with a $2\colon 1$-covering  $\Gamma\to \Pp^1$, up to isomorphisms compatible with the fibration and the interval $[-1,1]$.
\item Smooth real projective hyperelliptic curves $\Gamma$ of genus $\geq 1$ with one oval, together with a morphism \/$\Gamma\to \Pp^1$, which is a $2\colon 1$-cover and satisfies $\pi(\Gamma(\R))=[-1,1]$, up to isomorphisms  compatible with the fibration and the interval.
\item An uncountable set, which has a natural surjection to $\bigoplus\limits_{b\in\R_{>0}}\Z/2\Z$.
\end{enumerate}
\end{theorem} 
\begin{remark}
In (7), we can have genus $0$ but this corresponds to the reflection $\upsilon$. In $(6)$ we can also have genus $0$, there is in fact a real one-dimensional family of such maps, all conjugate to the family $(8)$ (see Lemma~\ref{Lem:SpecialMapsg1g2}).
\end{remark}
\begin{remark}
All elements in $(8)$ are conjugate in $\Bir(S_\C)$, this shows a big difference between the complex and real cases.
\end{remark}

\section{Surface automorphisms and pairs}\label{Ch:pairs}

In this section, it is shown that to classify conjugacy classes of a birational diffeomorphism of finite order of the sphere is equivalent to classify birational pairs $(X,g)$ where $g$ is an automorphisms of finite order of a smooth real projective surface $X$ obtained from the sphere after blowing up pairs of conjugate imaginary points. Moreover, Proposition~\ref{MinMod} gives what pairs $(X,g)$ need to be studied.

We start with some definitions and a classical result due to Comessatti (Theorem~\ref{t}), which states in particular that the sphere $S$ is a minimal real surface. 
\begin{definition}
Let $X$ be a smooth real projective surface. We say that $X$ is \emph{minimal} if any birational morphism $X\rightarrow Y$ with $Y$ a smooth real projective surface is an isomorphism.
\end{definition}

\begin{remark}
Any birational morphism between smooth projective algebraic surfaces is a sequence of contractions of
\begin{enumerate}[(i)]
\item one real $(-1)$-curve, or
\item two disjoint conjugate imaginary $(-1)$-curves.
\end{enumerate}
\end{remark}
Therefore, a surface is minimal if and only if it does not contain a real $(-1)$-curve or two disjoint conjugate imaginary $(-1)$-curves.
Let us cite the following classical result due to Comessatti \cite{Com12}:
\begin{thm}\label{t}
If $X$ is a minimal rational smooth real surface such that $X(\R)\neq\emptyset$, then $X$ is isomorphic to $\Pp_\R^2$, to $S$, or to a real Hirzebruch surface $\F_n$ with $n\neq1$. Moreover, $X(\R)$ is connected and homeomorphic to the real projective plane, the sphere, the torus ($n$ even), or the Klein bottle ($n$ odd) respectively.
\end{thm}

\begin{prop}\label{p}
Let $X$ be a smooth real projective surface with $X(\R)$ diffeomorphic to the sphere. Then $X$ does not contain any real $(-1)$-curve. In particular, any birational morphism $\zeta\colon X\rightarrow Y$, where $Y$ is a smooth real projective surface, restricts to a diffeomorphism $\zeta\colon X(\R)\rightarrow Y(\R)$.
\end{prop}
\begin{proof}
If $X$ contains a real $(-1)$-curve, then there is a birational morphism which corresponds to the blow-up of a real point of some smooth real projective surface whose preimage by such a birational morphism is the real $(-1)$-curve. Then the neighbourhood of the real locus of the $(-1)$-curve in $X(\R)$ is topologically a Möbius strip which implies that $X(\R)$ is not orientable and therefore non isomorphic to the sphere. 
\end{proof}
\begin{definition}
Let $(X,g)$ be a pair i.e. $X$ is a smooth real projective surface and $g$ is a non-trivial automorphism of $X$ of finite order. The pair $(X,g)$ is said to be \emph{minimal} if any birational morphism $\zeta\colon X\to X'$ such that there exist an automorphism $g'$ of $X'$ of finite order with $\zeta\circ g=g'\circ \zeta$ is an isomorphism.
\end{definition}
\begin{prop}\label{MinMod}
Let $g\in\Aut(S(\R))$ be an element of finite order and let $\pi\colon S\dasharrow \Pp^1$ be the map given  by $\pi(w:x:y:z)=(w:z)$. Replacing $g$ with a conjugate in the group $\Aut(S(\R))$, one of the following holds:
\begin{enumerate}[$(a)$]
\item There exists a birational morphism $\varepsilon\colon X\to S$ which is the blow-up of\/ $0$, $1$, $2$, or $3$ pairs of conjugate imaginary points in $S$, such that $\hat g=\varepsilon^{-1}\circ g\circ \varepsilon\in\Aut(X)$, $\Pic(X)^{\hat g}\cong \Z$, and $X$ is a Del Pezzo surface.
\item There exists $\alpha\in \Aut(\Pp^1)$ such that $\alpha\pi=\pi g$. Moreover, there exists a birational morphism $\varepsilon\colon X\to S$ that restricts to a diffeomorphism $X(\R)\to S(\R)$ such that $\hat g=\varepsilon^{-1}\circ g\circ\varepsilon\in\Aut(X)$, $\pi\circ\varepsilon\colon X\to\Pp^1$ is a conic bundle on $X$, and $\Pic(X)^{\hat g}\cong\Z^2$.
\end{enumerate}
\end{prop}
\begin{proof}
Let $g\in\Aut(S(\R))$ of finite order, $g:S\dasharrow S$ is a birational map with a finite number of imaginary base points, say $p_1,\overline{p_1},\ldots,p_n,\overline{p_n}$ that belong to $S$ as proper or infinitely near points. After blowing up all of them and their images under powers of $g$ (meaning the orbit of the points by $g$), we obtain a smooth projective surface $\tilde X$
\[\xymatrix@R-0.8pc{
\tilde X\  \ar[d]_\zeta \ar[r]^{\tilde{g}=\zeta^{-1}g\zeta} & \ \tilde X \ar[d]_\zeta  \\
S\ \ar@{-->}[r]^g & \ S
}\]
where $\tilde{g}$ is an automorphism of $\tilde X$.

Since $g$ is defined at every real point of $S$, the birational morphism $\zeta$ restricts to a diffeomorphism $\tilde X(\R)\rightarrow S(\R)$. After contracting all sets of disjoint $(-1)$-curves which are invariant by $\tilde{g}$ and defined over $\R$, we get a minimal pair $(X,\hat{g})$, with $X(\R)$ diffeomorphic to the sphere by the Proposition \ref{p}, which can be one of the two following possibilities (see \cite[Theorem 1G]{Isk80}):

\begin{enumerate}[(i).]
\item $\Pic(X)^{\hat{g}}$ has rank 1 and $X$ is a Del Pezzo surface.
\item $\Pic(X)^{\hat{g}}$ has rank 2, there is a morphism $X\xrightarrow{\pi_X}\Pp^1$, $X$ is a conic bundle.
\end{enumerate}

Recall that $\Pic(X)^{\hat{g}}$ is the part of $\Pic(X)$ which is invariant by $\hat{g}\in\Aut(X)$.

In the first case, there exists $\varepsilon\colon X\to Z$ a birational morphism to a minimal projective smooth real algebraic surface $Z$. By Proposition \ref{p}, $Z(\R)$ is diffeomorphic to the sphere and by Theorem \ref{t}, we have $Z\simeq S$. Then $(K_X)^2>0$, $K_X=\varepsilon^*(K_S)+E_1+\overline{E_1}+\cdots+E_r+\overline{E_r}$ $\Rightarrow$ $(K_X)^2=K_S^2-2r$ and consequently $X$ is the blow-up of 0, 2, 4 or 6 points in $S$ and $X$ is a Del Pezzo surface of degree 8, 6, 4 or 2 and this gives statement $(a)$. We study this case in detail in Section \ref{DPcase}.

For the second case, we denote by $(X,\pi_X,\hat{g})$ the minimal real conic bundle with rank $\Pic(X)^{\hat{g}}=2$. Recall that $X(\R)\simeq S(\R)$ implies that there is no real $(-1)$-curve on $X$. Forgetting the action of $\hat g$ on $X$, there is a birational morphism $X\to Z$ which is the contraction of disjoint imaginary $(-1)$-curves in fibres.  In this way, we obtain $\pi_Z\colon Z\rightarrow \Pp^1$ a minimal conic bundle with exactly two singular fibres because $Z(\R)$ is diffeomorphic to $S(\R)$ again by Proposition \ref{p}. Now, if we dismiss $\pi$ and keep contracting, we end up with $\tilde{Z}$ a minimal real surface such that $\tilde{Z}(\R)\simeq Z(\R)$ and by Theorem \ref{t} we have $\tilde{Z}\simeq S$ implying that $Z$ is the blow-up of two imaginary points on $S$. In this case, the surface $Z$ is unique and is the Del Pezzo surface of degree 6 that will be described in Subsection \ref{Pezzodeg6}. The explicit conic bundle structure on $Z$ corresponds to the lift of the projection $\pi\colon S\dasharrow\Pp^1$ sending $(w:x:y:z)$ to $(w:z)$. More precisely, $\pi_Z=\pi\circ\varepsilon$ where $\varepsilon\colon Z\to S$ is the blow-up of two imaginary conjugate points.
\end{proof}

\section{Del Pezzo surfaces with $\rk(\Pic(X)^{\hat{g}})=1$}\label{DPcase}
In this section, we study the pairs $(X,g)$ where $X$ is a Del Pezzo surface and $g$ is an automorphism of $X$. This corresponds to the first case in Proposition~\ref{MinMod}. 

Recall that the complex surface $S_\C$ is isomorphic to $\Pp_\C^1\times\Pp_\C^1$ via the isomorphism $\varphi\colon S_\C\to\Pp^1_\C\times\Pp^1_\C$ (see Remark~\ref{varphi}). 

We denote by $f$ and $\overline{f}$ the divisors of the fibres of the two projections i.e. $\Pic(S_\C)=\Z f\oplus\Z\bar f$ and by abuse of notation we denote again by $f$ and $\overline{f}$ the pullback $\varepsilon^{*}(f)$ and $\varepsilon^{*}(\overline{f})$ in $X$ for $\varepsilon:X\rightarrow S$ a birational morphism.

\subsection{Case: $(K_X)^2=8$.}\label{dP8}
In this subsection, our interest is to present the group of real automorphisms of $S$, $\Aut(S)$, and describe the conjugacy classes of it. We call $\sigma$ the corresponding antiholomorphic involution in $\Pp_\C^1\times\Pp_\C^1$ via the isomorphism $\varphi$, which is given by $\sigma(x,y)=(\overline{y},\overline{x})$.  

\begin{prop}\label{Aut(S)}
The group $\Aut(S)$ corresponds, via $\varphi$, to the subgroup of the group of complex automorphisms $\Aut(\Pp^1_\C\times\Pp^1_\C)$ generated by $\upsilon\colon (x,y)\mapsto(y,x)$ and by $\f=\{(A,\overline{A})\ |\ A\in\PGL(2,\C)\}$. Moreover, $\Aut(S)\cong\f\rtimes\langle\upsilon\rangle.$
\end{prop}
\begin{proof}
Using the $\C$-isomorphism $S_\C\simeq\Pp^1_\C\times\Pp^1_\C$, the group $\Aut(S)$ is the subgroup of $\Aut(\Pp^1_\C\times\Pp^1_\C)$ consists of elements that commute with $\sigma$, i.e. $\Aut(S)=\Aut(\Pp^1_\C\times\Pp^1_\C,\sigma)$. Let $(A,B)\in \PGL(2,\C)\times\PGL(2,\C)$, $(A,B)$ commutes with $\sigma$ if and only if $(A,B)\sigma(x,y)=\sigma(A,B)(x,y)=\sigma(Ax,By)$ and hence $(A\overline{y},B\overline{x})=\left(\overline{B}\overline{y},\overline{A}\overline{x}\right)$ and it is equivalent to $A=\overline{B}$. If we call $\upsilon:(x,y)\mapsto(y,x)$, which corresponds to $(w:x:y:z)\mapsto(w:-x:y:z)$ on $\Pp^3$, we see that $\upsilon\sigma=\sigma\upsilon$, then $\Aut(S)=\Aut(\Pp_\C^1\times\Pp_\C^1,\sigma)=\f\rtimes\langle\upsilon\rangle$.
\end{proof}
Automorphisms in $\f$ fix the divisors of fibres $f$ and $\overline{f}$ while elements of $\Aut(S)\setminus \f$ are thus of the form $(x,y)\mapsto(\overline{A}y,Ax)$ for $A\in\PGL(2,\C)$ i.e. automorphisms exchanging the divisors of the fibres $f$ and $\overline{f}$.
\begin{example}\label{rot-ref-antip}
The following automorphisms, already described in the introduction, are now presented as automorphisms of $\Pp^1_\C\times\Pp^1_\C$ via the isomorphism $\varphi$:
\begin{enumerate}
\item The rotation $r_\theta$ given in Example \ref{rotation1} belongs to $\Aut(S)$ and corresponds to the automorphism
$(x,y)\mapsto(xe^{-{\bf i}\theta},ye^{{\bf i}\theta})$ of $\Pp^1_\C\times\Pp_\C^1$. 
\item The reflection $\upsilon$ given in Example \ref{reflection1} belongs to $\Aut(S)$ and corresponds to the automorphism $\upsilon\colon (x,y)\mapsto(y,x)$ of $\Pp_\C^1\times\Pp_\C^1$. 
\item The antipodal automorphism of the sphere given in Example \ref{antipodal1} corresponds to the automorphism 
$\tilde a\colon(x,y)\mapsto\left(-\frac{1}{y},-\frac{1}{x}\right)$ of $\Pp^1_\C\times\Pp^1_\C.$
\end{enumerate}
\end{example}
\begin{prop}\label{ConjAut(S)}
Every element of $\Aut(S)$ of prime order is conjugate to a rotation $r_\theta$ , or to the reflection $\upsilon$, or to the antipodal involution $\tilde a$, which are given in Example~$\ref{rot-ref-antip}$.
\end{prop}
\begin{proof}
We work in $\Aut(\Pp^1_\C\times\Pp^1_\C)$ according to Proposition \ref{Aut(S)}. If $g\in\f$ then $g\colon(x,y)\mapsto(Ax,\bar Ay)$ for some $A\in\PGL(2,\C)$ of finite order. Hence, $A$ is conjugate to $\left[\begin{smallmatrix}
1 & \\
 & e^{-{\bf i}\theta}
\end{smallmatrix}\right]$ for some angle $\theta$ and locally we write $x\mapsto e^{-{\bf i}\theta}x$.  This shows that $g$ is conjugate in $\f$ to $(x,y)\mapsto(xe^{-{\bf i}\theta},ye^{{\bf i}\theta})$.

If $g\notin\f$, then $g\colon(x,y)\mapsto (Ay,\bar Ax)$ for some $A\in\PGL(2,\C)$. Since $g$ has prime order, $g^2$ is the identity so $A\bar A=1$ in $\PGL(2,\C)$. Notice that the action of $\upsilon$ on $\PGL(2,\C)$ is given by the action of $\upsilon$ on $\f$ in the first component, i.e. $\upsilon(A)=\bar A$ and the condition $A\bar A=1$ is equivalent to $A\upsilon(A)=1$. 

Let $A_0\in\GL(2,\C)$ be a representative of the element $A$, then $A_0\overline{A_0}=\left[\begin{smallmatrix}
\lambda & 0\\
0 & \lambda
\end{smallmatrix}\right]$ for some $\lambda\in\C^*$. Since $A_0$ commutes with $A_0\overline{A_0}$, $A_0$ commutes with $\overline{A_0}$. This implies that $\lambda\in\R$. Then we multiply $A_0$ with $\mu\in\C$ and assume that $\lambda=1$ or $\lambda=-1$. In the first case, there exists $B$ such that $B^{-1}A_0\overline B=\left[\begin{smallmatrix}
1 & 0\\
0 & 1
\end{smallmatrix}\right]$ because $H^1(\langle\upsilon\rangle,\GL(2,\C))$ is trivial by Proposition 3 in \cite[Chapter X]{Serre}. This implies that $g$ is conjugate to $\upsilon$ by $(x,y)\mapsto(Bx,\overline By)$. In the second case, we want to find $B\in\GL(2,\C)$ such that $B^{-1}A\overline B=\left[\begin{smallmatrix}
0 & -1\\
1 & 0
\end{smallmatrix}\right]$. This will imply that $g$ is conjugate to the antipodal involution $\tilde a$ in Example~\ref{rot-ref-antip} by the automorphism $(x,y)\mapsto(Bx,\overline By)$ as before.

Let $e_1=\left[\begin{smallmatrix} 1 \\ 0\end{smallmatrix}\right], e_2=\left[\begin{smallmatrix} 0 \\ 1\end{smallmatrix}\right]$ be the two standard vectors, and choose a vector $v_1\in \C^2$ such that $(v_1,A_0\overline{v_1})$ is a basis of $\C^2$. This is always possible, by taking $v_1\in \{e_1,e_2\}$. Indeed, otherwise $A_0$ would be diagonal, so $A_0\cdot \overline{A_0}$ would have positive coefficients. We choose then $B\in \GL(2,\C)$ such that $Be_1=v_1$, $Be_2=A_0\overline{v_1}$, and observe that
\[\begin{array}{rccccc}
-Be_1&=&-v_1=A_0\overline{A_0}v_1&=&A_0 \overline{B}e_2,\\
Be_2&=&A_0\overline{v_1}&=&A_0\overline{B}e_1.\end{array}\]
Multiplying by $B^{-1}$, we obtain $B^{-1}A_0\overline{B}(e_1)=e_2$ and $B^{-1}A_0\overline{B}(e_2)=-e_1,$ which corresponds to
\[B^{-1}A_0\overline{B}=\left[\begin{smallmatrix} 0 & -1 \\1 & 0\end{smallmatrix}\right].\qedhere\]
\end{proof}
\begin{remark}
The group $\f$ corresponds to the orientation-preserving automorphisms of $S$ denoted by $\Aut^+(S)$.
\end{remark}

In the sequel, we will also need the following result.
\begin{lemma}\label{Aut(dP6)}
Let $p=(0:{\bf i}:1:0)\in S$. The group of automorphisms of $S$ preserving the set $\{p,\bar p\}$ is denoted by $\Aut(S,\{p,\bar p\})$ and, via the isomorphism $\varphi$, has the following structure
\[\Aut(S,\{p,\bar p\})\cong\mathscr D\rtimes\langle\upsilon,\tilde\upsilon\rangle\]
where $\mathscr D$ is the subgroup of \ $\f$ of diagonal elements, the isomorphism $\tilde\upsilon$ is defined by $(x,y)~\mapsto~\left(\frac{1}{x},\frac{1}{y}\right)$, and $\langle\upsilon,\tilde\upsilon\rangle\cong(\Z/2\Z)^2$. Moreover, every element of prime order is one of the following:
\begin{enumerate}[$($a$)$]
\item a rotation $r_\theta$, given in Example $\ref{rot-ref-antip}$, corresponding to one element of \ $\mathscr D$,
\item conjugate to $\tilde\upsilon$,
\item conjugate to $\upsilon$,
\item equal to $\upsilon\tilde\upsilon$,
\item equal to the map $\tilde a\colon(x,y)\mapsto\left(-\frac{1}{y},-\frac{1}{x}\right)$, which corresponds on the sphere to the antipodal automorphism. 
\end{enumerate}
\begin{table}[!hbt]
\begin{center}
\begin{tabular}{|c||c|c|}
\hline
 & $\Pp^1\times \Pp^1$ & $S_\C$\\
\hline\hline
$\begin{array}{c}
\upsilon\phantom{\normalsize (}\\
\tilde\upsilon\phantom{\big (}\\
\upsilon\tilde\upsilon\phantom{\big (}\\
\tilde a\phantom{\big (}
\end{array}$ &
$\begin{array}{rl}
(x,y)\mapsto & (y,x)\\
(x,y)\mapsto & \left(\frac{1}{x},\frac{1}{y}\right)\\
(x,y)\mapsto & \left(\frac{1}{y},\frac{1}{x}\right)\\
(x,y)\mapsto & \left(-\frac{1}{y},-\frac{1}{x}\right)
\end{array}$ &
$\begin{array}{rl}
(w:x:y:z)\phantom{\normalsize (}\mapsto & (w:-x:y:z)\\
(w:x:y:z)\phantom{\big (}\mapsto & (w:-x:y:-z)\\
(w:x:y:z)\phantom{\big (}\mapsto & (w:x:y:-z)\\
(w:x:y:z)\phantom{\big (}\mapsto & (-w:x:y:z)
\end{array}$\\
\hline
\end{tabular}
\caption{List of automorphisms.}
\end{center}
\end{table}
\end{lemma}
\begin{proof}
The points $p$ and $\bar p$ correspond, via $\varphi$, to the points $(1:0)(0:1)$ and $(0:1)(1:0)$, respectively. Diagonal elements in $\PGL(2,\C)$ yield a subgroup of $\f$ preserving the points $p$ and $\bar p$ which is $\mathscr D$. The elements in $\f$ which interchange the two points are elements $(A,\bar A)$ in $\f$ with $A$ of the form $\left[\begin{smallmatrix}
0 & 1\\
a & 0
\end{smallmatrix}\right]\in\PGL(2,\C)$. Then the subgroup of $\f$ which preserve the set $\{p,\bar p\}$ has the structure $\mathscr D\rtimes\langle\tilde\upsilon\rangle$ with $\tilde\upsilon$ the automorphism of $\f$ defined by the element $\left[\begin{smallmatrix}
0 & 1\\
1 & 0
\end{smallmatrix}\right]$ and that locally is described in the statement. As $\tilde\upsilon$ commutes with $\upsilon$ that permutes the points, we get $\Aut(S,\{p,\bar p\})\cong\mathscr D\rtimes\langle\upsilon,\tilde\upsilon\rangle$.

\begin{enumerate}[$($a$)$]
\item An element of finite order in $\mathscr D$ is a rotation $r_\theta$ given in Example \ref{rot-ref-antip}. 
\item If $g\in\mathscr D\rtimes\langle\tilde\upsilon\rangle\subset \Aut(S,\{p,\bar p\})$ and is not a rotation, then $g\colon(x,y)\mapsto (Ax,\bar Ay)$ with $A=\left[\begin{smallmatrix}
0 & 1\\
b & 0
\end{smallmatrix}\right]$ for some $b\in\C$. Since $A$ is conjugate to $\left[\begin{smallmatrix}
0 & 1\\
1 & 0
\end{smallmatrix}\right]$ by the diagonal element $\left[\begin{smallmatrix}
1 & 0\\
0 & 1/\sqrt{b}
\end{smallmatrix}\right]$, then $g$ is conjugate to $\tilde\upsilon$ in $\Aut(S,\{p,\bar p\})$.
\item If $g\in\mathscr D\rtimes\langle\upsilon\rangle \subset\Aut(S,\{p,\bar p\})$ and is not a rotation, then $g\colon(x,y)\mapsto(Dy,\bar Dx)$ with $D=\left[\begin{smallmatrix}
1 & 0\\
0 & b
\end{smallmatrix}\right]$ for some $b\in\C$. Then $A\bar A=1$ because $g$ is of prime order and the action of $\upsilon$ on $\mathscr D$ is exactly the conjugation and the equality $A\bar A=1$ is the same as $A\upsilon(A)$=1. Then $g$ is conjugate to $\upsilon$ because the group $\mathcal D=\{D\in\PGL(2,\C)\ |\ D\text{\ is diagonal}\}$ is isomorphic to $\C^*$ and $H^1(\langle\upsilon\rangle,\mathcal D)=\{1\}$ by Hilbert's Theorem 90. 
\item[(d,e)] If $g\in\mathscr D\rtimes \langle\upsilon\tilde\upsilon\rangle$ and is not a rotation, then $g=(d,\upsilon\tilde\upsilon)$  for $d\in\mathscr D$ of finite order and in this case, $d$ commutes with $\upsilon\tilde\upsilon$ implying that $d$ has order 1 or 2 since the order of $g$ is prime. Then $g$ is either $\upsilon\tilde\upsilon$ and is given by the map $(x,y)\mapsto(1/y,1/x)$ on $\Pp^1\times\Pp^1$, which is the map $(w:x:y:z)\mapsto(w:x:y:-z)$ on $S$ or is given by the map $(x,y)\mapsto(-1/y,-1/x)$ on $\Pp^1\times\Pp^1$ and corresponds, on the sphere, to the antipodal automorphism $(w:x:y:z)\mapsto(-w:x:y:z)$.\qedhere\end{enumerate}
\end{proof}

\subsection{Case: $(K_X)^2=6$.}\label{Pezzodeg6}

\begin{prop}\label{dP6nominimal}
Let $\zeta:X\rightarrow S$ be the blow-up of two imaginary conjugate points $p,\overline{p}$. Then $\zeta\Aut(X)\zeta^{-1}\subset\Aut(S)$, so the pair $(X,\Aut(X))$ is not minimal.
\end{prop}
\noindent\begin{minipage}{0.52\textwidth}
\begin{proof}
On $X$, there are six $(-1)$-curves: the two exceptional divisors $E_p$ and $E_{\overline{p}}$ and the four curves corresponding to the strict transforms of the fibres $f$ and $\overline{f}$ passing through one point denoted by $f_p$, $f_{\overline{p}}$, $\overline{f_p}$, and $\overline{f_{\overline{p}}}$.

Since $\overline{f_{\overline{p}}\cap {\overline{f_{\overline{p}}}}}=f_{\overline{p}}\cap{\overline{f_{\overline{p}}}}$ and $\overline{f_p\cap \overline{f_p}}=f_p\cap \overline{f_p}$, these two intersection points are real (see the circles $\circ$ in Figure~\ref{picdP6}) and the other four vertices of the hexagon are imaginary, so any action of $Y$ can only exchange the two lines $E_p$ and $E_{\overline{p}}$ and this implies that $(X,\Aut(X))$ is not minimal.
\end{proof}
\end{minipage}
\begin{minipage}{0.48\textwidth}
\begin{center}
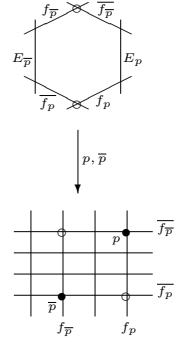

\begin{picture}(80,135)
\scalebox{0.8}{
\put(55,108){\line(-2,1){30}}
\put(32,109){\unboldmath{\tiny $\overline{f_p}$}}
\put(45,108){\line(2,1){30}}
\put(47,107.5){\unboldmath$\circ$}
\put(58,110){\unboldmath{\tiny $f_p$}}

\put(30,115){\line(0,1){35}}
\put(71,130){\unboldmath{\tiny $E_p$}}
\put(19,130){\unboldmath{\tiny $E_{\overline{p}}$}}
\put(70,115){\line(0,1){35}}

\put(55,158){\line(-2,-1){30}}

\put(33,153){\unboldmath{\tiny $f_{\overline{p}}$}}
\put(45,158){\line(2,-1){30}}
\put(47,152.5){\unboldmath$\circ$}
\put(59,153){\unboldmath{\tiny $\overline{f_{\overline{p}}}$}}

\put(50,98){\vector(0,-1){30}}
\put(52,83){\unboldmath{\tiny $p,\overline{p}$}}

\put(20,50){\line(1,0){65}}
\put(70,47){\unboldmath$\bullet$}
\put(40,47){\unboldmath$\circ$}
\put(66,45){\unboldmath{\tiny $p$}}
\put(87,50){\unboldmath{\tiny $\overline{f_{\overline{p}}}$}}
\put(70,3){\unboldmath{\tiny $f_p$}}
\put(20,40){\line(1,0){65}}
\put(20,30){\line(1,0){65}}
\put(20,20){\line(1,0){65}}
\put(40,17){\unboldmath$\bullet$}
\put(70,17){\unboldmath$\circ$}
\put(36,13){\unboldmath{\tiny $\overline{p}$}}
\put(87,20){\unboldmath{\tiny $\overline{f_p}$}}
\put(40,3){\unboldmath{\tiny $f_{\overline{p}}$}}
\put(28,60){\line(0,-1){50}}
\put(43,60){\line(0,-1){50}}
\put(58,60){\line(0,-1){50}}
\put(73,60){\line(0,-1){50}}}
\end{picture}
\captionof{figure}{Blow-up of $p,\bar p$}
\label{picdP6}
\end{center}
\end{minipage}

\subsection{Case: $(K_X)^2=4$.}\label{dP4}
There is $\zeta:X\rightarrow S$ the blow-up of four imaginary points $p,\overline{p},q,\overline{q}$.
We have $16$ $(-1)$-curves in $X$: the exceptional divisors $E_p$, $E_{\overline{p}}$, $E_q$, and $E_{\overline{q}}$; the strict transform of the fibres $f$ and $\overline{f}$ passing through one point that we denote by $f_p$, $f_{\overline{p}}$, $f_q$, $f_{\overline{q}}$, $\overline{f_p}$, $\overline{f_{\overline{p}}}$, $\overline{f_q}$, and $\overline{f_{\overline{q}}}$ as in the previous subsection; and the strict transform of the curves equivalent to $f+\overline{f}$ (e.g. of bidegree $(1,1)$) passing through three of the four points that we denote by $f_{p\overline{p}q}$, $f_{p\overline{p}\overline{q}}$, $f_{pq\overline{q}}$, and $f_{\overline{p}q\overline{q}}$. 

These $(-1)$-curves form the singular fibres of ten conic bundle structures on $X$ with four singular complex fibres each and are the following:
\begin{enumerate}
\begin{multicols}{2}
\item $f+\overline{f}-E_{\overline{p}}-E_{\overline{q}}$
\item $f+\overline{f}-E_p-E_q$
\item $f+\overline{f}-E_p-E_{\overline{p}}$
\item $f+\overline{f}-E_q-E_{\overline{q}}$
\item $f+\overline{f}-E_p-E_{\overline{q}}$
\item $f+\overline{f}-E_{\overline{p}}-E_q$
\item $f$
\item $\overline{f}$
\item $2f+\overline{f}-E_p-E_{\overline{p}}-E_q-E_{\overline{q}}$
\item $f+2\overline{f}-E_p-E_{\overline{p}}-E_q-E_{\overline{q}}$
\end{multicols}
\end{enumerate}

The anticanonical divisor of $X$ is $-K_X=2f+2\overline{f}-E_p-E_{\overline{p}}-E_q-E_{\overline{q}}$. We collect these conic bundles in pairs such that the sum of every pair is $-K_X$:
\begin{align*}
P_1:=&\{f+\overline{f}-E_p-E_{\overline{p}},\ f+\overline{f}-E_q-E_{\overline{q}}\},\\
P_2:=&\{f+\overline{f}-E_p-E_q,\ f+\overline{f}-E_{\overline{p}}-E_{\overline{q}}\},\\  
P_3:=&\{f+\overline{f}-E_p-E_{\overline{q}},\ f+\overline{f}-E_{\overline{p}}-E_q\},\\
P_4:=&\{f, f+2\overline{f}-E_p-E_{\overline{p}}-E_q-E_{\overline{q}}\},\\
P_5:=&\{\overline{f},\ 2f+\overline{f}-E_p-E_{\overline{p}}-E_q-E_{\overline{q}}\}.
\end{align*}
Since $K_X$ is invariant under any automorphism of $X$, then $\Aut(X)$ acts on the set of pairs obtaining the following exact sequence.
\begin{equation}\label{seq}
\xymatrix@-1.2pc{
0 \ar[r] & \F_\R \ar[r]\ar@{}[d]|-{\textrm{\normalsize\rotatebox{-90}{$\subseteq$}}} & \Aut(X) \ar[r]^-{\rho} \ar@{}[d]|-{\textrm{\normalsize\rotatebox{-90}{$\subseteq$}}}&Sym_5\\
& \F_\C\ar[r] & \Aut(X_\C) \ar[r]^-{\rho}  &Sym_5
}
\end{equation}
where $\F_\R$ is naturally a subgroup of $\F_2^5$. An element $(a_1,\ldots,a_5)$ exchanges the two conic bundles of the pair $P_i$ if $a_i=1$ and preserves each one if $a_i=0$.
We represent in Figure~\ref{picturepairs} the picture of the five pairs of conic bundles and with the next one, how the anti-holomorphic involution $\sigma$ acts on them.
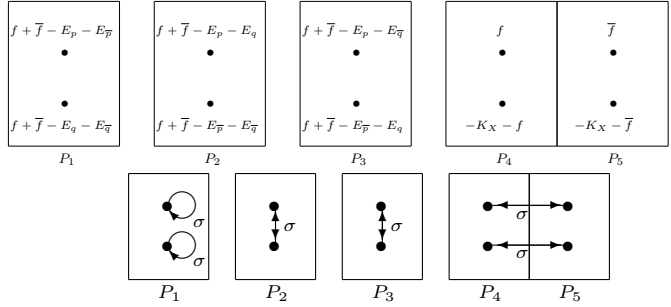
\begin{figure}[h]
\begin{center}
\begin{picture}(270,60)
\scalebox{0.65}{
\put(0,10){\line(0,1){84}}
\put(0,10){\line(1,0){64}}
\put(64,10){\line(0,1){84}}
\put(0,94){\line(1,0){64}}
\put(30,62){\unboldmath$\bullet$}
\put(1,75){\unboldmath{\scriptsize $f+\overline{f}-E_p-E_{\overline{p}}$}}
\put(30,32){\unboldmath$\bullet$}
\put(1,20){\unboldmath{\scriptsize $f+\overline{f}-E_q-E_{\overline{q}}$}}
\put(29,0){\unboldmath{\footnotesize $P_1$}}

\put(84,10){\line(0,1){84}}
\put(84,10){\line(1,0){64}}
\put(148,10){\line(0,1){84}}
\put(84,94){\line(1,0){64}}
\put(114,62){\unboldmath$\bullet$}
\put(85,75){\unboldmath{\scriptsize $f+\overline{f}-E_p-E_q$}}
\put(114,32){\unboldmath$\bullet$}
\put(85,20){\unboldmath{\scriptsize $f+\overline{f}-E_{\overline{p}}-E_{\overline{q}}$}}
\put(113,0){\unboldmath{\footnotesize $P_2$}}

\put(168,10){\line(0,1){84}}
\put(168,10){\line(1,0){64}}
\put(232,10){\line(0,1){84}}
\put(168,94){\line(1,0){64}}
\put(198,62){\unboldmath$\bullet$}
\put(169,75){\unboldmath{\scriptsize $f+\overline{f}-E_p-E_{\overline{q}}$}}
\put(198,32){\unboldmath$\bullet$}
\put(169,20){\unboldmath{\scriptsize $f+\overline{f}-E_{\overline{p}}-E_q$}}
\put(197,0){\unboldmath{\footnotesize $P_3$}}

\put(252,10){\line(0,1){84}}
\put(252,10){\line(1,0){64}}
\put(316,10){\line(0,1){84}}
\put(252,94){\line(1,0){64}}
\put(282,62){\unboldmath$\bullet$}
\put(281,75){\unboldmath{\scriptsize $f$}}
\put(282,32){\unboldmath$\bullet$}
\put(263,20){\unboldmath{\scriptsize $-K_X-f$}}
\put(281,0){\unboldmath{\footnotesize $P_4$}}

\put(316,10){\line(0,1){84}}
\put(316,10){\line(1,0){64}}
\put(380,10){\line(0,1){84}}
\put(316,94){\line(1,0){64}}
\put(346,62){\unboldmath$\bullet$}
\put(345,75){\unboldmath{\scriptsize $\overline{f}$}}
\put(346,32){\unboldmath$\bullet$}
\put(326,20){\unboldmath{\scriptsize $-K_X-\overline{f}$}}
\put(345,0){\unboldmath{\footnotesize $P_5$}}}
\end{picture}\vspace{0.1cm}
\begin{picture}(200,50)

\put(10,10){\line(0,1){40}}
\put(10,10){\line(1,0){30}}
\put(40,10){\line(0,1){40}}
\put(10,50){\line(1,0){30}}
\put(22,35){\unboldmath$\bullet$}
\put(22,20){\unboldmath$\bullet$}
\put(21,3){\unboldmath{\scriptsize $P_1$}}
\put(30,23){\circle{10}}
\put(28,33){\vector(-1,1){3}}
\put(30,38){\circle{10}}
\put(28,18){\vector(-1,1){3}}
\put(34,30){\unboldmath{\scriptsize $\sigma$}}
\put(34,15){\unboldmath{\scriptsize $\sigma$}}
\put(68,28.5){\unboldmath{\scriptsize $\sigma$}}
\put(108,28.5){\unboldmath{\scriptsize $\sigma$}}
\put(155,18){\unboldmath{\scriptsize $\sigma$}}
\put(155,33){\unboldmath{\scriptsize $\sigma$}}

\put(50,10){\line(0,1){40}}
\put(50,10){\line(1,0){30}}
\put(80,10){\line(0,1){40}}
\put(50,50){\line(1,0){30}}
\put(62,35){\unboldmath$\bullet$}
\put(62,20){\unboldmath$\bullet$}
\put(61,3){\unboldmath{\scriptsize $P_2$}}
\put(65,23){\vector(0,1){13}}
\put(65,38){\vector(0,-1){13}}

\put(90,10){\line(0,1){40}}
\put(90,10){\line(1,0){30}}
\put(120,10){\line(0,1){40}}
\put(90,50){\line(1,0){30}}
\put(102,35){\unboldmath$\bullet$}
\put(102,20){\unboldmath$\bullet$}
\put(101,3){\unboldmath{\scriptsize $P_3$}}
\put(105,23){\vector(0,1){13}}
\put(105,38){\vector(0,-1){13}}

\put(130,10){\line(0,1){40}}
\put(130,10){\line(1,0){30}}
\put(160,10){\line(0,1){40}}
\put(130,50){\line(1,0){30}}
\put(142,35){\unboldmath$\bullet$}
\put(142,20){\unboldmath$\bullet$}
\put(141,3){\unboldmath{\scriptsize $P_4$}}

\put(144,38){\vector(1,0){28}}
\put(176,38){\vector(-1,0){28}}
\put(144,23){\vector(1,0){28}}
\put(176,23){\vector(-1,0){28}}

\put(160,10){\line(0,1){40}}
\put(160,10){\line(1,0){30}}
\put(190,10){\line(0,1){40}}
\put(160,50){\line(1,0){30}}
\put(172,35){\unboldmath$\bullet$}
\put(172,20){\unboldmath$\bullet$}
\put(171,3){\unboldmath{\scriptsize $P_5$}}
\end{picture}
\end{center}
\caption{Representation of the five pairs of conic bundles and the action of $\sigma$ on them.}
\label{picturepairs}
\end{figure}
\begin{remark}\label{Imagerho}
The image of $\rho$ in the exact sequence (\ref{seq}) is contained in the group $\langle(2\ 3),(4\ 5)\rangle\subset Sym_5$ as a consequence of the action of the antiholomorphic involution $\sigma$. (See Figure~\ref{picturepairs}).
\end{remark}
\begin{lemma}\label{points}
Let $p, q\in \Pp^1_\C\times\Pp^1_\C\simeq S_\C$ be two distinct imaginary non conjugate points such that the blow-up of $p$, $\bar p$, $q$, $\bar q$ is a Del Pezzo surface. Then up to automorphisms of the sphere, the points $p$ and $q$ can be chosen to be $(1:0)(0:1)$ and $(1:1)(1:\mu)$ for some $\mu\in\C\setminus\{0,\pm1\}$, respectively.
\end{lemma}
\begin{proof}
Let $p=(r_1:s_1)(u_1:v_1)\in\Pp^1_\C\times\Pp^1_\C$. Applying the automorphism $(A,\overline{A})\in\f$ where $A=\left[
   \begin{smallmatrix}
     \overline{v_1} & -\overline{u_1} \\
     -s_1 & r_1 \\
   \end{smallmatrix}
 \right]$ maps $p$ into $(1:0)(0:1)$ and $\bar{p}$ into $(0:1)(1:0)$. Now, we may assume that $p=(1:0)(0:1)$ and $\bar{p}=(0:1)(1:0)$ and $q=(\lambda:1)(\rho:1)$ with $\lambda,\rho\in\C^*$ because by hypothesis the points are not on the same fibres by any projection. The automorphism $(x,y)\mapsto(\lambda x,\bar\lambda y)$ fixes $p$ and $\bar p$ and sends $q$ into $(1:1)(1:\mu)$ and $\bar q$ into $(1:\bar\mu)(1:1)$.

Notice that when $\mu=1$ the points $q$ and $\bar q$ are equal; when $\mu=0$ the points $p$ and $\bar q$ are on the same fibre, as well as the points $\bar p$ and $q$; and finally, when $\mu=-1$ there is a diagonal passing through the four points. Hence, the blow-up of $p,\bar p, q,\bar q$ is not a Del Pezzo surface.
\end{proof}
\begin{prop}\label{kerEq}
\begin{enumerate}[$($a$)$]
\item The kernel of the sequence $(\ref{seq})$ is 
\[\F_\R=\{(a_1,\ldots,a_5)\in(\F_2)^5\ |\ a_1+a_2+a_3=0\text{\ and\ }a_4+a_5=0\}\cong(\F_2)^3,\] and is generated by the elements $\gamma_1=(0,1,1,0,0)$, $\gamma_2=(1,0,1,0,0)$, and $\gamma=(0,0,0,1,1)$ which correspond to the automorphisms of $\ X$ with coordinates in $\Pp^4$ given as
\begin{align*}
\gamma_1&\colon(y_1:y_2:y_3:y_4:y_5)\mapsto (y_1:y_2:-y_3:y_4:-y_5),\\
\gamma_2&\colon(y_1:y_2:y_3:y_4:y_5)\mapsto (y_1:y_2:y_3:-y_4:-y_5),\\
\gamma&\colon(y_1:y_2:y_3:y_4:y_5)\mapsto (y_1:y_2:-y_3:-y_4:-y_5).
\end{align*}
\item The equation of the surface $X$ is given by the intersection of the following two quadrics,
\begin{align*}
Q_1\colon&(\mu-\mu\overline{\mu}+\overline{\mu})y_1^2-2y_1y_2+y_2^2+(1-\overline{\mu}+\mu\overline{\mu}-\mu)y_3^2+y_4^2=0,\\
Q_2\colon&\mu\overline{\mu}y_1^2-2\mu\overline{\mu}y_1y_2+(\mu-1+\overline{\mu})y_2^2+\mu\overline{\mu}y_4^2+(1-\overline{\mu}+\mu\overline{\mu}-\mu)y_5^2=0.
\end{align*}
\end{enumerate}
\end{prop}
\begin{proof}
We first prove that $\F_\R$ is contained in the group $\{(a_1,\ldots,a_5)\in(\F_2)^5\ |\ a_1+a_2+a_3=0\text{\ and\ }a_4+a_5=0\}$. To do so, we focus on the pairs $P_4$ and $P_5$ and observe that the action of the antiholomorphic involution on those pairs (see Figure~\ref{picturepairs}) implies that for an automorphism $g$ of $X$, which is in the kernel, is of the form either $(*,* ,*,0,0)$ or $(*,*,*,1,1)$, which is the same as the condition $a_4+a_5=0$. Hence, $a_1+a_2+a_3=0$ because over $\C$, the kernel of the map $\rho\colon\Aut(X_\C)\to Sym_5$ is the set $\{(a_1,\ldots,a_5)\in(\F_2)^5\ |\ \sum a_i=0\}$ \cite[Lemma~9.11]{Bla09}.

We show the existence of $\gamma$, $\gamma_1$, and $\gamma_2$ and compute the equation of the surface $X$ using the fact that the anticanonical divisor $-K_X$ is very ample and then the linear system of $|-K_X|$ gives an embedding into $\Pp^4$ as an intersection of two quadrics. We study then the following diagram
\[\xymatrix@R-0.5pc{
X\  \ar[d]_{p,\overline{p},q,\overline{q}} \ar@{^{(}->}[r]^{|-K_X|} & \ \Pp^4  \\
S\ \ar@{-->}[ur]_\xi &
}\]
where the vertical map is the blow-up of four imaginary points $p$, $\overline{p}$, $q$, $\overline{q}$ of $S$ viewed $S_\C$ as $\Pp_\C^1\times\Pp^1_\C$ via the isomorphism $\varphi$ given in Remark \ref{varphi}. As $-K_X=2f+2\overline{f}-E_p-E_{\overline{p}}-E_q-E_{\overline{q}}$, the linear system $|-K_X|$ corresponds to the curves of $S$ of bidegree $(2,2)$ viewed on $\Pp^1_\C\times\Pp^1_\C\simeq S_\C$ passing through the four blow-up points.

By Lemma~\ref{points}, we may assume that $p=(1:0)(0:1)$ and $q=(1:1)(1:\mu)$ for some $\mu\in\C^*\setminus\{0,\pm1\}$, and then $\bar p=(0:1)(1:0)$ and $\bar q=(1:\bar\mu)(1:1)$. 

In coordinates $(r:s)(u:v)$ on $\Pp_\C^1\times\Pp_\C^1$, a basis of the linear system $|-K_X|$ is given by:

$$\begin{array}{l|l}
\Gamma_1=sv(r-s)(v-u) & (f-E_p)+(\overline{f}-E_{\overline{p}})+(f-E_q) \\
& +(\overline{f}-E_{\overline{q}}) \vphantom{\Big(}	\\ \hline
 \vphantom{\Big(}\Gamma_2=(vs-\overline{\mu}ru)(r-s)(v-u) & (f+\overline{f}-E_p-E_{\overline{p}}-E_{\overline{q}})+E_{\overline{q}} \\
& +(f-E_q)+(\overline{f}-E_{\overline{q}}) \vphantom{\Big(}\\ \hline
\Gamma_3=ur(v-\mu u)(s-\overline{\mu}r)\vphantom{\Big(} & (\overline{f}-E_p)+(f-E_{\overline{p}})+(\overline{f}-E_q)\\
& +(f-E_{\overline{q}}) \vphantom{\Big(}	\\ \hline
\vphantom{\Big(}\Gamma_4=(vs-\mu ru)(\overline{\mu}(1-\mu)ru & (f+\overline{f}-E_p-E_{\overline{p}}-E_q)+E_q\\
+(\mu-\overline{\mu})su+(\overline{\mu}-1)sv) & +(f+\overline{f}-E_p-E_q-E_{\overline{q}})+E_p	\vphantom{\Big(}\\ \hline
\vphantom{\Big(}\Gamma_5=(\mu(\overline{\mu}-1)ru+(\mu-\overline{\mu})rv & (f+\overline{f}-E_{\overline{p}}-E_q-E_{\overline{q}})+E_{\overline{q}}\\
+(1-\mu)sv)u(s-\overline{\mu}r) & +(\overline{f}-E_p)+(f-E_{\overline{q}}) \\ 
\end{array}$$

The computation of the actions of $\gamma_1$, $\gamma_2$, and $\gamma$ on $\Pic(X)$ with respect to the basis $\{\Gamma_1, \Gamma_2, \Gamma_3, \Gamma_4, \Gamma_5\}$ described above, gives the following elements:
{\small \[M_1=\left(
\begin{smallmatrix}
 0 & -\frac{\mu-\overline{\mu}}{\mu}  & 1 & \mu-\overline{\mu} & 1-\overline{\mu}  \\
0  & 1  & 0 & 0 & 0  \\
 1 &  0 & 0 & \mu-\overline{\mu} & 1-\overline{\mu}\\
 0 & \frac{1}{\mu} & 0 & -1 & 0 \\
 0  & 0 & 0 & 0 & -1
\end{smallmatrix}
\right),\
M_2=\left(
\begin{smallmatrix}
 1 & \frac{2\mu-\overline{\mu}}{\mu}  & 0 & 0 & 1-\overline{\mu}  \\
0  & -1  & 0 & 0 & 0  \\
 0 &  1 & 1 & 0 & \mu-2\overline{\mu}+1\\
 0 & -\frac{1}{\mu} & 0 & 1 & -1 \\
 0  & 0 & 0 & 0 & -1
\end{smallmatrix}
\right),
\text{ and }
M=\left(
\begin{smallmatrix}
 0 & -\frac{\mu-\overline{\mu}}{\mu}  & 1 & \mu-\overline{\mu} & 0  \\
0  & 1  & 0 & 0 & 0  \\
 1 &  0 & 0 & \mu-\overline{\mu} & \overline{\mu}-\mu\\
 0 & \frac{1}{\mu} & 0 & -1 & 1 \\
 0  & 0 & 0 & 0 & 1
\end{smallmatrix}
\right).\]}
By a change of the basis, the matrices $M_1$, $M_2$, and $M$ can be diagonalised and the map $\xi\colon S\rightarrow\Pp^4$ is given by
$((r:s),(u:v))\mapsto N\cdot y^t$ where
\[N=\left(
\begin{smallmatrix}
1 & 1 & -1&  -\overline\mu-\mu & \overline\mu  \\
0  & -\frac{1}{\mu}  & 0 & 2 & -1  \\
 1 &  1 & 1 & \mu-\overline{\mu} & 1-\overline{\mu}\\
 0 & 0 & 0 & 0 & -{\bf i} \\
 0  & -\frac{1}{\mu} & 0 & 0 & 0
\end{smallmatrix}
\right) \text{ and }y=(\Gamma_1,\ldots,\Gamma_5).\]

With this new basis, the surface $X$, which is the image of the anticanonical embedding, is given by the intersection of the two quadrics $Q_1$ and $Q_2$ in the statement as well as the automorphisms $\gamma_1$, $\gamma_2$, and $\gamma$.
\end{proof}

\begin{prop}\label{im}
The image of the sequence $(\ref{seq})$, $\rho(\Aut(X))\subset Sym_5$, is $\langle(2\ 3)(4\ 5)\rangle$ if $|\mu|=1$ and trivial otherwise.
\end{prop}

\begin{proof}
As already mentioned in Remark~\ref{Imagerho}, $\rho(\Aut(X))\subset \langle(2\ 3),(4\ 5)\rangle$.
We show that the elements $(2\ 3)$ and $(4\ 5)$ do not belong to the image while $(2\ 3)(4\ 5)$ does it if and only if $|\mu|=1$.

We start explaining why there is no automorphism of type $(2\ 3)$.  If there were an automorphism $\alpha$ exchanging the pair $P_2$ with $P_3$ then $\alpha$ would act on $P_2$ and $P_3$ either like \ \ \
\begin{picture}(45,20)
\scalebox{0.6}{
\put(0,-10){\line(0,1){40}}
\put(0,-10){\line(1,0){30}}
\put(30,-10){\line(0,1){40}}
\put(0,30){\line(1,0){30}}
\put(12,15){\unboldmath$\bullet$}
\put(12,0){\unboldmath$\bullet$}
\put(11,-17){\unboldmath{\scriptsize $P_2$}}
\put(14,18){\vector(1,0){38}}
\put(14,3){\vector(1,0){38}}
\put(56,18){\vector(-1,0){38}}
\put(56,3){\vector(-1,0){38}}

\put(40,-10){\line(0,1){40}}
\put(40,-10){\line(1,0){30}}
\put(70,-10){\line(0,1){40}}
\put(40,30){\line(1,0){30}}
\put(52,15){\unboldmath$\bullet$}
\put(52,0){\unboldmath$\bullet$}
\put(51,-17){\unboldmath{\scriptsize $P_3$}}}
\end{picture}\vspace{0.3cm}
or like \ \ \begin{picture}(45,20)
\scalebox{0.6}{

\put(0,-10){\line(0,1){40}}
\put(0,-10){\line(1,0){30}}
\put(30,-10){\line(0,1){40}}
\put(0,30){\line(1,0){30}}
\put(12,15){\unboldmath$\bullet$}
\put(12,0){\unboldmath$\bullet$}
\put(11,-17){\unboldmath{\scriptsize $P_2$}}
\put(14,18){\vector(3,-1){40}}
\put(14,4){\vector(3,1){40}}
\put(56,18){\vector(-3,-1){40}}
\put(56,4){\vector(-3,1){40}}

\put(40,-10){\line(0,1){40}}
\put(40,-10){\line(1,0){30}}
\put(70,-10){\line(0,1){40}}
\put(40,30){\line(1,0){30}}
\put(52,15){\unboldmath$\bullet$}
\put(52,0){\unboldmath$\bullet$}
\put(51,-17){\unboldmath{\scriptsize $P_3$}}}
\end{picture}.\vspace{0.3cm}

We may assume that the action on the pairs $P_2$ and $P_3$ is the first since we can multiply the second one by the element of $\F_\R$ that corresponds to $\gamma_1=(0,1,1,0,0)$. On the pairs $P_4$ and $P_5$, the action of $\alpha$ is either \ \begin{picture}(40,20)
\scalebox{0.6}{
\put(0,-10){\line(0,1){40}}
\put(0,-10){\line(1,0){30}}
\put(30,-10){\line(0,1){40}}
\put(0,30){\line(1,0){30}}
\put(12,15){\boldmath$\cdot$}
\put(12,0){\boldmath$\cdot$}
\put(11,-17){\unboldmath{\scriptsize $P_4$}}

\put(30,-10){\line(0,1){40}}
\put(30,-10){\line(1,0){30}}
\put(60,-10){\line(0,1){40}}
\put(30,30){\line(1,0){30}}
\put(42,15){\boldmath$\cdot$}
\put(42,0){\boldmath$\cdot$}
\put(41,-17){\unboldmath{\scriptsize $P_5$}}}
\end{picture}\vspace{0.3cm}
or \ \begin{picture}(40,20)
\scalebox{0.6}{
\put(0,-10){\line(0,1){40}}
\put(0,-10){\line(1,0){30}}
\put(30,-10){\line(0,1){40}}
\put(0,30){\line(1,0){30}}
\put(12,15){\unboldmath$\bullet$}
\put(12,0){\unboldmath$\bullet$}
\put(11,-17){\unboldmath{\scriptsize $P_4$}}
\put(15,3){\vector(0,1){13}}
\put(15,18){\vector(0,-1){13}}

\put(30,-10){\line(0,1){40}}
\put(30,-10){\line(1,0){30}}
\put(60,-10){\line(0,1){40}}
\put(30,30){\line(1,0){30}}
\put(42,15){\unboldmath$\bullet$}
\put(42,0){\unboldmath$\bullet$}
\put(41,-17){\unboldmath{\scriptsize $P_5$}}
\put(45,3){\vector(0,1){13}}
\put(45,18){\vector(0,-1){13}}}
\end{picture}. And as before, we may assume that it is the first one by multiplying the second one by $\gamma=(0,0,0,1,1)$. Summarising, we have to study only two cases:
\begin{enumerate}[(a)]
\begin{multicols}{2}
\item \begin{center}
\begin{picture}(150,15)
\scalebox{0.6}{
\put(10,-10){\line(0,1){40}}
\put(10,-10){\line(1,0){30}}
\put(40,-10){\line(0,1){40}}
\put(10,30){\line(1,0){30}}
\put(22,15){\boldmath$\cdot$}
\put(22,0){\boldmath$\cdot$}
\put(21,-17){\unboldmath{\scriptsize $P_1$}}

\put(50,-10){\line(0,1){40}}
\put(50,-10){\line(1,0){30}}
\put(80,-10){\line(0,1){40}}
\put(50,30){\line(1,0){30}}
\put(62,15){\unboldmath$\bullet$}
\put(62,0){\unboldmath$\bullet$}
\put(61,-17){\unboldmath{\scriptsize $P_2$}}
\put(64,18){\vector(1,0){38}}
\put(64,3){\vector(1,0){38}}
\put(106,18){\vector(-1,0){38}}
\put(106,3){\vector(-1,0){38}}

\put(90,-10){\line(0,1){40}}
\put(90,-10){\line(1,0){30}}
\put(120,-10){\line(0,1){40}}
\put(90,30){\line(1,0){30}}
\put(102,15){\unboldmath$\bullet$}
\put(102,0){\unboldmath$\bullet$}
\put(101,-17){\unboldmath{\scriptsize $P_3$}}

\put(130,-10){\line(0,1){40}}
\put(130,-10){\line(1,0){30}}
\put(160,-10){\line(0,1){40}}
\put(130,30){\line(1,0){30}}
\put(142,15){\boldmath$\cdot$}
\put(142,0){\boldmath$\cdot$}
\put(141,-17){\unboldmath{\scriptsize $P_4$}}

\put(160,-10){\line(0,1){40}}
\put(160,-10){\line(1,0){30}}
\put(190,-10){\line(0,1){40}}
\put(160,30){\line(1,0){30}}
\put(172,15){\boldmath$\cdot$}
\put(172,0){\boldmath$\cdot$}
\put(171,-17){\unboldmath{\scriptsize $P_5$}}}
\end{picture}
\end{center}
\item \begin{center}
\begin{picture}(150,15)
\scalebox{0.6}{
\put(10,-10){\line(0,1){40}}
\put(10,-10){\line(1,0){30}}
\put(40,-10){\line(0,1){40}}
\put(10,30){\line(1,0){30}}
\put(22,15){\unboldmath$\bullet$}
\put(22,0){\unboldmath$\bullet$}
\put(21,-17){\unboldmath{\scriptsize $P_1$}}
\put(25,3){\vector(0,1){13}}
\put(25,18){\vector(0,-1){13}}

\put(50,-10){\line(0,1){40}}
\put(50,-10){\line(1,0){30}}
\put(80,-10){\line(0,1){40}}
\put(50,30){\line(1,0){30}}
\put(62,15){\unboldmath$\bullet$}
\put(62,0){\unboldmath$\bullet$}
\put(61,-17){\unboldmath{\scriptsize $P_2$}}
\put(64,18){\vector(1,0){38}}
\put(64,3){\vector(1,0){38}}
\put(106,18){\vector(-1,0){38}}
\put(106,3){\vector(-1,0){38}}

\put(90,-10){\line(0,1){40}}
\put(90,-10){\line(1,0){30}}
\put(120,-10){\line(0,1){40}}
\put(90,30){\line(1,0){30}}
\put(102,15){\unboldmath$\bullet$}
\put(102,0){\unboldmath$\bullet$}
\put(101,-17){\unboldmath{\scriptsize $P_3$}}

\put(130,-10){\line(0,1){40}}
\put(130,-10){\line(1,0){30}}
\put(160,-10){\line(0,1){40}}
\put(130,30){\line(1,0){30}}
\put(142,15){\boldmath$\cdot$}
\put(142,0){\boldmath$\cdot$}
\put(141,-17){\unboldmath{\scriptsize $P_4$}}

\put(160,-10){\line(0,1){40}}
\put(160,-10){\line(1,0){30}}
\put(190,-10){\line(0,1){40}}
\put(160,30){\line(1,0){30}}
\put(172,15){\boldmath$\cdot$}
\put(172,0){\boldmath$\cdot$}
\put(171,-17){\unboldmath{\scriptsize $P_5$}}}
\end{picture}
\end{center}
\end{multicols}
\end{enumerate}

In both cases (a) and (b), $f$, $\overline{f}$ are fixed and hence $f+\overline{f}$ is fixed. In the case (a), looking at the pair $P_1$ we see that $f+\overline{f}-E_p-E_{\overline{p}}$, $f+\overline{f}-E_q-E_{\overline{q}}$ are fixed, then $E_p+E_{\overline{p}}$ and $E_q+E_{\overline{q}}$ are fixed while the action on pairs $P_2$ and $P_3$ gives that $\alpha$ interchanges $E_p+E_q$ with $E_p+E_{\overline{q}}$ and $E_{\overline{p}}+E_{\overline{q}}$ with $E_{\overline{p}}+E_q$. This implies that $E_p$, $E_{\overline{p}}$ are fixed and $E_q$, $E_{\overline{q}}$ are exchanged.
So $\alpha$ would come from an automorphism $\alpha'$ of $\Pp^1\times\Pp^1$ which fixes $p,\overline{p}$ and interchanges $q$ and $\overline{q}$. Let us see that such an $\alpha'$ does not exist.

The automorphism $\alpha'$ would be given by $(x,y)\mapsto(Ax,\overline{A}y)$ where $A\in\PGL(2,\C)$ with $\alpha'(p)=p$, $\alpha'(\overline{p})=\overline{p}$ then  $\alpha':(x,y)\mapsto(\lambda x,\overline{\lambda}y)$ with $\lambda\in\C$ under the choice of the points $p=(1:0)(0:1)$ and $q=(1:1)(1:\mu)$ for $\mu\notin\{0,\pm1\}$ (Lemma~\ref{points}). Since $\alpha'(q)=\overline{q}$, we have $\lambda=\overline\mu$ and $\overline{\lambda}\mu=1$ and hence $\mu^2=1$, which gives a contradiction.

In the case (b), $\alpha$ is not even an automorphism of the Picard group because the matrix corresponding to an action described in (b) with basis $\{f,\overline{f},E_p,E_{\overline{p}},E_q,E_{\overline{q}}\}$ is
\[\left(
  \begin{smallmatrix}
    1 & 0 & 0 & 0 & 0 & 0 \\
    0 & 1 & 0 & 0 & 0 & 0 \\
    0 & 0 & 1/2 & -1/2 & 1/2 & 1/2 \\
    0 & 0 & -1/2 & 1/2 & 1/2 & 1/2 \\
    0 & 0 & 1/2 & 1/2 & -1/2 & 1/2 \\
    0 & 0 & 1/2 & 1/2 & 1/2 & -1/2 \\
  \end{smallmatrix}
\right).\]
Therefore, an automorphism that acts as $(2\ 3)$ does not belong to the image.

Now, we prove that automorphisms of type $(4,5)$ are not in the image and we proceed in the same way as we did for $(2\ 3)$. The action of an automorphism of type $(4\ 5)$ on the pairs $P_4$ and $P_5$ is either like \begin{picture}(50,20)
\scalebox{0.6}{
\put(10,-10){\line(0,1){40}}
\put(10,-10){\line(1,0){30}}
\put(40,-10){\line(0,1){40}}
\put(10,30){\line(1,0){30}}
\put(22,15){\unboldmath$\bullet$}
\put(22,0){\unboldmath$\bullet$}
\put(21,-17){\unboldmath{\scriptsize $P_4$}}

\put(24,18){\vector(1,0){28}}
\put(56,18){\vector(-1,0){28}}
\put(24,3){\vector(1,0){28}}
\put(56,3){\vector(-1,0){28}}

\put(40,-10){\line(0,1){40}}
\put(40,-10){\line(1,0){30}}
\put(70,-10){\line(0,1){40}}
\put(40,30){\line(1,0){30}}
\put(52,15){\unboldmath$\bullet$}
\put(52,0){\unboldmath$\bullet$}
\put(51,-17){\unboldmath{\scriptsize $P_5$}}}
\end{picture}\vspace{0.3cm} or like
\begin{picture}(50,20)
\scalebox{0.6}{
\put(10,-10){\line(0,1){40}}
\put(10,-10){\line(1,0){30}}
\put(40,-10){\line(0,1){40}}
\put(10,30){\line(1,0){30}}
\put(22,15){\unboldmath$\bullet$}
\put(22,0){\unboldmath$\bullet$}
\put(21,-17){\unboldmath{\scriptsize $P_4$}}

\put(24,18){\vector(2,-1){28}}
\put(56,18){\vector(-2,-1){28}}
\put(24,2){\vector(2,1){28}}
\put(56,2){\vector(-2,1){28}}

\put(40,-10){\line(0,1){40}}
\put(40,-10){\line(1,0){30}}
\put(70,-10){\line(0,1){40}}
\put(40,30){\line(1,0){30}}
\put(52,15){\unboldmath$\bullet$}
\put(52,0){\unboldmath$\bullet$}
\put(51,-17){\unboldmath{\scriptsize $P_5$}}}
\end{picture}.
Multiplying by $(0,0,0,1,1)$ we may assume that is the first one. With respect to the action on the first three pairs $P_1$, $P_2$, and, $P_3$ we assume that the action on $P_1$ and $P_3$ is the identity since we can multiply by $(1,1,0,0,0)$ or by $(0,1,1,0,0)$. Then, we have two cases to focus on:
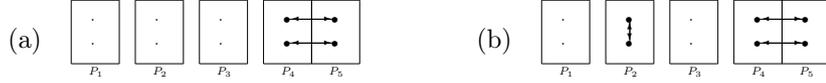
\begin{enumerate}[(a)]
\begin{multicols}{2}
\item \begin{center}
\begin{picture}(150,15)
\scalebox{0.6}{
\put(10,-10){\line(0,1){40}}
\put(10,-10){\line(1,0){30}}
\put(40,-10){\line(0,1){40}}
\put(10,30){\line(1,0){30}}
\put(22,15){\boldmath$\cdot$}
\put(22,0){\boldmath$\cdot$}
\put(21,-17){\unboldmath{\scriptsize $P_1$}}

\put(50,-10){\line(0,1){40}}
\put(50,-10){\line(1,0){30}}
\put(80,-10){\line(0,1){40}}
\put(50,30){\line(1,0){30}}
\put(62,15){\boldmath$\cdot$}
\put(62,0){\boldmath$\cdot$}
\put(61,-17){\unboldmath{\scriptsize $P_2$}}

\put(90,-10){\line(0,1){40}}
\put(90,-10){\line(1,0){30}}
\put(120,-10){\line(0,1){40}}
\put(90,30){\line(1,0){30}}
\put(102,15){\boldmath$\cdot$}
\put(102,0){\boldmath$\cdot$}
\put(101,-17){\unboldmath{\scriptsize $P_3$}}

\put(130,-10){\line(0,1){40}}
\put(130,-10){\line(1,0){30}}
\put(160,-10){\line(0,1){40}}
\put(130,30){\line(1,0){30}}
\put(142,15){\unboldmath$\bullet$}
\put(142,0){\unboldmath$\bullet$}
\put(141,-17){\unboldmath{\scriptsize $P_4$}}

\put(144,18){\vector(1,0){28}}
\put(176,18){\vector(-1,0){28}}
\put(144,3){\vector(1,0){28}}
\put(176,3){\vector(-1,0){28}}

\put(160,-10){\line(0,1){40}}
\put(160,-10){\line(1,0){30}}
\put(190,-10){\line(0,1){40}}
\put(160,30){\line(1,0){30}}
\put(172,15){\unboldmath$\bullet$}
\put(172,0){\unboldmath$\bullet$}
\put(171,-17){\unboldmath{\scriptsize $P_5$}}}
\end{picture}
\end{center}
\item \begin{center}
\begin{picture}(150,15)
\scalebox{0.6}{
\put(10,-10){\line(0,1){40}}
\put(10,-10){\line(1,0){30}}
\put(40,-10){\line(0,1){40}}
\put(10,30){\line(1,0){30}}
\put(22,15){\boldmath$\cdot$}
\put(22,0){\boldmath$\cdot$}
\put(21,-17){\unboldmath{\scriptsize $P_1$}}

\put(50,-10){\line(0,1){40}}
\put(50,-10){\line(1,0){30}}
\put(80,-10){\line(0,1){40}}
\put(50,30){\line(1,0){30}}
\put(62,15){\unboldmath$\bullet$}
\put(62,0){\unboldmath$\bullet$}
\put(61,-17){\unboldmath{\scriptsize $P_2$}}
\put(65,3){\vector(0,1){13}}
\put(65,18){\vector(0,-1){13}}

\put(90,-10){\line(0,1){40}}
\put(90,-10){\line(1,0){30}}
\put(120,-10){\line(0,1){40}}
\put(90,30){\line(1,0){30}}
\put(102,15){\boldmath$\cdot$}
\put(102,0){\boldmath$\cdot$}
\put(101,-17){\unboldmath{\scriptsize $P_3$}}

\put(130,-10){\line(0,1){40}}
\put(130,-10){\line(1,0){30}}
\put(160,-10){\line(0,1){40}}
\put(130,30){\line(1,0){30}}
\put(142,15){\unboldmath$\bullet$}
\put(142,0){\unboldmath$\bullet$}
\put(141,-17){\unboldmath{\scriptsize $P_4$}}

\put(144,18){\vector(1,0){28}}
\put(176,18){\vector(-1,0){28}}
\put(144,3){\vector(1,0){28}}
\put(176,3){\vector(-1,0){28}}

\put(160,-10){\line(0,1){40}}
\put(160,-10){\line(1,0){30}}
\put(190,-10){\line(0,1){40}}
\put(160,30){\line(1,0){30}}
\put(172,15){\unboldmath$\bullet$}
\put(172,0){\unboldmath$\bullet$}
\put(171,-17){\unboldmath{\scriptsize $P_5$}}}
\end{picture}
\end{center}
\end{multicols}
\end{enumerate}
The case (a) corresponds to an automorphism which interchanges $f$ with $\overline{{f}}$ and fixes $E_p$, $E_{\overline{p}}$, $E_q$, and $E_{\overline{q}}$. It would be the lift of an automorphism of $S$ fixing 4 points which does not exist. On the other hand, the case (b) is not an automorphism of the Picard group because the matrix corresponding to it is
\[\left(
  \begin{smallmatrix}
    0 & 1 & 0 & 0 & 0 & 0 \\
    1 & 0 & 0 & 0 & 0 & 0 \\
    0 & 0 & 1/2 & 1/2 & -1/2 & 1/2 \\
    0 & 0 & 1/2 & 1/2 & 1/2 & -1/2 \\
    0 & 0 & -1/2 & 1/2 & 1/2 & 1/2 \\
    0 & 0 & 1/2 & -1/2 & 1/2 & 1/2 \\
  \end{smallmatrix}
\right).\]

Finally, we check that there is an automorphism which acts as $(2\ 3)(4\ 5)$ if and only if $|\mu|=1$. As before, we can see that automorphisms corresponding to $(2\ 3)(4\ 5)$ are, up to composition with an element of $\F_\R$, of the form
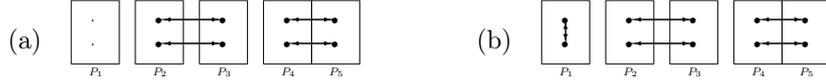
\begin{enumerate}[(a)]
\begin{multicols}{2}
\item
\begin{center}
\begin{picture}(150,15)
\scalebox{0.6}{
\put(10,-10){\line(0,1){40}}
\put(10,-10){\line(1,0){30}}
\put(40,-10){\line(0,1){40}}
\put(10,30){\line(1,0){30}}
\put(22,15){\boldmath$\cdot$}
\put(22,0){\boldmath$\cdot$}
\put(21,-17){\unboldmath{\scriptsize $P_1$}}

\put(50,-10){\line(0,1){40}}
\put(50,-10){\line(1,0){30}}
\put(80,-10){\line(0,1){40}}
\put(50,30){\line(1,0){30}}
\put(62,15){\unboldmath$\bullet$}
\put(62,0){\unboldmath$\bullet$}
\put(61,-17){\unboldmath{\scriptsize $P_2$}}
\put(64,18){\vector(1,0){38}}
\put(64,3){\vector(1,0){38}}
\put(106,18){\vector(-1,0){38}}
\put(106,3){\vector(-1,0){38}}

\put(90,-10){\line(0,1){40}}
\put(90,-10){\line(1,0){30}}
\put(120,-10){\line(0,1){40}}
\put(90,30){\line(1,0){30}}
\put(102,15){\unboldmath$\bullet$}
\put(102,0){\unboldmath$\bullet$}
\put(101,-17){\unboldmath{\scriptsize $P_3$}}

\put(130,-10){\line(0,1){40}}
\put(130,-10){\line(1,0){30}}
\put(160,-10){\line(0,1){40}}
\put(130,30){\line(1,0){30}}
\put(142,15){\unboldmath$\bullet$}
\put(142,0){\unboldmath$\bullet$}
\put(141,-17){\unboldmath{\scriptsize $P_4$}}

\put(144,18){\vector(1,0){28}}
\put(176,18){\vector(-1,0){28}}
\put(144,3){\vector(1,0){28}}
\put(176,3){\vector(-1,0){28}}

\put(160,-10){\line(0,1){40}}
\put(160,-10){\line(1,0){30}}
\put(190,-10){\line(0,1){40}}
\put(160,30){\line(1,0){30}}
\put(172,15){\unboldmath$\bullet$}
\put(172,0){\unboldmath$\bullet$}
\put(171,-17){\unboldmath{\scriptsize $P_5$}}}
\end{picture}
\end{center}
\item
\begin{center}
\begin{picture}(150,15)
\scalebox{0.6}{
\put(10,-10){\line(0,1){40}}
\put(10,-10){\line(1,0){30}}
\put(40,-10){\line(0,1){40}}
\put(10,30){\line(1,0){30}}
\put(22,15){\unboldmath$\bullet$}
\put(22,0){\unboldmath$\bullet$}
\put(21,-17){\unboldmath{\scriptsize $P_1$}}
\put(25,3){\vector(0,1){13}}
\put(25,18){\vector(0,-1){13}}

\put(50,-10){\line(0,1){40}}
\put(50,-10){\line(1,0){30}}
\put(80,-10){\line(0,1){40}}
\put(50,30){\line(1,0){30}}
\put(62,15){\unboldmath$\bullet$}
\put(62,0){\unboldmath$\bullet$}
\put(61,-17){\unboldmath{\scriptsize $P_2$}}
\put(64,18){\vector(1,0){38}}
\put(64,3){\vector(1,0){38}}
\put(106,18){\vector(-1,0){38}}
\put(106,3){\vector(-1,0){38}}

\put(90,-10){\line(0,1){40}}
\put(90,-10){\line(1,0){30}}
\put(120,-10){\line(0,1){40}}
\put(90,30){\line(1,0){30}}
\put(102,15){\unboldmath$\bullet$}
\put(102,0){\unboldmath$\bullet$}
\put(101,-17){\unboldmath{\scriptsize $P_3$}}

\put(130,-10){\line(0,1){40}}
\put(130,-10){\line(1,0){30}}
\put(160,-10){\line(0,1){40}}
\put(130,30){\line(1,0){30}}
\put(142,15){\unboldmath$\bullet$}
\put(142,0){\unboldmath$\bullet$}
\put(141,-17){\unboldmath{\scriptsize $P_4$}}

\put(144,18){\vector(1,0){28}}
\put(176,18){\vector(-1,0){28}}
\put(144,3){\vector(1,0){28}}
\put(176,3){\vector(-1,0){28}}

\put(160,-10){\line(0,1){40}}
\put(160,-10){\line(1,0){30}}
\put(190,-10){\line(0,1){40}}
\put(160,30){\line(1,0){30}}
\put(172,15){\unboldmath$\bullet$}
\put(172,0){\unboldmath$\bullet$}
\put(171,-17){\unboldmath{\scriptsize $P_5$}}}
\end{picture}
\end{center}
\end{multicols}
\end{enumerate}
For the case (a), looking at the pairs $P_4$ and $P_5$ we see that $f$ and $\overline{f}$ are exchanged and then $f+\overline{f}$ is fixed. The exchange of pairs $P_2$ and $P_3$ gives that $f+\overline{f}-E_p-E_q$ and $f+\overline{f}-E_p-E_{\overline{q}}$ are interchanged and so are $f+\overline{f}-E_{\overline{p}}-E_{\overline{q}}$ and $f+\overline{f}-E_{\overline{p}}-E_q$. This implies that $E_p+E_q$ with $E_p+E_{\overline{q}}$ are interchanged and $E_{\overline{p}}+E_{\overline{q}}$ with $E_{\overline{p}}+E_q$ are interchanged, respectively. So an automorphism of type $(2\ 3)(4\ 5)$ for case (a) comes from an automorphism $\delta$ of $\Pp^1\times\Pp^1$ which interchanges $f$ with $\overline{f}$, $q$ with $\overline{q}$ and fixes $p$ and $\overline{p}$. We want to show that $\delta$ exists if and only if $|\mu|=1$. So $\delta$ is given by
$\delta:(x,y)\mapsto(\overline{A}y,Ax)$ satisfying
$\overline{A}\left[\begin{smallmatrix}
		0\\
		1\end{smallmatrix}\right]=\left[\begin{smallmatrix}
		1\\
		0\end{smallmatrix}\right]$,
$A\left[\begin{smallmatrix}
		1\\
		0\end{smallmatrix}\right]
		=\left[\begin{smallmatrix}
		0\\
		1\end{smallmatrix}\right].$
This implies that $A=\left[\begin{smallmatrix}
                     0 & \lambda \\
                     1 & 0 \\
                   \end{smallmatrix}\right].$
Since $\delta$ interchanges $q$ with $\overline{q}$, then $\left[\begin{smallmatrix}
                     0 & \overline{\lambda} \\
                     1 & 0 \\
                   \end{smallmatrix}\right]
                   \left[\begin{smallmatrix}
		   1\\
		   \mu\end{smallmatrix}\right]
		=\left[\begin{smallmatrix}
		1\\
		\overline{\mu}\end{smallmatrix}\right]
		=\left[\begin{smallmatrix}
		\overline{\lambda}\mu\\
		1\end{smallmatrix}\right]$ and
$\left[\begin{smallmatrix}
                     0 & \lambda \\
                     1 & 0 \\
                   \end{smallmatrix}\right]
                   \left[\begin{smallmatrix}
		1\\
		1\end{smallmatrix}\right]
		=\left[\begin{smallmatrix}
		1\\
		1\end{smallmatrix}\right]
		=\left[\begin{smallmatrix}
		\lambda\\
		1\end{smallmatrix}\right].$
Hence, $\lambda=1$ and $\mu\overline{\mu}=1$. Therefore this automorphism exists if $|\mu|=1$.

The case (b) is not possible because the matrix of the action of it on the Picard group with basis $\{f,\overline{f},E_p,E_{\overline{p}},E_q,E_{\overline{q}}\}$ is \[\left(
  \begin{smallmatrix}
    0 & 1 & 0 & 0 & 0 & 0 \\
    1 & 0 & 0 & 0 & 0 & 0 \\
    0 & 0 & 1/2 & -1/2 & 1/2 & 1/2 \\
    0 & 0 & -1/2 & 1/2 & 1/2 & 1/2 \\
    0 & 0 & 1/2 & 1/2 & -1/2 & 1/2 \\
    0 & 0 & 1/2 & 1/2 & 1/2 & -1/2 \\
  \end{smallmatrix}
\right)\]
and this shows that it is not an automorphism of the Picard group.
\end{proof}

\begin{prop} \label{dP4rank1}
If $g\in\Aut(X)$ and $\Pic(X)^g$ has rank one, then $g$ is either $\alpha_1=(1,1,0,1,1)$ or $\alpha_2=(1,0,1,1,1)$ in $\F_\R$ which are given by
\begin{align*}
\alpha_1\colon&(y_1:y_2:y_3:y_4:y_5)\mapsto (y_1:y_2:y_3:y_4:-y_5),\\
\alpha_2\colon&(y_1:y_2:y_3:y_4:y_5)\mapsto (y_1:y_2:-y_3:y_4:y_5).
\end{align*}
\end{prop}
\begin{proof}
Let $g\in\Aut(X)$ of prime order. If $g\in\F_\R$, $g=(a_1,\ldots,a_5)$ and the condition on the rank forces that the first component $a_1=1$,  $g$ is thus either $(1,1,0,*,*)$ or $(1,0,1,*,*)$. Moreover, we observe that $g$ must interchange the two conic bundles in the pairs $P_4$ and $P_5$ because otherwise, $g(f+\bar f)=f+\bar f\in\Pic(X)^g$ implying that the rank of $\Pic(X)^g>1$ since $f+\bar f$ is not multiple of $-K_X$. Then the two possibilities for $g$ when $g\in\F_\R$ are $\alpha_1=(1,1,0,1,1)$ and $\alpha_2=(1,0,1,1,1)$.

Now if $g\notin\F_\R$, Proposition \ref{im} tells us that the action of $\Aut(X)$ on the five pairs is $\langle(2\ 3)(4\ 5)\rangle$. To ask that $\Pic(X)^g\cong\Z$ forces that the two conic bundle structures in the first pair are interchanged for the same reason as before. On the other hand, the action of $(2\ 3)(4\ 5)$ on the pairs $P_2$ and $P_3$ cannot be of the form 
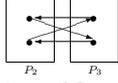
\begin{picture}(40,25)
\scalebox{0.6}{
\put(0,-5){\line(0,1){40}}
\put(0,-5){\line(1,0){30}}
\put(30,-5){\line(0,1){40}}
\put(0,35){\line(1,0){30}}
\put(12,20){\unboldmath$\bullet$}
\put(12,5){\unboldmath$\bullet$}
\put(11,-12){\unboldmath{\scriptsize $P_2$}}

\put(14,23){\vector(3,-1){40}}
\put(14,9){\vector(3,1){40}}

\put(56,23){\vector(-1,0){38}}
\put(56,8){\vector(-1,0){38}}

\put(40,-5){\line(0,1){40}}
\put(40,-5){\line(1,0){30}}
\put(70,-5){\line(0,1){40}}
\put(40,35){\line(1,0){30}}
\put(52,20){\unboldmath$\bullet$}
\put(52,5){\unboldmath$\bullet$}
\put(51,-12){\unboldmath{\scriptsize $P_3$}}}
\end{picture}\vspace{0.2cm} \ (or the one reversing the arrows) because in this case the order of $g$ is 4. In addition, we observe that if the action of $(2\ 3)(4\ 5)$ on the pairs $P_4$ and $P_5$ is as in this picture:  \begin{picture}(50,20)
\scalebox{0.6}{
\put(10,-10){\line(0,1){40}}
\put(10,-10){\line(1,0){30}}
\put(40,-10){\line(0,1){40}}
\put(10,30){\line(1,0){30}}
\put(22,15){\unboldmath$\bullet$}
\put(22,0){\unboldmath$\bullet$}
\put(21,-17){\unboldmath{\scriptsize $P_4$}}

\put(24,17.5){\vector(1,0){28}}
\put(56,17.5){\vector(-1,0){28}}
\put(24,2.5){\vector(1,0){28}}
\put(56,2.5){\vector(-1,0){28}}

\put(40,-10){\line(0,1){40}}
\put(40,-10){\line(1,0){30}}
\put(70,-10){\line(0,1){40}}
\put(40,30){\line(1,0){30}}
\put(52,15){\unboldmath$\bullet$}
\put(52,0){\unboldmath$\bullet$}
\put(51,-17){\unboldmath{\scriptsize $P_5$}}}
\end{picture}\vspace{0.3cm}, the divisor $f+\bar f$ is preserved under $g$ and $\sigma$, then $f+\bar f\in\Pic(X)^g$. This implies that $\rk(\Pic(X)^g)>1$. 

We have then to check the remaining cases,
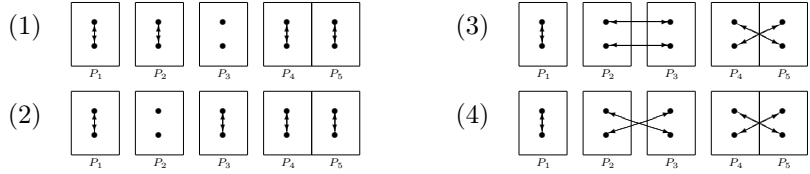
\begin{enumerate}
\begin{multicols}{2}
\item \begin{center}
\begin{picture}(150,0)
\scalebox{0.6}{
\put(10,-20){\line(0,1){40}}
\put(10,-20){\line(1,0){30}}
\put(40,-20){\line(0,1){40}}
\put(10,20){\line(1,0){30}}
\put(22,5){\unboldmath$\bullet$}
\put(22,-10){\unboldmath$\bullet$}
\put(21,-27){\unboldmath{\scriptsize $P_1$}}
\put(24.5,-7){\vector(0,1){13}}
\put(24.5,8){\vector(0,-1){13}}

\put(50,-20){\line(0,1){40}}
\put(50,-20){\line(1,0){30}}
\put(80,-20){\line(0,1){40}}
\put(50,20){\line(1,0){30}}
\put(62,5){\unboldmath$\bullet$}
\put(62,-10){\unboldmath$\bullet$}
\put(61,-27){\unboldmath{\scriptsize $P_2$}}
\put(64.5,-7){\vector(0,1){13}}
\put(64.5,8){\vector(0,-1){13}}

\put(90,-20){\line(0,1){40}}
\put(90,-20){\line(1,0){30}}
\put(120,-20){\line(0,1){40}}
\put(90,20){\line(1,0){30}}
\put(102,5){\unboldmath$\bullet$}
\put(102,-10){\unboldmath$\bullet$}
\put(101,-27){\unboldmath{\scriptsize $P_3$}}

\put(130,-20){\line(0,1){40}}
\put(130,-20){\line(1,0){30}}
\put(160,-20){\line(0,1){40}}
\put(130,20){\line(1,0){30}}
\put(142,5){\unboldmath$\bullet$}
\put(142,-10){\unboldmath$\bullet$}
\put(141,-27){\unboldmath{\scriptsize $P_4$}}
\put(144.5,-7){\vector(0,1){13}}
\put(144.5,8){\vector(0,-1){13}}

\put(160,-20){\line(0,1){40}}
\put(160,-20){\line(1,0){30}}
\put(190,-20){\line(0,1){40}}
\put(160,20){\line(1,0){30}}
\put(172,5){\unboldmath$\bullet$}
\put(172,-10){\unboldmath$\bullet$}
\put(171,-27){\unboldmath{\scriptsize $P_5$}}
\put(174.5,-7){\vector(0,1){13}}
\put(174.5,8){\vector(0,-1){13}}
}
\end{picture}
\end{center}
\item \begin{center}
\begin{picture}(150,30)
\scalebox{0.6}{
\put(10,-20){\line(0,1){40}}
\put(10,-20){\line(1,0){30}}
\put(40,-20){\line(0,1){40}}
\put(10,20){\line(1,0){30}}
\put(22,5){\unboldmath$\bullet$}
\put(22,-10){\unboldmath$\bullet$}
\put(21,-27){\unboldmath{\scriptsize $P_1$}}
\put(24.5,-7){\vector(0,1){13}}
\put(24.5,8){\vector(0,-1){13}}

\put(50,-20){\line(0,1){40}}
\put(50,-20){\line(1,0){30}}
\put(80,-20){\line(0,1){40}}
\put(50,20){\line(1,0){30}}
\put(62,5){\unboldmath$\bullet$}
\put(62,-10){\unboldmath$\bullet$}
\put(61,-27){\unboldmath{\scriptsize $P_2$}}

\put(90,-20){\line(0,1){40}}
\put(90,-20){\line(1,0){30}}
\put(120,-20){\line(0,1){40}}
\put(90,20){\line(1,0){30}}
\put(102,5){\unboldmath$\bullet$}
\put(102,-10){\unboldmath$\bullet$}
\put(101,-27){\unboldmath{\scriptsize $P_3$}}
\put(104.5,-7){\vector(0,1){13}}
\put(104.5,8){\vector(0,-1){13}}

\put(130,-20){\line(0,1){40}}
\put(130,-20){\line(1,0){30}}
\put(160,-20){\line(0,1){40}}
\put(130,20){\line(1,0){30}}
\put(142,5){\unboldmath$\bullet$}
\put(142,-10){\unboldmath$\bullet$}
\put(141,-27){\unboldmath{\scriptsize $P_4$}}
\put(144.5,-7){\vector(0,1){13}}
\put(144.5,8){\vector(0,-1){13}}

\put(160,-20){\line(0,1){40}}
\put(160,-20){\line(1,0){30}}
\put(190,-20){\line(0,1){40}}
\put(160,20){\line(1,0){30}}
\put(172,5){\unboldmath$\bullet$}
\put(172,-10){\unboldmath$\bullet$}
\put(171,-27){\unboldmath{\scriptsize $P_5$}}
\put(174.5,-7){\vector(0,1){13}}
\put(174.5,8){\vector(0,-1){13}}
}
\end{picture}
\end{center} 
\item \begin{center} 
\begin{picture}(150,0) 
\scalebox{0.6}{
\put(10,-20){\line(0,1){40}}
\put(10,-20){\line(1,0){30}}
\put(40,-20){\line(0,1){40}}
\put(10,20){\line(1,0){30}}
\put(22,5){\unboldmath$\bullet$}
\put(22,-10){\unboldmath$\bullet$}
\put(21,-27){\unboldmath{\scriptsize $P_1$}}
\put(24.5,-7){\vector(0,1){13}}
\put(24.5,8){\vector(0,-1){13}}

\put(50,-20){\line(0,1){40}}
\put(50,-20){\line(1,0){30}}
\put(80,-20){\line(0,1){40}}
\put(50,20){\line(1,0){30}}
\put(62,5){\unboldmath$\bullet$}
\put(62,-10){\unboldmath$\bullet$}
\put(61,-27){\unboldmath{\scriptsize $P_2$}}
\put(65,8){\vector(1,0){38}}
\put(65,-7){\vector(1,0){38}}
\put(105,8){\vector(-1,0){38}}
\put(105,-7){\vector(-1,0){38}}

\put(90,-20){\line(0,1){40}}
\put(90,-20){\line(1,0){30}}
\put(120,-20){\line(0,1){40}}
\put(90,20){\line(1,0){30}}
\put(102,5){\unboldmath$\bullet$}
\put(102,-10){\unboldmath$\bullet$}
\put(101,-27){\unboldmath{\scriptsize $P_3$}}

\put(130,-20){\line(0,1){40}}
\put(130,-20){\line(1,0){30}}
\put(160,-20){\line(0,1){40}}
\put(130,20){\line(1,0){30}}
\put(142,5){\unboldmath$\bullet$}
\put(142,-10){\unboldmath$\bullet$}
\put(141,-27){\unboldmath{\scriptsize $P_4$}}

\put(144,8){\vector(2,-1){28}}
\put(176,8){\vector(-2,-1){28}}
\put(144,-8){\vector(2,1){28}}
\put(176,-8){\vector(-2,1){28}}

\put(160,-20){\line(0,1){40}}
\put(160,-20){\line(1,0){30}}
\put(190,-20){\line(0,1){40}}
\put(160,20){\line(1,0){30}}
\put(172,5){\unboldmath$\bullet$}
\put(172,-10){\unboldmath$\bullet$}
\put(171,-27){\unboldmath{\scriptsize $P_5$}}
}
\end{picture}
\end{center} 
\item \begin{center}
\begin{picture}(150,30) 
\scalebox{0.6}{
\put(10,-20){\line(0,1){40}}
\put(10,-20){\line(1,0){30}}
\put(40,-20){\line(0,1){40}}
\put(10,20){\line(1,0){30}}
\put(22,5){\unboldmath$\bullet$}
\put(22,-10){\unboldmath$\bullet$}
\put(21,-27){\unboldmath{\scriptsize $P_1$}}
\put(24.5,-7){\vector(0,1){13}}
\put(24.5,8){\vector(0,-1){13}}

\put(50,-20){\line(0,1){40}}
\put(50,-20){\line(1,0){30}}
\put(80,-20){\line(0,1){40}}
\put(50,20){\line(1,0){30}}
\put(62,5){\unboldmath$\bullet$}
\put(62,-10){\unboldmath$\bullet$}
\put(61,-27){\unboldmath{\scriptsize $P_2$}}

\put(64,7){\vector(3,-1){40}}
\put(64,-7){\vector(3,1){40}}
\put(106,7){\vector(-3,-1){40}}
\put(106,-7){\vector(-3,1){40}}

\put(90,-20){\line(0,1){40}}
\put(90,-20){\line(1,0){30}}
\put(120,-20){\line(0,1){40}}
\put(90,20){\line(1,0){30}}
\put(102,5){\unboldmath$\bullet$}
\put(102,-10){\unboldmath$\bullet$}
\put(101,-27){\unboldmath{\scriptsize $P_3$}}

\put(130,-20){\line(0,1){40}}
\put(130,-20){\line(1,0){30}}
\put(160,-20){\line(0,1){40}}
\put(130,20){\line(1,0){30}}
\put(142,5){\unboldmath$\bullet$}
\put(142,-10){\unboldmath$\bullet$}
\put(141,-27){\unboldmath{\scriptsize $P_4$}}

\put(144,8){\vector(2,-1){28}}
\put(176,8){\vector(-2,-1){28}}
\put(144,-8){\vector(2,1){28}}
\put(176,-8){\vector(-2,1){28}}

\put(160,-20){\line(0,1){40}}
\put(160,-20){\line(1,0){30}}
\put(190,-20){\line(0,1){40}}
\put(160,20){\line(1,0){30}}
\put(172,5){\unboldmath$\bullet$}
\put(172,-10){\unboldmath$\bullet$}
\put(171,-27){\unboldmath{\scriptsize $P_5$}}
}
\end{picture}
\end{center} 
\end{multicols}
\end{enumerate}
\vspace{0.3cm}

The case (2) can be seen from case (1) conjugating it by the automorphism of the Picard group interchanging the divisors $E_q$ with $E_{\bar q}$ and fixing $f$, $\bar f$, $E_p$, and $E_{\bar p}$. Now, the action of the automorphisms of the case (1) on the Picard group $\Pic(X)$ with respect to the basis $\{f,\overline{f},E_p,E_{\overline{p}},E_q,E_{\overline{q}}\}$ is
\[\left(\begin{smallmatrix}
1 & 2 & 1 & 1& 1& 1\\
2& 1 & 1& 1& 1& 1\\
-1& -1& -1& -1& -1& 0\\
-1& -1& -1& -1& 0& -1\\
-1& -1& -1& 0& -1& -1\\
-1& -1& 0& -1& -1& -1
\end{smallmatrix}
\right).\]
In this case that corresponds to $\alpha_1$, the eigenspace for the eigenvalue $1$ is generated by the two conic bundles of the pair $P_3$ which are not in $\Pic(X)^g$ because of the action of $\sigma$ interchanges them but whose sum is $-K_X$. Hence, $\Pic(X)^g\cong\Z$ and therefore in case (2) as well when $g=\alpha_2$. By Proposition \ref{kerEq}, $\alpha_1=\gamma_1\gamma_2\gamma$ and $\alpha_2=\gamma_2\gamma$ which are exactly the maps in the statement.

Finally, for cases (3) and (4), the element $g$ is not even an automorphism of the Picard group because matrices corresponding to an action described in these cases with basis $\{f,\overline{f},E_p,E_{\overline{p}},E_q,E_{\overline{q}}\}$ are
\[\left(\begin{smallmatrix}
2 & 1 & 1 & 1& 1& 1\\
1 & 2 & 1& 1& 1& 1\\
-1& -1& -\frac{1}{2} & -\frac{3}{2} & -\frac{1}{2} & -\frac{1}{2}\\
-1& -1& -\frac{3}{2} & -\frac{1}{2} & -\frac{1}{2} & -\frac{1}{2}\\
-1& -1& -\frac{1}{2} & -\frac{1}{2} & -\frac{3}{2} & -\frac{1}{2}\\
-1& -1& -\frac{1}{2} & -\frac{1}{2} & -\frac{1}{2} & -\frac{3}{2}
\end{smallmatrix}
\right)\text{\ and\ } \left(\begin{smallmatrix}
2 & 1 & 1 & 1& 1& 1\\
1& 2 & 1& 1& 1& 1\\
-1& -1& -\frac{3}{2} & -\frac{1}{2} & -\frac{1}{2} & -\frac{1}{2}\\
-1& -1& -\frac{1}{2} & -\frac{3}{2} & -\frac{1}{2} & -\frac{1}{2}\\
-1& -1& -\frac{1}{2} & -\frac{1}{2} & -\frac{1}{2} & -\frac{3}{2}\\
-1& -1& -\frac{1}{2} & -\frac{1}{2} & -\frac{3}{2} & -\frac{1}{2}
\end{smallmatrix}
\right),\text{ respectively.}\qedhere\]
\end{proof}

There are automorphisms of Del Pezzo surfaces of degree $4$ which are minimal but preserve a conic bundle structure. These will be needed in the sequel. We give here a special family of examples.
\begin{lemma}\label{Lem:g1g2DP4}
If $\lvert\mu\rvert =1$, then $X$ admits two automorphisms $g_1$, $g_2\in \Aut(X)$ of order $2$, acting on the conic bundles like
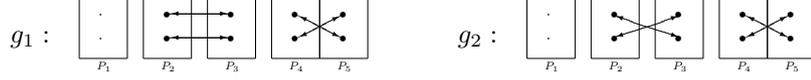
\begin{enumerate}
\begin{multicols}{2}
\item[$g_1:$]
\begin{center}
\begin{picture}(150,15)
\scalebox{0.6}{

\put(10,-10){\line(0,1){40}}
\put(10,-10){\line(1,0){30}}
\put(40,-10){\line(0,1){40}}
\put(10,30){\line(1,0){30}}
\put(22,15){\boldmath$\cdot$}
\put(22,0){\boldmath$\cdot$}
\put(21,-17){\unboldmath{\scriptsize $P_1$}}

\put(50,-10){\line(0,1){40}}
\put(50,-10){\line(1,0){30}}
\put(80,-10){\line(0,1){40}}
\put(50,30){\line(1,0){30}}
\put(62,15){\unboldmath$\bullet$}
\put(62,0){\unboldmath$\bullet$}
\put(61,-17){\unboldmath{\scriptsize $P_2$}}
\put(64,18){\vector(1,0){38}}
\put(64,3){\vector(1,0){38}}
\put(106,18){\vector(-1,0){38}}
\put(106,3){\vector(-1,0){38}}

\put(90,-10){\line(0,1){40}}
\put(90,-10){\line(1,0){30}}
\put(120,-10){\line(0,1){40}}
\put(90,30){\line(1,0){30}}
\put(102,15){\unboldmath$\bullet$}
\put(102,0){\unboldmath$\bullet$}
\put(101,-17){\unboldmath{\scriptsize $P_3$}}

\put(130,-10){\line(0,1){40}}
\put(130,-10){\line(1,0){30}}
\put(160,-10){\line(0,1){40}}
\put(130,30){\line(1,0){30}}
\put(142,15){\unboldmath$\bullet$}
\put(142,0){\unboldmath$\bullet$}
\put(141,-17){\unboldmath{\scriptsize $P_4$}}

\put(144,18){\vector(2,-1){28}}
\put(176,18){\vector(-2,-1){28}}
\put(144,2){\vector(2,1){28}}
\put(176,2){\vector(-2,1){28}}

\put(160,-10){\line(0,1){40}}
\put(160,-10){\line(1,0){30}}
\put(190,-10){\line(0,1){40}}
\put(160,30){\line(1,0){30}}
\put(172,15){\unboldmath$\bullet$}
\put(172,0){\unboldmath$\bullet$}
\put(171,-17){\unboldmath{\scriptsize $P_5$}}}
\end{picture}
\end{center}
\item[$g_2:$]
\begin{center}
\begin{picture}(150,15)
\scalebox{0.6}{
\put(10,-10){\line(0,1){40}}
\put(10,-10){\line(1,0){30}}
\put(40,-10){\line(0,1){40}}
\put(10,30){\line(1,0){30}}
\put(22,15){\boldmath$\cdot$}
\put(22,0){\boldmath$\cdot$}
\put(21,-17){\unboldmath{\scriptsize $P_1$}}

\put(50,-10){\line(0,1){40}}
\put(50,-10){\line(1,0){30}}
\put(80,-10){\line(0,1){40}}
\put(50,30){\line(1,0){30}}
\put(62,15){\unboldmath$\bullet$}
\put(62,0){\unboldmath$\bullet$}
\put(61,-17){\unboldmath{\scriptsize $P_2$}}

\put(64,17){\vector(3,-1){40}}
\put(64,3){\vector(3,1){40}}
\put(106,17){\vector(-3,-1){40}}
\put(106,3){\vector(-3,1){40}}

\put(90,-10){\line(0,1){40}}
\put(90,-10){\line(1,0){30}}
\put(120,-10){\line(0,1){40}}
\put(90,30){\line(1,0){30}}
\put(102,15){\unboldmath$\bullet$}
\put(102,0){\unboldmath$\bullet$}
\put(101,-17){\unboldmath{\scriptsize $P_3$}}

\put(130,-10){\line(0,1){40}}
\put(130,-10){\line(1,0){30}}
\put(160,-10){\line(0,1){40}}
\put(130,30){\line(1,0){30}}
\put(142,15){\unboldmath$\bullet$}
\put(142,0){\unboldmath$\bullet$}
\put(141,-17){\unboldmath{\scriptsize $P_4$}}

\put(144,18){\vector(2,-1){28}}
\put(176,18){\vector(-2,-1){28}}
\put(144,2){\vector(2,1){28}}
\put(176,2){\vector(-2,1){28}}

\put(160,-10){\line(0,1){40}}
\put(160,-10){\line(1,0){30}}
\put(190,-10){\line(0,1){40}}
\put(160,30){\line(1,0){30}}
\put(172,15){\unboldmath$\bullet$}
\put(172,0){\unboldmath$\bullet$}
\put(171,-17){\unboldmath{\scriptsize $P_5$}}}
\end{picture}
\end{center}
\end{multicols}
\end{enumerate}
and having the following properties:
\begin{enumerate}[$(a)$]
\item The two automorphisms $g_1$, $g_2$ are conjugate by $\gamma_2\in \Aut(X)$ and satisfy $\rk(\Pic(X)^{g_i})=2$ for $i=1,2$. 

\item Both $g_1$ and $g_2$ preserve the two real conic bundles of the pair $P_1$. The action on one is trivial on the basis, but non-trivial on the other one.

\item The fixed points of $g_i$ on $X(\C)$ consists of two isolated real points, and one smooth rational curve having no real point.
\item
The action of $g_1,g_2$ on $\Pp^1_\C\times\Pp^1_\C$, via the blow-up $X\to S$ and the isomorphism $\varphi: S_\C \to  \Pp^1_\C\times\Pp^1_\C$, are respectively given by
\[\begin{array}{llll}
\vphantom{\Big)}(s,v)&\dasharrow &\left(
\frac{s(\mu s v-(1+\mu) v+\mu)}{\mu (-s v+(1+\mu)s-1)},  \frac{\mu v(-s v+(1+\mu)s-1)}{\mu s v-(1+\mu) v+\mu}\right)\\
\vphantom{\Big)}
(s,v)&\dasharrow &
\left(\frac{-s v+(1+\mu)s-1}{s (\mu s v-(1+\mu) v+\mu)}, \frac{\mu s v-(1+\mu) v+\mu}{v (-s v+(1+\mu)s-1)})\right)\end{array}\]
on the chart $\{(1:s),(1:v)\mid (s,v)\in \A^2_\C\}$.
\end{enumerate}
\end{lemma}
\begin{proof}
The existence can be checked by using Proposition~\ref{im} and the description of $\mathbb{F}_\R$. Using the action on the conic bundles to compute the matrices of $g_1$, $g_2$ with respect to the basis $\{f,\bar f,E_p,E_{\bar p},E_q,E_{\bar q}\}$, we respectively get
\[\left(\begin{smallmatrix}
2 & 1 & 1 & 1& 1& 1\\
1 & 2 & 1& 1& 1& 1\\
-1& -1& 0 & -1 & -1 & -1\\
-1& -1& -1 & 0 & -1 & -1\\
-1& -1& -1 & -1 & -1 & 0\\
-1& -1& -1 & -1 & 0 & -1
\end{smallmatrix}
\right)\text{\ and\ } \left(\begin{smallmatrix}
2 & 1 & 1 & 1& 1& 1\\
1 & 2 & 1& 1& 1& 1\\
-1& -1& -1 & 0 & -1 & -1\\
-1& -1& 0 & -1 & -1 & -1\\
-1& -1& -1 & -1 & 0 & -1\\
-1& -1& -1 & -1 & -1 & 0
\end{smallmatrix}
\right).\]
Using the fact that the points $p,\bar p ,q,\bar q$ on $\Pp^1_\C\times\Pp^1_\C$ are respectively $(1:0)(0:1)$, $(0:1)(1:0)$, $(1:1)(1:\mu)$, $(1:\bar{\mu})(1:1)$ and the above matrices, we obtain the explicit description of the birational maps of $\Pp^1_\C\times\Pp^1_\C$, given in $(d)$. Assertion $(a)$  follows from the description of $g_1$, $g_2$; it remains to show $(b)$, $(c)$. 
The singular fibres of the two conic bundles of the pair $P_1$ are given in Figure \ref{fig:singularfibres}, together with the action of $g_1$, which follows from the description of the matrix above.
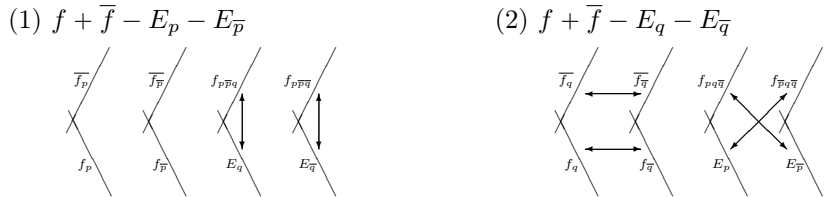
\begin{figure}[h]
\centering
\noindent\begin{minipage}{\textwidth}
\noindent\begin{minipage}{0.5\textwidth}
(1) $f+\overline{f}-E_p-E_{\overline{p}}$
\begin{center}
\begin{picture}(140,50)
\scalebox{0.7}{
\put(25,-5){\line(-1,2){23}}
\put(7,10){\unboldmath{\scriptsize $f_p$}}
\put(24,75){\line(-1,-2){23}}
\put(5,55){\unboldmath{\scriptsize $\overline{f_p}$}}
\put(65,-5){\line(-1,2){23}}
\put(47,10){\unboldmath{\scriptsize $f_{\overline{p}}$}}
\put(65,75){\line(-1,-2){23}}
\put(45,55){\unboldmath{\scriptsize $\overline{f_{\overline{p}}}$}}
\put(105,-5){\line(-1,2){23}}
\put(86,10){\unboldmath{\scriptsize $E_q$}}
\put(105,75){\line(-1,-2){23}}
\put(78,55){\unboldmath{\scriptsize $f_{p\overline{p}q}$}}
\put(145,-5){\line(-1,2){23}}
\put(125,10){\unboldmath{\scriptsize $E_{\overline{q}}$}}
\put(145,75){\line(-1,-2){23}}
\put(117,55){\unboldmath{\scriptsize $f_{p\overline{p} \overline{q}}$}}
\put(95,20){\vector(0,1){30}}
\put(95,50){\vector(0,-1){30}}
\put(136,20){\vector(0,1){30}}
\put(136,50){\vector(0,-1){30}}}
\end{picture}
\end{center}
\end{minipage}
\noindent\begin{minipage}{0.5\textwidth}
(2) $f+\overline{f}-E_q-E_{\overline{q}}$
\begin{center}
\begin{picture}(140,50)
\scalebox{0.7}{
\put(25,-5){\line(-1,2){23}}
\put(7,10){\unboldmath{\scriptsize $f_q$}}
\put(25,75){\line(-1,-2){23}}
\put(4,55){\unboldmath{\scriptsize $\overline{f_q}$}}
\put(65,-5){\line(-1,2){23}}
\put(47,10){\unboldmath{\scriptsize $f_{\overline{q}}$}}
\put(65,75){\line(-1,-2){23}}
\put(44,55){\unboldmath{\scriptsize $\overline{f_{\overline{q}}}$}}
\put(105,-5){\line(-1,2){23}}
\put(86,10){\unboldmath{\scriptsize $E_p$}}
\put(105,75){\line(-1,-2){23}}
\put(78,55){\unboldmath{\scriptsize $f_{pq\overline{q}}$}}
\put(145,-5){\line(-1,2){23}}
\put(125,10){\unboldmath{\scriptsize $E_{\overline{p}}$}}
\put(145,75){\line(-1,-2){23}}
\put(117,55){\unboldmath{\scriptsize $f_{\overline{p}q\overline{q}}$}}
\put(96,20){\vector(1,1){30}}
\put(96,50){\vector(1,-1){30}}
\put(126,20){\vector(-1,1){30}}
\put(126,50){\vector(-1,-1){30}}
\put(18,20){\vector(1,0){30}}
\put(48,20){\vector(-1,0){30}}
\put(48,50){\vector(-1,0){30}}
\put(18,50){\vector(1,0){30}}}
\end{picture}
\end{center}
\end{minipage}
\caption{Singular fibres of the two conic bundles, together with the action of $g_1$.}
\label{fig:singularfibres}
\end{minipage}
\end{figure}
This shows that the action on the basis is trivial in the first case and not trivial in the second. The fixed points are then contained in the two fibres of the second fibration that are fixed, and which are then two smooth rational curves. Looking at the first fibration, we obtain two fixed points in each smooth fibre, three points in the first two singular fibres and one in the last two. The only real points in these fibres are $f_p\cap \overline{f_p}$ and $f_{\bar{p}}\cap \overline{f_{\bar{p}}}$, so we obtain on $X(\C)$ exactly two isolated real points and one smooth rational curve with no real point.\end{proof}

\begin{lemma}\label{Lem:DP4Links}
Let $g\in\Aut(X)$ of prime order that preserves a real conic bundle structure and such that $\rk(\Pic(X)^g)=2$, in particular, $g$ preserves the pair $P_1$. Then, one of the following occurs:
\begin{enumerate}[$(1)$]
\item there is $h\in\Cc(g)\subset\Aut(X)$, the centraliser of $g$, whose action on $P_1$ is the exchange of the two conic bundle structures. In other words, the following diagram commutes
\[\xymatrix@R-0.8pc{
& \scalebox{0.7}{$g\acts $}X\ \ar[dl]_{\zeta_1}\ \ar[r]^h \ar[d]_{\pi_1} &\scalebox{0.7}{$g\acts $}X\ \  \ar[d]_{\pi_2}\ar[dr]^{\zeta_2} & & \\
S\ar[r]^\pi &\Pp^1  \ar[r]^{\simeq}& \Pp^1 & S\ar[l]_\pi}\]
where $\zeta_1$, $\zeta_2$ are the blow-up of four points on $S_\C$ and $\pi_1$, $\pi_2$ are the morphisms corresponding to the conic bundle structures for $f+\bar f-E_p-E_{\bar p}$ and $f+\bar f-E_q-E_{\bar q}$, respectively.
\item The map $g$ is equal to $g_1$ or $g_2$ given in Lemma~$\ref{Lem:g1g2DP4}$.
\end{enumerate}
\end{lemma}
\begin{proof}
Non trivial automorphisms in $\F_\R$ preserving the first pair $P_1$ are $\gamma_1$, $\gamma$, and $\gamma_1\gamma$. In this case, we are in $(1)$ and can choose $h=\gamma_2$. 

When $g\notin\F_\R$, then $g$ exchanges $P_2$ and $P_3$. This plus the fact that $g$ has prime order implies that $g$ has order $2$. On the other hand, the action of $g$ on the pairs $P_4$ and $P_5$ cannot be like \begin{picture}(50,20)
\scalebox{0.6}{
\put(10,-10){\line(0,1){40}}
\put(10,-10){\line(1,0){30}}
\put(40,-10){\line(0,1){40}}
\put(10,30){\line(1,0){30}}
\put(22,15){\unboldmath$\bullet$}
\put(22,0){\unboldmath$\bullet$}
\put(21,-17){\unboldmath{\scriptsize $P_4$}}

\put(24,18){\vector(1,0){28}}
\put(56,18){\vector(-1,0){28}}
\put(24,3){\vector(1,0){28}}
\put(56,3){\vector(-1,0){28}}

\put(40,-10){\line(0,1){40}}
\put(40,-10){\line(1,0){30}}
\put(70,-10){\line(0,1){40}}
\put(40,30){\line(1,0){30}}
\put(52,15){\unboldmath$\bullet$}
\put(52,0){\unboldmath$\bullet$}
\put(51,-17){\unboldmath{\scriptsize $P_5$}}}
\end{picture}\vspace{0.3cm}, since this would imply that $\rk(\Pic(X)^g)>2$ since in this case, $g$ also fixes $f+\bar f$.
Then, the action of $g$ on the conic bundles is one of the two given in Lemma~\ref{Lem:g1g2DP4}.
\end{proof}

\subsection{Case: $(K_X)^2=2$.}\label{dP2}

The birational morphism $\zeta\colon X\to S$ is the blow-up of 3 pairs of conjugate points, say $p,\bar p, q, \bar q, r, \bar r\in S$. 
Since X is a Del Pezzo surface of degree two, the linear system of the anticanonical divisor defines a double covering $|-K_X|\colon X\rightarrow \Pp^2$ ramified over a quartic $\Gamma$. From the fact that $X(\R)\simeq S(\R)$, $\Gamma $ is a real smooth quartic with one oval. We see $X$ as $w^2=F(x,y,z)$ in $\Pp(2,1,1,1)$ and $\Gamma$ the zero set of $F$.

\begin{prop}\label{Prop:Geiser}
There exists an exact sequence
\begin{equation*}
\xymatrix{
1 \ar[r] & \langle\nu\rangle \ar[r] & \Aut(X) \ar[r] & \Aut(\Gamma)\ar[r] &1
}
\end{equation*}
where $\nu$ represents the Geiser involution which exchanges the two points of any fibre i.e. the involution given by $(w,x,y,z)\mapsto(-w,x,y,z)$.
\end{prop}
\begin{proof}
We have the following exact sequence 
\begin{equation}\label{seq2}
\xymatrix{
1 \ar[r] & \langle\nu\rangle \ar[r] & \Aut(X) \ar[r] & \Aut(\Pp^2,\Gamma)\ar[r] &1
}
\end{equation}
\noindent where $\Aut(\Pp^2,\Gamma)$ denotes the automorphisms of $\Pp^2$ which preserves the quartic and is isomorphic to $\Aut(\Gamma)$ because the restrictions gives a map from $\Aut(\Pp^2,\Gamma)$ to $\Aut(\Gamma)$ which is injective since the only automorphism that preserves the quartic pointwise is the identity (an automorphism of $\Pp^2$ can only fixed 3 points or a point and a line but not a quartic). To see that the restriction map is surjective, we compute the canonical divisor of the quartic by adjunction formula getting that $K_\Gamma=(K_{\Pp^2}+\Gamma)|_\Gamma=(-3L+4L)|_\Gamma=L|_\Gamma$. Hence, every automorphism of $\Gamma$ extends to $\Pp^2$.
\end{proof}

\begin{lemma}\label{numinimal}
\begin{enumerate}[$($a$)$]
\item Let $C$ be a $(-1)$-curve in $X$, then the $(-1)$-curve $\nu(C)$ is equal to $\nu(C)=-K_X-C$.
\item $\rk(\Pic(X)^\nu)=1$. In particular, the pair $(X,\langle\nu\rangle)$ is minimal.
\end{enumerate}
\end{lemma}
\begin{proof}
\begin{enumerate}[(a)]
\item We call $\varepsilon$ the map defined by $|-K_X|$. Then, $\varepsilon(C)$ is a curve of degree $d$ for some $d$. If we call $D=\varepsilon^*(\varepsilon(C))$, we have that $D= d(-K_X)\neq C$. This implies that $D=C+C'=d(-K_X)$ for $C'$ a $(-1)$-curve, $C'=\nu(C)$. Intersecting $D$ with $-K_X$ we have $2=2d$ and hence $d=1$. Then $\nu(C)=C'=-K_X-C$.
\item Let $C$ be a $(-1)$-curve in $X$, then by item (a) we have $C\cdot\nu(C)=C(-K_X-C)=2$. Moreover, the fact that $\Pic(X_\C)$ is generated by the divisors in the set $A:=\{-K_X,E_p,E_{\bar p},E_q,E_{\bar q},E_r,E_{\bar r}\}$ then, for any divisor $D\in\Pic(X_\C)$, $D=\sum a_iC_i$ with $a_i\in\Z$ and $C_i\in A$. We have $D+\nu(D)=D+a_i\sum\nu(C_i)=a_i(\sum -K_X-C_i)=m(-K_X)$ for some $m\in\Z$.\qedhere
\end{enumerate}
\end{proof}
\begin{lemma}\label{dP2nonrank1}
Let $g\in\Aut(X)$ of prime order and $g\neq\nu$. Then $\rk(\Pic(X)^g)\neq1$.
\end{lemma}
\begin{proof}
Let $g\in\Aut(X)$. Since a basis of $\Pic(X_\C)\cong\Z^8$ is $\{f,\bar f, E_p,E_{\bar p},E_q,E_{\bar q},E_r,E_{\bar r}\}$, we get that the action of $g$ on $\Pic(X)=\Pic(X_\C)^{\sigma}$ is an element in $\GL(4,\Z)\subset\GL(4,\C)$ and is diagonalisable in $\GL(4,\C)$ for $g\in\Aut(X)$.
If $g$ is an involution in $\Aut(X)$ with $\rk(\Pic(X)^g)=1$, the only possibility for the action of $g$ on $\Pic(X)^g$ in $\GL(4,\C)$ is given by $\left(\begin{smallmatrix}
1 & & & \\
& -1 & & \\
& & -1& \\
& & & -1
\end{smallmatrix}\right)$ assuming that the first entry $1$ corresponds to the anticanonical divisor for some basis containing it. On the other hand, since every element $g$ in $\Aut(X)$ commutes with $\nu$, then in the same basis, $g$ and $\nu$ are conjugate to a diagonal action as the element presented above. This implies that $g$ and $\nu$ are the same.

Let $g\in\Aut(X)$ be of prime order $p\geq3$. We obtain then an element of $\GL(4,\Z)$ of order $p$ which fixes $K_X$. Then, the characteristic polynomial $Q\in\Z[x]$ vanishes at $1$ and all other roots in $\C$ are roots of the polynomial $x^{p-1}+\cdots+1$, irreducible over $\Q$. Hence, $Q$ is a multiple of $(x-1)(x^{p-1}+\cdots+1)=x^p-1$. This implies that $p\leq4$, so $p=3$ and then $Q=(x-1)^2(x^2+x+1)$. Therefore $\Pic(X)^g\cong\Z^2$.
\end{proof}

\section{Conic bundle case}\label{Ch:ConicCase}
In this section, we describe the elements in $\Aut(S(\R))$ of prime order corresponding to the second case of Proposition \ref{MinMod}, i.e.\ that belong to the group $\Aut(S(\R),\pi)$. Let us recall the following notation:
\begin{align*}
\Bir(S,\pi)=&\{g \in \Bir(S)\ |\ \exists\alpha\in\Aut(\Pp^1) \text{\ such that } \alpha\pi=\pi g\},\\
\Aut(S(\R),\pi)=&\{g \in \Aut(S(\R))\ |\ \exists\alpha\in\Aut(\Pp^1) \text{\ such that } \alpha\pi=\pi g\},
\end{align*}
and that $\Phi\colon\Bir(S,\pi)\to \Aut(\Pp^1)$ is the corresponding group homomorphism (see the exact sequence $(\ref{seqCB})$) whose kernel is denoted by $\Bir(S/\pi)$ and by $\Aut(S(\R)/\pi)$ for the corresponding group homomorphism $\Aut(S(\R),\pi)\to\Aut(\Pp^1)$.
\subsection{Image of the action on the basis}
Recall that $\pi\colon S\dasharrow \Pp^1$ is the map given  by $\pi(w:x:y:z)=(w:z)$. Hence, the natural coordinates on $\Pp^1$ are $(w:z)$ or simply $(1:z)$ for affine coordinates. With the choice of these coordinates, the group $\Aut(\Pp^1)$ is naturally isomorphic to $\PGL(2,\R)$:  an element 
$\left[\begin{smallmatrix}
a & b\\
c & d
\end{smallmatrix}
\right]\in \PGL(2,\R)$ acts as \[z\mapsto \frac{az+b}{cz+d}\ \ \mbox{ or }\ \ (w:z)\mapsto (cz+dw:az+bw).\]

In the following two lemmas, the image of the map $\Phi\colon\Bir(S,\pi)\to \Aut(\Pp^1)$ in the sequence (\ref{seqCB}) is presented and the image of elements of finite order is characterised.
\begin{lemma}\label{imphigrande}
The image of\/ $\Phi\colon\Bir(S,\pi)\to\Aut(\Pp^1)$ is the same as the image of its restriction to $\Aut(S(\R),\pi)$.

The corresponding subgroup of $\Aut(\Pp^1)$ is given by the following semidirect product, where the generator of \ $\Z/2\Z$ is the automorphism $\eta\colon z\mapsto -z$.
\begin{equation}\label{imphi}
\Phi(\Bir(S,\pi))=\Phi(\Aut(S(\R),\pi))=\left\{\left[\begin{smallmatrix}
1 & b\\
b & 1
\end{smallmatrix}
\right];b\in(-1,1)\subset\R\right\}\rtimes\Z/2\Z
\end{equation}
\end{lemma}
\begin{proof}
Since the sphere $S(\R)$ is preserved by elements in $\Bir(S,\pi)$ (respectively in $\Aut(S(\R),\pi)$) and is mapped subjectively to the interval $[-1,1]\subset\R$ on the basis of the fibration. This interval is then invariant on the basis and the group $\Phi(\Bir(S,\pi))$ is contained in the group generated by $z\mapsto \frac{z+b}{bz+1}$, $b\in(-1,1)\subset\R$ and $z\mapsto -z$ because those are exactly the automorphisms of $\Pp^1$ which fix or interchanged the points $-1$ and $1$. On the other hand, for each $b\in (-1,1)\subset \R$ the map $g_b\!:\!(x,y,z)\mapsto\left(x\frac{\sqrt{1-b^2}}{bz+1},y\frac{\sqrt{1-b^2}}{bz+1},\frac{z+b}{bz+1}\right)$ belongs to $\Aut(S(\R),\pi)$ and is sent to $\left[\begin{smallmatrix}
1 & b\\
b & 1
\end{smallmatrix}
\right]$ and the map $\tilde\eta\colon(x,y,z)\mapsto(x,y,-z)$ is sent to $\left[\begin{smallmatrix}
-1 & 0\\
0 & 1
\end{smallmatrix}
\right]$, corresponding to $z\mapsto -z$, which proves Equality (\ref{imphi}).
\end{proof}

\begin{lemma}\label{ImFinOrd}
Let {\small $g\in\Aut(S(\R),\pi)$} be of finite order. After conjugation in {\small $\Aut(S(\R),\pi)$}, the map $\Phi(g)$ is the identity or equal to $\left[\begin{smallmatrix}
1 & \ 0\\
0 & -1
\end{smallmatrix}
\right]$.
\end{lemma}
\begin{proof}
Elements of the form $\left[\begin{smallmatrix}
1 & b\\
b & 1
\end{smallmatrix}
\right]$ with $b\in(-1,1)\setminus\{0\}$ are not of finite order; indeed the eigenvalues of $\left[\begin{smallmatrix}
1 & b\\
b & 1
\end{smallmatrix}
\right]$ are $1\pm b$, so the element $\left[\begin{smallmatrix}
1 & b\\
b & 1
\end{smallmatrix}
\right]$ is conjugate to $\left[\begin{smallmatrix}
\frac{1+b}{1-b} & 0\\
0 & 1
\end{smallmatrix}
\right]$ in $\PGL(2,\R)$ and $\frac{1+b}{1-b}\in\R^*$ has infinite order because $\frac{1+b}{1-b}\neq-1$. Moreover, $\left[\begin{smallmatrix}
1 & -b\\
b & -1
\end{smallmatrix}
\right]$ is conjugate to $\left[\begin{smallmatrix}
1 & 0\\
0 & -1
\end{smallmatrix}
\right]$ by the matrix $\left[\begin{smallmatrix}
1 & c\\
c & 1
\end{smallmatrix} 
\right]$ with $c=\frac{1\pm\sqrt{1-b^2}}{b}.$
\end{proof}
\subsection{Algebraic description of $\Bir(S/\pi)$}\label{Sec:AlgDescription}

Extending the scalars from $\R$ to $\C$, the general fibre of $\pi\colon S_\C\rightarrow\C$, $(x,y,z)\mapsto z$ is rational. The group of birational maps of $S_\C$ preserving any general fibre of $\pi$ is then equal to $\PGL(2,\C(z))$. The group $\Bir(S/\pi)$ can thus be viewed as a subgroup of $\PGL(2,\C(z))$.

\begin{definition}\label{bar}\
\begin{enumerate}[$(i)$]
\item For each $A\in\GL(2,\C(z))$, we define $\bar A\in \GL(2,\C(z))$, as the matrix obtained by replacing every coefficient of every entry of $A$ by its conjugate. 

\item In the same way, we define $\bar A$ for any element in $\PGL(2,\C(z))$ and we observe that $\bar A$ does not depend on the representative because if $A_1,A_2\in\PGL(2,\C(z))$ are in the class of the element $A$ then $A_1=\lambda A_2$ for some $\lambda\in\C(z)^*$ and then $\bar A_1=\bar\lambda\bar A_2$ implying that $\bar{A_1}$ and $\bar{A_2}$ are both in the class of $\bar A$.
\end{enumerate}
\end{definition}

\begin{lemma}\label{LemAction}\ 
\begin{enumerate}[$(a)$]
\item The complex surface $S_\C$ is birational to $\A^2_\C$ via $\psi\colon (x,y,z)\dasharrow (x-{\bf i} y,z)$.
\item The group $\PGL(2,\C(z))$ acts on $\A^2_\C$ via
\begin{equation}\label{actionPGL}
\begin{array}{ccccc}
\ \PGL(2,\C(z))&\times& \A_\C^2& \dasharrow& \A_\C^2\\
\Big(\left[\begin{smallmatrix}\alpha(z) & \beta(z) \\ \gamma(z) & \delta(z)\end{smallmatrix}\right]&,&(t,z)\Big)&\dasharrow& \left(\frac{\alpha(z)t+\beta(z)}{\gamma(z)t+\delta(z)},z\right)\end{array}
\end{equation}
and thus also acts on $S_\C$ via the conjugation by $\psi^{-1}$.
\item 
For any $A\in\PGL(2,\C(z))$, the corresponding action of $A$ and $\tau \bar A\tau$ on $S_\C$, via $\psi$ and denoted by $\mathcal A$ and $\mathcal{\tau\bar A\tau}$ respectively, are conjugate by the anti-holomorphic involution $\sigma$ $($i.e. $\sigma\colon (x,y,z)\mapsto (\bar x, \bar y, \bar z)$ $)$, where $\tau:=\left[\begin{smallmatrix}
0 & 1-z^2\\
1 & 0
\end{smallmatrix}
\right]\in  \PGL(2,\C(z)),$ which means that the following diagram commutes
\[\xymatrix@R-0.8pc{
S_\C\  \ar[d]_\sigma \ar@{-->}[r]^{\mathcal A} & \ S_\C \ar[d]^\sigma  \\
S_\C\ \ar@{-->}[r]^{\mathcal{\tau\bar A\tau}} & \ S_\C .
}\]
In particular, the group $\Bir(S/\pi)$ corresponds, via the action of $\ \PGL(2,\C(z))$ on $S_\C$, to the group
\[\mathscr G:=\{A\in  \PGL(2,\C(z))\ |\ \tau A\tau= \bar{A}\}\]
\end{enumerate}
\end{lemma}

\begin{proof}
\begin{enumerate}[$(a)$]
\item The map $\psi$ is a rational map and its inverse is given by
\[\psi^{-1}\colon (t,z)\dasharrow \left(\frac{t^2-z^2+1}{2t},{\bf i} \cdot \frac{t^2+z^2-1}{2t},z\right).\]
\item 
Clearly, the identity in $\PGL(2,\C(z))$ gives the identity map of $\A^2_\C$. Let $A=\left[\begin{smallmatrix}
\alpha(z) & \beta(z)\\
\gamma(z) & \delta(z)
\end{smallmatrix}
\right]$ and $A'=\left[\begin{smallmatrix}
\alpha'(z) & \beta'(z)\\
\gamma'(z) & \delta'(z)
\end{smallmatrix}
\right]$ be elements in $\PGL(2,\C(z))$. We compute \[(A,A'(t,z))\mapsto\left(\frac{\alpha\left(\frac{\alpha't+\beta'}{\gamma't+\delta'}\right)+\beta}{\gamma\left(\frac{\alpha't+\beta'}{\gamma't+\delta'}\right)+\delta},z\right)=\left(\frac{(\alpha\alpha'+\beta\gamma')t+\alpha\beta'+\beta\delta'}{(\gamma\alpha'+\delta\gamma')t+\gamma\beta'+\delta\delta'},z\right),\] which is the same as \[(AA',(t,z))\mapsto\left(\frac{(\alpha\alpha'+\beta\gamma')t+\alpha\beta'+\beta\delta'}{(\gamma\alpha'+\delta\gamma')t+\gamma\beta'+\delta\delta'},z\right).\]

The action of $\PGL(2,\C(z))$ on $\A^2_\C$ gives an action on $S_\C$ in the following way: for any element $A=\left[\begin{smallmatrix}
\alpha(z) & \beta(z)\\
\gamma(z) & \delta(z)
\end{smallmatrix}
\right]\in\PGL(2,\C(z))$ we denote by $A\acts\A^2_\C$ the action of $A$ on $\A^2_\C$ given by the map $(t,z)\dasharrow\left(\frac{\alpha(z)t+\beta(z)}{\gamma(z)t+\delta(z)},z\right)$, thus the following diagram gives the action on $S_\C$ that we denote by $\psi^{-1} A\psi$ or simply $\mathcal A$ if no confusion:
\[\xymatrix@R-0.5pc{
S_\C\  \ar@{-->}[d]_{\mathcal A} \ar@{-->}[r]^{\psi} & \ \A_\C^2 \ar@{-->}[d]^{A\scalebox{0.7}{$\acts\A^2_\C$} }  \\
S_\C\  & \ \A^2_\C \ar@{-->}[l]_{\psi^{-1}}
}\]
\item We name $\sigma_1\colon (t,z)\mapsto (\bar t,  \bar z)$ the anti-holomorphic involution on $\A^2_\C$, then via the birational map $\psi$ we have 
\[\psi \sigma \psi^{-1}=\sigma_1\tau=\tau\sigma_1\colon (t,z)\dasharrow \left(\frac{1-\bar{z}^2}{\bar{t}},\bar{z}\right).\]
Let $A\in\PGL(2,\C(z))$. We want to show that $\mathcal{\tau\bar A\tau(\sigma}(x,y,z))=\mathcal{\sigma (A}(x,y,z))$ for any $(x,y,z)\in S_\C$ which is the same as showing $\psi^{-1}(\tau\bar A\tau)(\psi\sigma(x,y,z))=\sigma (\psi^{-1}A(\psi(x,y,z)))$ for any $(x,y,z)\in S_\C$, where the action of $A$ and $\tau\bar A\tau$ are now on $\A^2_\C$. Notice that according to Definition \ref{bar}$(ii)$, the action of $\bar A$ on $\A^2_\C$ is the same as the action of $\sigma_1 A\sigma_1$ and in this way, for any $(x,y,z)\in S_\C$ we have
\begin{align*}
\psi^{-1}(\tau\bar A\tau)(\psi\sigma(x,y,z))&=\psi^{-1}(\tau\sigma_1A\sigma_1\tau)(\psi\sigma(x,y,z))\\
&=\psi^{-1}((\psi\sigma\psi^{-1})A(\psi\sigma\psi^{-1}))(\psi\sigma(x,y,z))\\
&=\sigma\psi^{-1}A(\psi\sigma(\sigma(x,y,z)))=\sigma(\psi^{-1}A(\psi(x,y,z))).
\end{align*}

The elements in $\Bir(S/\pi)$ correspond to the elements in $\PGL(2,\C(z))$ which commute with $\psi\sigma\psi^{-1}$, in other words, for $A\in\PGL(2,\C(z))$ we have that $A$ belongs to $\psi^{-1}\Bir(S/\pi)\psi$ if $\tau\sigma_1A\sigma_1\tau=A$ which is equivalent to $\tau \bar A\tau=A$ and hence we get the description of the group $\mathscr G=\psi^{-1}\Bir(S/\pi)\psi$.\qedhere
\end{enumerate}
\end{proof}
\begin{remark}\label{RemTau}The element $\tau=\left[\begin{smallmatrix}
0 & 1-z^2\\
1 & 0
\end{smallmatrix}
\right]\in \PGL(2,\C(z))$ belongs to $\mathscr G$ and corresponds to the element of $\Bir(S/\pi)$ given by 
$(x,y,z)\mapsto (x,-y,z),$ which is a reflection that belongs then to $\Aut(S)\subset \Aut(S(\R))$.
\end{remark}

The group $\mathscr G\subset \PGL(2,\C(z))$ defined in Lemma~\ref{LemAction} is the algebraic version of $\Bir(S/\pi)$, that we will study in the sequel. In the following lemma, we give a more precise description of elements of this group.

\begin{lemma}\label{G}
Each $A\in\mathscr G\subset \PGL(2,\C(z))$ is equal to $\left[\begin{smallmatrix}
a(z) & b(z)h\\
\bar b(z) & \bar a(z) 
\end{smallmatrix}
\right]$ for some polynomials $a,b\in\C[z]$ with no common real roots, $h=1-z^2$. Moreover, the corresponding matrix $\left[\begin{smallmatrix}
a(z) & b(z)h\\
\bar b(z) & \bar a(z) 
\end{smallmatrix}
\right]\in\mathrm{GL}(2,\C(z))$ has a determinant $a(z) \bar a(z)-b(z)\bar b(z) h \in \R[z]$ which is positive when $z^2>1$.
\end{lemma}
\begin{remark}
Conversely, if $A=\left[\begin{smallmatrix}
a(z) & b(z)h\\
\bar b(z) & \bar a(z) 
\end{smallmatrix}
\right]\in \PGL(2,\C(z))$ for some $a,b\in \C(z)$ (and in particular when $a,b\in \C[z]$), then $A$ belongs to $\mathscr G$, since $\tau A\tau=\overline{A}$.
\end{remark}
\begin{proof}
Let $A=\left[\begin{smallmatrix}
a(z) & b(z)\\
c(z) & d(z)
\end{smallmatrix}
\right]\in\mathscr G$. The equality $\tau A\tau =\bar A$ gives 
\begin{align*}
\left[\begin{smallmatrix}
\bar a(z) & \bar b(z)\\
\bar c(z) & \bar d(z)
\end{smallmatrix}
\right]=\left[\begin{smallmatrix}
0 & 1-z^2\\
1 & 0
\end{smallmatrix}
\right]\left[\begin{smallmatrix}
a(z) & b(z)\\
c(z) & d(z)
\end{smallmatrix}
\right]\left[\begin{smallmatrix}
0 & 1-z^2\\
1 & 0
\end{smallmatrix}
\right].
\end{align*}
Hence $b(z)=\lambda\bar c(z)$, $d(z)(1-z^2)=\lambda\bar a(z)$, $c(z)(1-z^2)^2=\lambda\bar b(z)$, and $a(z)(1-z^2)=\lambda\bar d(z)$ for some $\lambda\in\C(z)^*$. From first and third equation we get that $c((1-z^2)^2-\lambda\bar\lambda)=0$ and from second and fourth equation we get that $\bar a((1-z^2)^2-\lambda\bar\lambda)=0$. In both cases, $\lambda\bar\lambda=(1-z^2)^2$ which is equivalent to $\frac{\lambda}{(1-z^2)}\cdot\left(\overline{\frac{\lambda}{(1-z^2)}}\right)=1$, then by Hilbert's Theorem 90 there is $\mu\in\C(z)^*$ such that $\lambda=\frac{\mu}{\bar\mu}(1-z^2)$ and $A=\left[\begin{smallmatrix}
a(z)\bar\mu & \mu \bar c(z)(1-z^2)\\
c(z)\bar\mu & \bar a(z)\mu
\end{smallmatrix}
\right]$. Calling again $a(z)\colon=a(z)\bar\mu(z)$ and $b(z)\colon=\mu(z)\bar c(z)$ we get
$A=\left[\begin{smallmatrix}
a(z) & b(z)h\\
\bar b(z) & \bar a(z)
\end{smallmatrix}
\right]$.

When $a=\frac{p}{q}$, $b=\frac{r}{s}$ with $p,q,r,s\in\C[z]$, we can multiply $A$ by $q\bar q s\bar s$ and we obtain an element in the same class with entries in $\C[z]$. Now, if $z_0$ is a common real root of $a$ and $b$ thus $z_0$ is also a real root of $\bar a$ and $\bar b$ which means that we may divide by $z-z_0$ all entries of $A$ and remain in the same class. Then $A$ is of the desired form. The determinant of the corresponding element of $\mathrm{GL}(2,\C(z))$ is then $a\bar a-b\bar b(1-z^2)=a\bar a+b\bar b(z^2-1)\in\R[z]$. Notice that for $z^2>1$, $a\bar a+b\bar b(z^2-1)>0$ because $a\bar a\geq 0$, $b\bar b\geq 0$ implies $a\bar a+b\bar b(z^2-1)\geq0$ and the fact $a$ and $b$ have non common real roots implies that the inequality is strict.
\end{proof}
\begin{remark}In the sequel, we will always denote by $h$ the polynomial $1-z^2\in \mathbb{R}[z].$\end{remark}
Now, we would like to characterise elements in $\Aut(S(\R)/\pi)$ and $\Aut^+(S(\R)/\pi)$ inside the group $\mathscr G=\psi^{-1}\Bir(S/\pi)\psi$. In order to do this, we need to understand the birational map $\psi\colon S_\C\dasharrow \A^2_\C$  given by $(x,y,z)\dasharrow(x-{\bf i} y,z)$. The following result describes the extension of the map, that we again denote by $\psi$.
\begin{lemma}\label{DescriptionPsi}
$\psi$ satisfies:
\begin{enumerate}[$($a$)$]
\item The birational map
\[\begin{array}{rccc}
\psi\colon &S_\C & \dasharrow & \Pp^1_\C\times \Pp^1_\C\\
& (1:x:y:z)& \dasharrow & ((1:x-{\bf i} y),(1:z))\\
&(w:x:y:z) & \dasharrow & ((w:x-{\bf i} y),(w:z))\end{array}\]
has three base-points, namely $q=(0:{\bf i}:1:0)$, $\bar{q}=(0:-{\bf i}:1:0)$, and one point $\omega$, infinitely near $q$.
\item Its inverse is 
\[\begin{array}{rccc}
\psi^{-1}\colon & \Pp^1_\C\times \Pp^1_\C & \dasharrow &S_\C\\[0.5em]
& ((1:t),(1:z))& \dasharrow & \left(1:\frac{t^2-z^2+1}{2t}:{\bf i} \cdot \frac{t^2+z^2-1}{2t}:z\right)\\[0.5em]
&((u:t),(v:z)) & \dasharrow & \splitfrac{(2tuv^2:t^2v^2-z^2u^2+u^2v^2:}{{\bf i}(t^2v^2+z^2u^2-u^2v^2):2tzuv)}\end{array}\]
and has exactly three base-points, namely \[(0:1)(0:1), (1:0)(1:1),\mbox{ and }(1:0)(1:-1).\]
\item The map $\psi$ can be decomposed as the blow-up of $q$, $\bar{q}$, $\omega$, followed by the contraction of the strict transforms of the curves $L$, $M$, $D\subset S_{\C}$ given respectively by
\[\begin{array}{lll}
L\colon& x={\bf i}y,w=-z\\
M\colon& x={\bf i}y,w=z\\
D\colon& w=0\end{array}\]
\end{enumerate}
This can be described by the diagram in Figure $\ref{mappsi}$, where $P_N=(1:0:0:1)$, $P_S=(1:0:0:-1)\in S(\R)$ are the north and south poles, where $\overline{L}$, $\overline{M}$ are the image of $L$, $M$ by the anti-holomorphic involution and where the strict transforms of the curves are again denoted by the same names.

\begin{figure}[h]
\begin{center}
\begin{picture}(224,137.6)
\scalebox{0.8}{

\put(42,-9){\unboldmath{\scriptsize $S_\C$}}

\put(20,50){\line(1,0){52}}
\put(61.5,47.5){\unboldmath$\bullet$}

\put(42.5,53){\unboldmath{\tiny $\overline L$}}
\put(66,45){\unboldmath{\tiny $\overline q$}}
\put(29,45){\unboldmath{\tiny $P_S $}}
\put(20,20){\line(1,0){52}}

\put(22,15){\line(6,5){48}}
\put(42.5,38){\unboldmath{\tiny $D$}}

\put(25.5,47.5){\unboldmath$\circ$}
\put(61.5,17.5){\unboldmath$\circ$}

\put(42.5,22){\unboldmath{\tiny $M$}}
\put(65,32){\unboldmath{\tiny $\overline M$}}
\put(22,33){\unboldmath{\tiny $L$}}
\put(25.5,17.5){\unboldmath$\bullet$}
\put(29,15){\unboldmath{\tiny $q$}}
\put(65,15){\unboldmath{\tiny $P_N$}}
\put(28,60){\line(0,-1){50}}
\put(64,60){\line(0,-1){50}}

\put(70,83){\vector(-1,-1){20}}
\put(58,68){\unboldmath{\tiny $q,\overline{q}$}}

\put(95,93){\line(-2,1){30}}
\put(72,93){\unboldmath{\tiny $E_q$}}
\put(77,98){\unboldmath$\bullet$}
\put(82,103){\unboldmath{\tiny $\omega$}}
\put(85,93){\line(2,1){30}}
\put(98,95){\unboldmath{\tiny $M$}}

\put(78,97){\line(2,5){18}}
\put(88.5,118){\unboldmath{\tiny $D$}}

\put(67.5,127.5){\unboldmath$\circ$}
\put(71,125){\unboldmath{\tiny $P_S$}}
\put(107.5,102.5){\unboldmath$\circ$}
\put(111,100){\unboldmath{\tiny $P_N$}}

\put(70,100){\line(0,1){35}}
\put(111,115){\unboldmath{\tiny $\overline M$}}
\put(63,115){\unboldmath{\tiny $L$}}
\put(110,101){\line(0,1){35}}

\put(95,143){\line(-2,-1){30}}

\put(73,138){\unboldmath{\tiny $\overline L$}}
\put(85,143){\line(2,-1){30}}
\put(99,138){\unboldmath{\tiny $E_{\overline q}$}}

\put(150,155){\vector(-2,-1){20}}
\put(145,148){\unboldmath{\tiny $\omega$}}

\put(195,133){\line(-2,1){30}}
\put(172,134){\unboldmath{\tiny $E_q$}}
\put(177,138){\line(1,1){20}}
\put(178.5,150){\unboldmath{\tiny $E_\omega$}}
\put(187,163){\unboldmath{\tiny $D$}}
\put(185,133){\line(2,1){30}}
\put(198,135){\unboldmath{\tiny $M$}}

\put(170,140){\line(0,1){35}}
\put(211,155){\unboldmath{\tiny $\overline M$}}
\put(163,155){\unboldmath{\tiny $L$}}
\put(210,140){\line(0,1){35}}

\put(167.5,167.5){\unboldmath$\circ$}
\put(172,167){\unboldmath{\tiny $P_S$}}
\put(207.5,142.5){\unboldmath$\circ$}
\put(211,140){\unboldmath{\tiny $P_N$}}

\put(195,183){\line(-2,-1){30}}

\put(173,178){\unboldmath{\tiny $\overline L$}}

\put(185,183){\line(2,-1){30}}
\put(199,178){\unboldmath{\tiny $E_{\overline q}$}}
\put(198,180){\line(-1,-4){9}}

\put(210,120){\vector(1,-2){25}}
\put(220,100){\unboldmath{\tiny $L,M,D$}}

\put(236,-9){\unboldmath{\scriptsize $\Pp^1_\C\times\Pp^1_\C$}}

\put(220,50){\line(1,0){52}}
\put(228,50){\circle{2}}
\put(228,50){\circle{4}}
\put(245,53){\unboldmath{\tiny $\overline L$}}

\put(229,45){\unboldmath{\tiny $P_S$}}
\put(220,20){\line(1,0){52}}
\put(245,22){\unboldmath{\tiny $\overline M$}}

\put(220,40){\line(1,0){52}}
\put(263,40){\circle{4.5}}
\put(260.5,37.5){\unboldmath$\bullet$}
\put(245,35){\unboldmath{\tiny $E_\omega$}}

\put(228,20){\circle{2}}
\put(228,20){\circle{4}}
\put(229,15){\unboldmath{\tiny $P_N$}}
\put(228,60){\line(0,-1){50}}
\put(219,30){\unboldmath{\tiny $E_q$}}
\put(263,60){\line(0,-1){50}}
\put(263,30){\unboldmath{\tiny $E_{\overline q}$}}

\multiput(90,35)(8,0){14}{\line(1,0){4}}
\put(204,35){\vector(1,0){1}}
\put(150,38){\unboldmath{\footnotesize $\psi$}}}
\end{picture}
\end{center}
\caption{The decomposition of $\psi$ into blow-ups and blow-downs.}
\label{mappsi}
\end{figure}
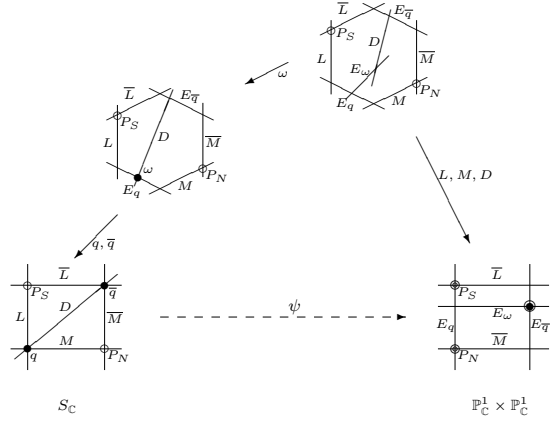
\end{lemma}
\begin{proof}
Parts $(a)$ and $(b)$ follow from a direct calculation. Hence, denoting by $\zeta\colon X\to S_\C$ the blow-up of $q,\bar{q},\omega$, the map $\psi\zeta$ is a birational morphism $X\to  \Pp^1_\C\times \Pp^1_\C $, which is the blow-up of three points since both $S_\C$ and $\Pp^1_\C\times \Pp^1_\C$ have a complex Picard group of rank $2$. Looking at coordinates, one checks that the three curves are $L,M,N$, and the remaining part of the picture can be checked by computing the intersection between the curves.
\end{proof}

Since $M\cup\overline M$ is the fibre of $(1:1)\in\Pp^1$ by $\pi$ and is singular with only real point, every element of $\Bir(S/\pi)$ preserves the north pole  $P_N=M\cap \overline{M}$ and either preserves each of the two curves or interchanges them. This result is proved in the following lemma, that describes moreover algebraically the distinct possible cases.
\begin{lemma}\label{ExchangeOrNotBir}
Let $A=\left[\begin{smallmatrix}
a(z) & b(z)h\\
\bar b(z) & \bar a(z) 
\end{smallmatrix}
\right]\in \mathscr G\subset \PGL(2,\C(z)),$ for some polynomials $a,b\in\C[z]$ with no common real roots $($see Lemma~$\ref{G})$, and let $\mathcal A\in\Bir(S/\pi)$ be the corresponding element $($see Lemma~$\ref{LemAction})$.

The map $\mathcal A$ is defined at the north and south poles $P_N=M\cap \overline M$ and $P_S=L\cap \overline L$. Moreover, the following hold:

\begin{enumerate}[$(1)$]
\item
If $a(1)=0$, then $\mathcal A$ exchanges $M$ with $\overline{M}$.
\item
If $a(1)\neq0$, then $\mathcal A$ preserves both $M$ and $\overline{M}$.
\item
If $a(-1)=0$, then $\mathcal A$ exchanges $L$ with $\overline{L}$.
\item
If $a(-1)\neq0$, then $\mathcal A$ preserves both $L$ and $\overline{L}$.
\end{enumerate}
\begin{remark}
Note that $a(1)\neq0$ (respectively $a(-1)\neq0$) is equivalent to the fact that the determinant $a(z)\overline{a}(z)+b(z)\overline{b}(z)h$ is positive when $z=1$ (respectively $z=-1$).
\end{remark}
\end{lemma}
\begin{proof}
Recall that $A$ acts on $\A^2_\C$ via \[(t,z)\dasharrow \left(\frac{a(z)t+b(z)(1-z^2)}{\overline{b}(z) t+\overline{a}(z)},z\right)\]
(see Lemma~\ref{LemAction}).

Suppose first that $a(1)\neq0$. This implies that the determinant $a(z)\overline{a}(z)+(z^2-1)b(z)\overline{b}(z)$ is not zero (and in fact positive) when $z=1$. Hence, the above birational map is a local isomorphism near the fixed point $(t,z)=(0,1)$, and restricts to an isomorphism of the curve $z=1$. After blowing up $(0,1)$, we obtain thus a local isomorphism in the neighbourhood of the exceptional divisor and of the strict transform of the curve $z=1$. By Lemma~\ref{DescriptionPsi}, these maps correspond to respectively $M$ and $\overline{M}$ via $\psi$. This shows that $\mathcal A$ is defined at $P_N=M\cap \overline{M}$ and preserves each of the two curves $M$ and $\overline{M}$.

If $a(-1)\neq0$, we find similarly that $\mathcal A$ is defined at $P_S=L\cap \overline{L}$ and preserves each of the two curves $L$ and $\overline{L}$.

If $a(1)=0$, we write $a(z)=a_0(z)(1-z)$ for some polynomial $a_0\in \C[z]$ and have $b(1)\neq0$, since $a,b$ have no common real root. We consider $\tau=\left[\begin{smallmatrix}
0 & 1-z^2\\
1 & 0
\end{smallmatrix}
\right]\in \mathscr G$, that corresponds to the reflection 
$(x,y,z)\mapsto (x,-y,z)$
of the sphere $S$ (see Remark~\ref{RemTau}). Note that this map is defined at the north and south poles, interchanges $L$ with $\overline L$ and interchanges $M$ with $\overline M$. It remains to study the map 
\[A\tau=\left[\begin{smallmatrix}
b(1-z^2) & a(1-z^2)\\
\overline{a} & \overline{b} (1-z^2)
\end{smallmatrix}
\right]=\left[\begin{smallmatrix}
b(1+z) & a_0(1-z^2)\\
\overline{a_0} & \overline{b} (1+z)
\end{smallmatrix}
\right]\in \mathscr G\subset \PGL(2,\C(z))\]
and to see that it is equal to $\left[\begin{smallmatrix}
a' & b'(1-z^2)\\
\overline{b'} & \overline{a'}
\end{smallmatrix}
\right]$, where $a'=b\cdot (1+z)$, $b'=a_0\in \C[z]$ have no common real root, and such that $a'(1)=2b(1)\neq0$. This reduces to the previous case.

The case where $a(-1)=0$ is similar.
\end{proof}

\begin{lemma}\label{BasePts}
Let $A=\left[\begin{smallmatrix}
a(z) & b(z)h\\
\bar b(z) & \bar a(z) 
\end{smallmatrix}
\right]\in \mathscr G\subset \PGL(2,\C(z)),$ for some polynomials $a,b\in\C[z]$ with no common real roots $($see Lemma~$\ref{G})$, and let $\mathcal A\in\Bir(S/\pi)$ be the corresponding element $($see Lemma~$\ref{LemAction})$. We denote by $D(z)=a(z)\bar a(z)-b(z)\bar b(z)(1-z^2)\in \R[z]$ the corresponding determinant.

Let $z_0\in (-1,1)\subset \R$, and let $\Gamma_{z_0}\subset S$ be the conic given by $z=z_0$. Then, the following hold:

\begin{enumerate}[$($a$)$]
\item The map $\mathcal A$ is a local isomorphism at each point of $\ \Gamma_{z_0}$ if and only if $D(z_0)\neq0$.
\item The map $\mathcal A$ contracts the curve $\Gamma_{z_0}$ onto a real point of $\ \Gamma_{z_0}$ if and only if $D(z_0)=0$. In this case, it has exactly one proper base-point on $\Gamma_{z_0}$, which is real.
\end{enumerate}
\end{lemma}
\begin{proof}
Observe that $\psi$ is a local isomorphism at a general point of $\Gamma_{z_0}$ by Lemma \ref{DescriptionPsi}. Hence, $\mathcal A$ contracts $\Gamma_{z_0}$ or is a local isomorphism at each point of it if and only if so does $A$ on the curve of $\A^2_\C$ given by $z=z_0$. Recall that $A$ acts as
\[(t,z)\dasharrow \left(\frac{a(z)t+b(z)(1-z^2)}{\overline{b}(z) t+\overline{a}(z)},z\right).\]
If $D(z_0)\neq0$, we obtain thus a local isomorphism along $\Gamma_{z_0}$. If $D(z_0)=0$, then $\frac{a(z_0)t+b(z_0)(1-(z_0)^2)}{\overline{b}(z_0) t+\overline{a}(z_0)}$ does not depend on $t$. The fact that $a$ and $b$ cannot both vanish at $z_0$ implies that the curve $\Gamma_{z_0}$ is then contracted onto one point, which is thus real. It has moreover exactly one proper base-point on this curve, which corresponds to the vanishing of the denominator and numerator of the above fraction. 
\end{proof}

\subsection{Algebraic description of $\Aut(S(\R)/\pi)$.}

The fact that an element in the group $\Aut(S(\R)/\pi)$ exchanges or not the lines $L$ and $\overline{L}$ can be checked geometrically, as the following result shows. This will help to describe algebraically the groups $\Aut(S(\R)/\pi)$ and $\Aut^+(S(\R)/\pi)$ as subgroups of $\mathscr G$ (Proposition~\ref{diffG} below).
\begin{lemma}\label{preserving}
Let $\mathcal A\in\Aut(S(\R)/\pi)$, and let $L,\overline L,M,\overline M\subset S_\C$ be the four curves given in Lemma~$\ref{DescriptionPsi}$. Then, one of the following holds:
\begin{enumerate}[$($a$)$]
\item
 $\mathcal A\in\Aut^+(S(\R)/\pi)$ and $\mathcal A$ preserves each of the four curves $L,\overline L,M,\overline M$.
 \item
 $\mathcal A\in\Aut(S(\R)/\pi)\setminus \Aut^+(S(\R)/\pi)$ and $\mathcal A$ exchanges $L$ with $\overline L$ and $M$ with $\overline M$.
 \end{enumerate}
\end{lemma}

\begin{proof}
Since $M\cup\overline M$ is the fibre of $(1:1)\in\Pp^1$ by $\pi$, every element of $\Aut(S(\R)/\pi)$ either preserves each of the two curves or interchanges them. 

We study the action of $\mathcal A$ on the lines $M$ and $\overline M$ near the point $P_N=M\cap \overline M=(1:0:0:1)$, the situation near $P_S=L\cap \overline L$ is similar. The equation of the sphere being $(w-z)(w+z)=x^2+y^2$, the complex tangent plane $T_{P_N}S_\C$ is given by $w=z=0$, and contains the two lines $M$ and $\overline M$, which correspond to $x=\pm {\bf i}y$.

The real tangent plane is contained in the complex tangent plane i.e. $T_{P_N}S(\R)\subset T_{P_N}S_\C$ and the action of $\mathcal A$ on the lines $M$ and $\overline M$ is the same as the action of its differential at $P_N$ denoted by $D_{P_N}\mathcal A\in\GL(2,\C)$ which also preserves $T_{P_N}S(\R)$ and is linear. Then $D_{P_N}\mathcal A$ can be presented as a matrix in $\GL(2,\R)$. 

Matrices in $\GL(2,\C)$ which preserve the two lines $x=\pm {\bf i}y$  are of the form $\left[\begin{smallmatrix}
a & b\\
-b & a
\end{smallmatrix}
\right]$ for some $a$, $b\in\C$. Imposing the condition of preserving the real plane is equivalent to ask for $a,b\in\R$. This tell us that if $D_{P_N}\mathcal A$ is the differential at $P_N$ of a diffeomorphism $\mathcal A$ which fixes $P_N$ and preserves the lines $M$ and $\overline M$, then $D_{P_N}\mathcal A$ restricted to $T_{P_N}(S(\R))$ is of the form $\left[\begin{smallmatrix}
a & b\\
-b & a
\end{smallmatrix}
\right]$ for some $a,b\in\R$ and is positive defined because its determinant is $a^2+b^2>0$ and therefore such a diffeomorphism $\mathcal A$ is an orientation-preserving one.

On the other hand, matrices in $\GL(2,\C)$ which interchange the lines $M$ and $\overline M$ and preserve the real tangent plane are of the form $\left[\begin{smallmatrix}
a & b\\
b & -a
\end{smallmatrix}
\right]$ for some $a$, $b\in\R$. Then if $D_{P_N}\mathcal A$ is the differential at $P_N$ of a diffeomorphism $\mathcal A$ which fixes $P_N$ and interchanges the lines $M$ and $\overline M$, we obtain that $D_{P_N}\mathcal A$ restricted to $T_{P_N}(S(\R))$ is of the form $\left[\begin{smallmatrix}
a & b\\
b & -a
\end{smallmatrix}
\right]$ for some $a,b\in\R$ and its determinant is $-(a^2+b^2)<0$  which implies that $\mathcal A$ is an orientation-reversing diffeomorphism.
\end{proof}
\begin{definition}
We denote by $\R[z]_+$ the multiplicative submonoid of $\R[z]$ defined as
$\R[z]_+:=\{f\in\R[z]\ |\ f(z_0)>0 \text{\ for each\ } z_0\in\R\}$.
\end{definition}
\begin{prop}\label{diffG}
Let $\mathscr H$ and $\mathscr H_0$ be the subgroups of \ $\mathscr G$ given respectively by $\psi\Aut(S(\R)/\pi)\psi^{-1}$ and $\psi\Aut^+(S(\R)/\pi)\psi^{-1}$. 

Then $\mathscr H=\mathscr H_0 \rtimes \langle \tau \rangle$, where $\tau=\left[\begin{smallmatrix}
0 & 1-z^2\\
1 & 0
\end{smallmatrix}
\right]=\left[\begin{smallmatrix}
0 & h\\
1 & 0
\end{smallmatrix}
\right]$ as before, and
\[\mathscr H_0=\left\{
\left[\begin{smallmatrix}
a(z) & b(z)h\\
\bar b(z) & \bar a(z) 
\end{smallmatrix}
\right] ;\ a,b\in \C[z], a \bar a-b\bar b h\in \R[z]_+\right\}.\] 
\end{prop}
\begin{proof}The fact that  $\mathscr H=\mathscr H_0 \rtimes \langle \tau \rangle$ follows from the fact that $\tau$ corresponds to a reflection in $\Aut(S(\R)/\pi)\setminus \Aut^+(S(\R)/\pi)$; it remains to describe $\mathscr H_0$.

Let $A\in\mathscr G$ be some element, that we write as $\left[\begin{smallmatrix}
a(z) & b(z)h\\
\bar b(z) & \bar a(z) 
\end{smallmatrix}
\right]$ for some polynomials $a,b\in\C[z]$ with no common real roots (Lemma~\ref{G}), and let $D=a \bar a -b\bar b h\in \R[z]$ be the corresponding determinant. We have $D(z)>0$ if $z^2>1$ (see Lemma~\ref{G}).
We denote by $\mathcal{A}\in \Bir(S/\pi)$ the corresponding element, given by $\psi^{-1} A \psi$.

Suppose that $A\in \mathscr H_0$. By Lemmas~\ref{ExchangeOrNotBir} and \ref{preserving}, this implies that $a(1)a(-1)\neq0$, hence $D(1)$ and $D(-1)$ are both positive. Moreover, $D(z)\neq0$ for each $z_0\in (-1,1)$ by Lemma~\ref{BasePts}. This implies that $D\in \R[z]_+$.

Conversely, suppose that $D\in \R[z]_+$. By Lemmas~\ref{ExchangeOrNotBir} and \ref{BasePts}, this implies that $\mathcal{A}$ is defined at each real point of the sphere, hence $A\in \mathscr H$. The fact that $A\in \mathscr H_0$ is given by Lemma~\ref{preserving}.
\end{proof}
\subsection{Involutions in $\Bir(S/\pi)$}\label{Sec:InvBir(S/pi)}
Recall that the group of elements of $\Bir(S,\pi)$ acting tri-vially on the basis of the fibration is denoted by $\Bir(S/\pi)$. This group is conjugate to 
\begin{align*}\mathscr G=\left\{
A=\left[\begin{smallmatrix}
a(z) & b(z)h\\
\bar b(z) & \bar a(z) 
\end{smallmatrix}
\right]\right. &;\ a,b \in \C[z] \text{ with no common real roots,}\\
&\left. \text{ and }
a(z) \bar a(z)-b(z)\bar b(z) h>0 \text{ for } z^2>1\phantom{\huge{l}}\right\}\subset\PGL(2,\C(z))
\end{align*}
by the birational map $\psi$ (see Lemma~\ref{LemAction}). In this subsection, we study involutions in $\Bir(S/\pi)$ or equivalently in $\mathscr G$ up to conjugacy.

We also recall that the action of $\PGL(2,\C(z))$ on $\A^2_\C$ was given in Equation~(\ref{actionPGL}) by $(t,z)\dasharrow\left(\frac{a(z)t+b(z)}{c(z)t+d(z)},z\right)$ for $\left[\begin{smallmatrix}a(z)&b(z)\\ c(z)&d(z)\end{smallmatrix}\right]=A\in\PGL(2,\C(z))$. Notice that when $A$ has order $2$, the restriction of $A$ to the $\Pp^1_\C$ corresponding to $z=z_0$, for a general $z_0 \in \C$, is an automorphism of order $2$ with two fixed points. We denote by $\Gamma_A$ the closure of the set of those fixed points as $z$ varies in $\C$ and call it the curve of fixed points of $A$ or just the curve fixed by $A$. The corresponding definition for the sphere is presented below, see Definition~\ref{Def:fixedCurve}.

The following results will be useful for the proof of the main result of this subsection in Theorem~\ref{thm}, which states that two involutions are conjugate in $\mathscr G$ if and only if their respective fixed curves are birational over $\R$.
\begin{lemma}\label{ConjPGL1}\ 
\begin{enumerate}[$(a)$]
	\item If $A\in\PGL(2,\C(z))$ is an element of order $2$, then $A$ is conjugate to $\left[\begin{smallmatrix}
0 & p\\
1 & 0
\end{smallmatrix}
\right]$ for some $p\in\C(z)^*$,
\item the elements $\left[\begin{smallmatrix}
0 & p\phantom{'}\\
1 & 0
\end{smallmatrix}
\right]$, $\left[\begin{smallmatrix}
0 & p'\\
1 & 0
\end{smallmatrix}
\right]\in\PGL(2,\C(z))$ with $p$, $p'\in\C(z)^*$ are conjugate in $\PGL(2,\C(z))$ if and only if $p/p'$ is a square in $\C(z)$. 
\item Let $A$, $B\in\PGL(2,\C(z))$ of order $2$. Then $A$ and $B$ are conjugate in $\PGL(2,\C(z))$  $(A\sim B)$ if and only if there exists a birational map $\rho$ defined over $\C$ 
\[\xymatrix@R-0.8pc{
\Gamma_A\  \ar[d]_\pi \ar@{-->}[r]^{\rho} & \ \Gamma_B \ar[d]_\pi  \\
\C \ar@{}[r]|-{=}   & \ \C}\] 
where $\Gamma_A$, $\Gamma_B\subset\C^2$ are the curves fixed by $A$ and $B$, respectively.
\end{enumerate}
\end{lemma}
\begin{proof}
\begin{enumerate}[$(a)$]
\item Let $A=\left[\begin{smallmatrix}
a & b\\
c & d
\end{smallmatrix}
\right]$ be an element of order 2 in $\PGL(2,\C(z))$. From $A^2
=\left[\begin{smallmatrix}
1 & 0\\
0 & 1
\end{smallmatrix}
\right]$, we get that $a=-d$ or $b=0=c$, but in the second case, $a^2=d^2$ thus $a=\pm d$. If $a=d$ and $b=c=0$ then $A=I$ and therefore $A$ does not have order 2. This implies that $a=-d$ in any case so we can write $A=\left[\begin{smallmatrix}
a & b\\
c & -a
\end{smallmatrix}
\right]$.
Now $A$ is conjugate to $\left[\begin{smallmatrix}
0 & a^2+bc\\
1 & 0
\end{smallmatrix}
\right]$ by $\left[\begin{smallmatrix}
-a & -b\\
1 & 0
\end{smallmatrix}
\right]$ when $b\neq0$ or by $\left[\begin{smallmatrix}
-c & a\\
0 & 1
\end{smallmatrix}
\right]$ when $c\neq0$. The case when $b=c=0$, we have $A=\left[\begin{smallmatrix}
1& 0\\
0 & -1
\end{smallmatrix}
\right]$ and is conjugate to $\left[\begin{smallmatrix}
0 & 1\\
1 & 0
\end{smallmatrix}
\right]$ by $\left[\begin{smallmatrix}
1 & 1\\
1 & -1
\end{smallmatrix}
\right]$. We have proved that $A$ is always conjugate to $\left[\begin{smallmatrix}
0 & p\\
1 & 0
\end{smallmatrix}
\right]$.
\item If $\left[\begin{smallmatrix}
0 & p\phantom{'}\\
1 & 0
\end{smallmatrix}
\right]$, $\left[\begin{smallmatrix}
0 & p'\\
1 & 0
\end{smallmatrix}
\right]$ are conjugate in $\PGL(2,\C(z))$ then the determinants are equal up to square and then $p/p'$ is a square. Reciprocally, if $p/p'=a^2$ for some $a\in\C(z)^*$ then $\left[\begin{smallmatrix}
0 & p\\
1 & 0
\end{smallmatrix}
\right]$ is conjugate to $\left[\begin{smallmatrix}
0 & p'\\
1 & 0
\end{smallmatrix}
\right]$ by $\left[\begin{smallmatrix}
1\phantom{'} & 0\\
0 & a\phantom{'}
\end{smallmatrix}
\right]$.
\item If $A$ and $B$ are conjugate elements of order 2 in $\PGL(2,\C(z))$, there is an element $\zeta\in\PGL(2,\C(z))$ such that the following diagram commutes:
\[\xymatrix@-1.3pc{
& \C^2 \ar@{-->}[rr]^{B} \ar[dd] & & \C^2 \ar[dd]^\pi \\
\C^2 \ar@{-->}[ur]_{\zeta} \ar[dd]_\pi \ar@{-->}[rr]^{\ \ \ A} & & \C^2 \ar@{-->}[ur]_{\zeta} \ar[dd]^{\pi} &  \\
& \C \ \ar[rr]_{=\ \ \ \ \ \ \ \ \ } & & \ \C\\
\C \ar[ur]^{=} \ar[rr]^{=} & & \C \ar[ur]_{=} &}\]
Then the existence of the birational map $\rho$ is given by the restriction of $\zeta$ to $\Gamma_A$.
Conversely, we assume the existence of $\rho:\Gamma_A\dasharrow\Gamma_B$.
By part (a), the fact that $A$ and $B$ are of order 2 implies that they are conjugate to an element of the form $\left[\begin{smallmatrix}
0 & f\\
1 & 0
\end{smallmatrix}
\right]$ and $\left[\begin{smallmatrix}
0 & g\\
1 & 0
\end{smallmatrix}
\right]$ respectively, for some $f,g\in\C(z)^*$. In this way, the equations for the curves $\Gamma_A$ and $\Gamma_B$ are $t^2=f(z)$ and $t^2=g(z)$. Since $\Gamma_A$ and $\Gamma_B$ are birational, this implies that the corresponding fields of rational functions are isomorphic i.e. $\C(z)[\sqrt{f}]\cong\C(z)[\sqrt{g}]$. The isomorphism will send $z\mapsto z$ and $\sqrt{f}\mapsto a\sqrt{g}+b$ for some $a,b\in\C(z)$ with $a\neq 0$. Since $f=g(=t^2)$, we have $f=(\sqrt{f})^2\mapsto (a\sqrt{g}+b)^2=a^2g+2ab\sqrt{g}+b^2=f$ then $a^2g+b^2-f=-2ab\sqrt g$ in $\C(z)[\sqrt g]$ which implies that $2ab\sqrt{g}=0$ and therefore $b=0$. Hence $f=a^2g$ and then $\left[\begin{smallmatrix}
0 & f\\
1 & 0
\end{smallmatrix}
\right]$ and $\left[\begin{smallmatrix}
0 & g\\
1 & 0
\end{smallmatrix}
\right]$ are conjugate by part (b).
\end{enumerate}
\end{proof}
\begin{lemma}\label{ConjPGL}
Let $A$, $B\in \mathscr G\subset\PGL(2,\C(z))$ be of order two. If $A$ and $B$ are conjugate in $\PGL(2,\C(z))$ then there are elements $\alpha$, $\beta\in\PGL(2,\C(z))$ such that $A=\alpha P\alpha^{-1}$, $B=\beta P\beta^{-1}$ for some $P=\left[\begin{smallmatrix}
0 & p\\
1 & 0
\end{smallmatrix}
\right]$, $p\in\R(z)^*$
\end{lemma}
\begin{proof}
By Lemma \ref{ConjPGL1} we can present $A$ and $B$ as in the statement for the same $P$ for some $p\in\C(z)^*$, what remains to show is that we can pick $p\in\R(z)^*$ (equivalently  $p=\bar{p}$). Let $A_0,\tau_0\in \GL(2,\C(z))$ be elements corresponding to $A,\tau\in \PGL(2,\C(z))$. We can choose $A_0$ so that $\det(A_0)=p$ and want to find an element $\mu\in\C(z)^{*}$ such that $\overline{p\mu^2}=p\mu^2$ because $\left[\begin{smallmatrix}
0 & p\\
1 & 0
\end{smallmatrix}
\right]$ is conjugate to $\left[\begin{smallmatrix}
0 & p\mu^2\\
1 & 0
\end{smallmatrix}
\right]$ by $\left[\begin{smallmatrix}
\mu & 0\\
0 & -1
\end{smallmatrix}
\right]$.  

The equality 
$\tau A \tau=\bar{A}$ in $\PGL(2,\C(z))$  implies that $(\tau_0)^{-1}A_0 \tau_0=\lambda \bar{A_0}$ for some element $\lambda\in \C(z)^{*}$. Taking the determinant, we obtain $\det(A_0)=\lambda^2\overline{\det(A_0)}$, which means that $p=\lambda^2\overline{p}$. It suffices to find $\mu$ with $\lambda=\frac{\overline{\mu}}{\mu}$. Since $\lambda^2=p/\overline{p}$, we obtain $\lambda^2\cdot \overline{\lambda}^2=1$, and thus $\lambda\overline{\lambda}=\pm 1$. If $\lambda\overline{\lambda}=1$ then by Hilbert's Theorem 90 there is $\mu\in\C(z)^*$ such that $\lambda=\frac{\overline{\mu}}{\mu}$. The case $\lambda\overline{\lambda}=-1$ is not possible in $\C(z)$ otherwise $\lambda$ would be the quotient of two polynomials in $\C(z)$, say $\lambda=\frac{f}{g}$ with $f,g\in\C[z]^*$ and then $\frac{f\bar{f}}{g\bar{g}}=-1$ which is equivalent to $f\bar{f}=-g\bar{g}$. But the leading coefficient of any element of the set $\{f\bar{f}\ :\ f\in\C[z]\}\subset\R[z]^*$ is always positive implying that $f\bar{f}$ cannot be equal to $-g\bar{g}$ for any $g\in\C(z)^*$.
\end{proof}
\begin{prop}\label{C(H)}
Let $F=\left[\begin{smallmatrix}
0 & f\\
1 & 0
\end{smallmatrix}
\right]$ with $f\in\C(z)^*$, 
\begin{enumerate}[$(a)$]
\item the centralizer of $F$ in $\PGL(2,\C(z))$, that we denote by $\Cc(F)$, is the semi-direct product $J_f\rtimes\Z/2\Z$ where $J_f$ is the image in $\PGL(2,\C(z))$ of $T_f$ where
\[T_f:=
\left\{\left[\begin{smallmatrix}
a & fb\\
b & a
\end{smallmatrix}
\right]\in\GL(2,\C(z))\ ;\ a,b\in\C(z), a^2-fb^2\neq0\right\}\] and $\Z/2\Z$ is generated by the element $\nu=\left[\begin{smallmatrix}
1 & 0\\
0 & -1
\end{smallmatrix}
\right]$ in $\PGL(2,\C(z))$.
\item The group $T_f$ is isomorphic to the multiplicative group $\C(\Gamma)^*$ where $\C(\Gamma)$ is the field of rational functions on $\Gamma$, the hyperelliptic curve $\Gamma$ of equation $t^2=f(z)$ in $\A^2_\C$ (the fixed curve of the birational map corresponding to the element $F$).
\item $H^1(\langle\nu\rangle,J_f)=\{1\}$.
\end{enumerate}
\end{prop}
\begin{proof}
\begin{enumerate}[$(a)$]
\item Let $A=\left[\begin{smallmatrix}
a & b\\
c & d
\end{smallmatrix}
\right]\in\PGL(2,\C(z))$, from $AF=FA$ we get $\left[\begin{smallmatrix}
b & fa\\
d & fc
\end{smallmatrix}
\right]=\left[\begin{smallmatrix}
fc & fd\\
a & b
\end{smallmatrix}
\right]$ implying that $d=\lambda a$, $b=\lambda f c$, $a=\lambda d$, and $fc=\lambda b$ for some $\lambda\in\C(z)^*$. If $a\neq0$ we have $a=\lambda^2a$ hence $\lambda=\pm1$ and $A=\left[\begin{smallmatrix}
a & fb\\
b & a
\end{smallmatrix}
\right]$ or $\left[\begin{smallmatrix}
a & -fb\\
b & -a
\end{smallmatrix}
\right]$. When $a=0$, we get $d=0$, and $fc=\lambda^2fc$ implying $\lambda=\pm1$ and $A=\left[\begin{smallmatrix}
0 & fb\\
b & 0
\end{smallmatrix}
\right]$ or $\left[\begin{smallmatrix}
0 & -fb\\
b & 0
\end{smallmatrix}
\right]$. Then $\Cc(F)=J_f\rtimes \left\langle\left[\begin{smallmatrix}
1 & 0\\
0 & -1
\end{smallmatrix}
\right]\right\rangle.$
\item An element of the field $\C(\Gamma)$ can be written as $a+b t$ with $a$, $b\in\C(z)$ and then we see that $\C(z)[\sqrt{f}]$ is isomorphic to $\C(\Gamma)$ by sending $a+b t$ to $a+b\sqrt{f}$. Hence we define the map from $\C(z)[\sqrt{f}]^*$ to $T_f$ given by $a+b\sqrt{f}\mapsto\left[\begin{smallmatrix}
a & fb\\
b & a
\end{smallmatrix}
\right]$ which is clearly bijective and is a group homomorphism since \[(a+b\sqrt{f})(c+d\sqrt{f})=(ac+f bd)+(ad+bc)\sqrt{f}\] corresponds to the product \[\left[\begin{smallmatrix}
a & fb\\
b & a 
\end{smallmatrix}
\right]\left[\begin{smallmatrix}
c & fd\\
d & c
\end{smallmatrix}
\right]=\left[\begin{smallmatrix}
ac+f bd & f(ad+bc)\\
ad+bc & ac+f bd
\end{smallmatrix}
\right].\]
\item From the exact sequence
\begin{equation}\label{origExSeq}
1\rightarrow\C(z)^*\xrightarrow i T_f\xrightarrow\mfp J_f\rightarrow 1
\end{equation} we obtain the cohomology exact sequence
\[H^1(\langle\nu\rangle,T_f)\rightarrow H^1(\langle\nu\rangle,J_f)\rightarrow H^2(\langle\nu\rangle,\C(z)^*).\]
The first cohomology group $H^1(\langle\nu\rangle,T_f)$ is trivial by Hilbert's Theorem 90 and the second cohomology group $H^2(\langle\nu\rangle,\C(z)^*)$ is trivial by Tsen's Theorem (\cite[Chapter X, Section 7]{Serre}). Then we get that $H^1(\langle\nu\rangle,J_f)=\{1\}$.\qedhere
\end{enumerate}
\end{proof} 

\begin{lemma}\label{murho}
Let $A\in\mathscr G$ of order $2$ and let $\alpha\in\PGL(2,\C(z))$ such that $A=\alpha P\alpha^{-1}$ for some $P=\left[\begin{smallmatrix}
0 & p\\
1 & 0
\end{smallmatrix}
\right]$, $p\in\R(z)^*$. Then the element $\mu_A\colon=\alpha^{-1}\tau\bar\alpha$ belongs to $J_p$ where $\tau=\left[\begin{smallmatrix}
0 & h\\
1 & 0
\end{smallmatrix}
\right]$ for $h=1-z^2$ and $J_p$ is defined in Proposition $\ref{C(H)}$.
\end{lemma}
\begin{proof}
The fact that $A\in \mathscr G$ implies that $\mu_A\in \Cc(P)$ because
\begin{align*}
\mu_A P\mu_A^{-1}&=(\alpha^{-1}\tau\overline{\alpha})P(\overline{\alpha}^{-1}\tau\alpha)=\alpha^{-1}\tau(\overline{\alpha}\overline{P}\overline{\alpha}^{-1})\tau\alpha\\
&=\alpha^{-1}(\tau \overline{A}\tau)\alpha =\alpha^{-1}A\alpha=P.
\end{align*}
In order to check that indeed $\mu_A$ belongs to $J_p$, we compute $P$ and $\alpha$ explicitly. First, we observe that if $A$ is an involution in $\mathscr G$ then $A$ is of the form $\left[\begin{smallmatrix}
\textbf{i}\cdot a(z) & b(z)h\\
\bar b(z) & -\textbf{i}\cdot a(z)
\end{smallmatrix}
\right]$ with $a(z)\in\R(z)$, $b(z)\in\C(z)$.
In $\PGL(2,\C(z))$, this involution is conjugate to the element $P=\left[\begin{smallmatrix}
0 & -(a^2-b\bar bh)\\
1 & 0 
\end{smallmatrix}
\right]$ by $\alpha=\left[\begin{smallmatrix}
0 & b(z)h\\
-1 & -\textbf{i}\cdot a(z) 
\end{smallmatrix}
\right]$. In this case, $p=-(a^2-b\bar bh)$ and then $\mu_A$ is explicitly $\left[\begin{smallmatrix}
\textbf{i}\cdot a(z) & -p\\
-1 & \textbf{i}\cdot a(z) 
\end{smallmatrix}
\right]$ which belongs to $J_p$. If $\alpha'$ is another element in $\PGL(2,\C(z))$ such that $\alpha'^{-1}A\alpha'=\left[\begin{smallmatrix}
0 & p\\
1 & 0 
\end{smallmatrix}
\right]$ then $\alpha'^{-1}\alpha\in\Cc(P)$, say $\theta=\alpha'^{-1}\alpha$. Then $\mu_A'=\alpha'^{-1}\tau\bar\alpha'=(\theta\alpha^{-1})\tau(\overline{\alpha\theta^{-1}})=\theta(\alpha^{-1}\tau\overline{\alpha})\overline{\theta^{-1}}$ that lies in $J_p$ as well. 
\end{proof}

\begin{definition}\label{Def:fixedCurve}
Let $\mathcal A\in \Bir(S/\pi)\setminus\{1\}$ be of finite order. For a general $z_0\in \R$ the birational map given by $\mathcal A$ fixes the conic $\Gamma_{z_0}$ corresponding to the preimage of $z_0$ by $\pi$. Note that $\mathcal A$ restricted to $\Gamma_{z_0}$ ($\mathcal A_{\Gamma_{z_0}}\colon\Gamma_{z_0}\rightarrow\Gamma_{z_0}$) is an isomorphism with exactly two fixed points, which can be two real points or two imaginary conjugate points. The (closure of) the set of these fixed points, for every $z\in\Pp^1$, gives the \emph{curve of fixed points} that we denote by $\Fix(\mathcal A)$ and that is a double covering of $\Pp^1$. Note that some isolated points can also be fixed and not belong to $\Fix(\mathcal A)$. 
\end{definition}

\begin{thm}\label{thm}
Let $\mathcal A$, $\mathcal B\in \Bir(S/\pi)$ of order $2$. The elements $\mathcal A$ and $\mathcal B$ are conjugate in $\Bir(S/\pi)$  $(\mathcal A\sim_{\Bir(S/\pi)} \mathcal B)$ if and only if there exists a birational map $\rho$ defined over $\R$ 
\[\xymatrix@R-0.8pc{
\Fix(\mathcal A)\  \ar[d]_\pi \ar@{-->}[r]^{\rho} & \ \Fix(\mathcal B) \ar[d]_\pi  \\
\R \ar@{}[r]|-{=}   & \ \R}\] 
with $\Fix(\cdot)$ as in the precedent paragraph.
\end{thm}
\begin{proof}
If $\mathcal A$ and $\mathcal B$ are conjugate in $\Bir(S/\pi)$, then there is an element $\zeta\in\Bir(S/\pi)$ such that $\zeta \mathcal A\zeta^{-1}=\mathcal B$ and then the map $\rho$ is given by the restriction of $\zeta$ to $\Fix(\mathcal A)$ which is defined over $\R$.

In order to prove the sufficiency, we assume that there is $\rho\colon\Fix(\mathcal A)\dasharrow\Fix(\mathcal B)$ with $\sigma\rho=\rho\sigma$. Then by Lemma \ref{ConjPGL1}(c), we obtain that $A:=\psi \mathcal A\psi^{-1}\in\mathscr G$ and $B:=\psi \mathcal B\psi^{-1}\in\mathscr G$ are conjugate in $\PGL(2,\C(z))$ and by Lemma \ref{ConjPGL} there are $\alpha$, $\beta\in\PGL(2,\C(z))$ such that $A=\alpha F\alpha^{-1}$, $B=\beta F\beta^{-1}$ and $F=\left(\begin{smallmatrix}0 & f \\ 1 & 0\end{smallmatrix}\right)$, for some $f\in \R(z)^{*}$. Observe that the action  of $\alpha$ and $\beta$ on $S_\C$ restrict to birational maps  $\Fix(F)\dasharrow \Fix(A)$ and $\Fix(F)\dasharrow \Fix(B)$, respectively. To sum up, we have the following diagram (which is not necessarily commutative, since $\rho\colon\Fix(A)\dasharrow\Fix(B)$ may be not the restriction of $\beta\alpha^{-1}$):
\[\xymatrix@R-0.8pc{
\Fix(A)\  \ar@/^2pc/ @{-->}[rr]^{\rho\text{\ defined over $\R$}} \ar@/_/ @{-->}[r]_{\alpha^{-1}} \ar[d]_\pi &\ \Fix(F) \ar@/^/ @{-->}[r]^{\beta} \ar@/_/ @{-->}[l]_{\alpha} \ar[d]_\pi &\ \Fix(B) \ar[d]_\pi \ar@/^/ @{-->}[l]^{\beta^{-1}} \\
\C \ar@{}[r]|-{=}  & \C  \ar@{}[r]|-{=} & \ \C}.\]

Since we want to show that $A\sim_{\mathscr G} B$ (or equivalently $\mathcal A\sim_{\Bir(S/\pi)} \mathcal B$), we need to find $\gamma\in \mathscr G$ such that $\gamma A\gamma^{-1}=B$ i.e. $\gamma\alpha F\alpha^{-1}\gamma^{-1}=\beta F\beta^{-1}$ $\Longleftrightarrow$ $\beta^{-1}\gamma\alpha F(\beta^{-1}\gamma\alpha)^{-1}=F$, hence $\beta^{-1}\gamma\alpha\in \Cc(F)$. In other words, finding $\gamma\in \mathscr G$ so that $\gamma A\gamma^{-1}=B$ is equivalent to find $\xi\in \Cc(F)$ such that $\beta\xi\alpha^{-1}\in \mathscr G$.  

The condition $\beta\xi\alpha^{-1}\in \mathscr G$ is the same as $\tau(\beta\xi\alpha^{-1})\tau=\overline{\beta\xi\alpha^{-1}}$ which is equivalent to $\xi=(\beta^{-1}\tau\overline{\beta})\overline{\xi}(\overline{\alpha}^{-1}\tau\alpha)$. We define $\mu_B:=\beta^{-1}\tau\overline{\beta}$ and $\mu_A^{-1}:=\overline{\alpha}^{-1}\tau\alpha$ and like this, we need to find $\xi\in\Cc(F)$ such that $\xi=\mu_B\bar{\xi}\mu_A^{-1}$. By Lemma \ref{murho} $\mu_A$, $\mu_B\in J_f$ and then also $\mu_A^{-1}\in J_f$. On the other hand, $\mu_A^{-1}\overline{\mu_A^{-1}}=1$ and $\mu_B\overline{\mu_B}=1$ and as $J_f$ is abelian, we get $\mu_B\mu_A^{-1}\cdot \overline{\mu_B\mu_A^{-1}}=1$ and then by Proposition \ref{C(H)}(c) there is $\xi\in J_f$ such that $\xi/\bar{\xi}=\mu_B\mu_A^{-1}$ $\Longrightarrow$ $\xi=\mu_A\bar{\xi}\mu_A^{-1}$.
\end{proof}

\subsection{Involutions in $\Aut(S(\R)/\pi)$.}
In Proposition \ref{diffG}, we have described algebraically the orientation preserving birational diffeomorphisms as the group \[\mathscr H_0=\left\{
\left[\begin{smallmatrix}
a(z) & b(z)h\\
\bar b(z) & \bar a(z) 
\end{smallmatrix}
\right] ;\ a,b\in \C[z], a \bar a-b\bar b h\in \R[z]_+\right\}.\] 
We want to describe involutions in $\mathscr H\simeq\mathscr H_0\rtimes \langle\tau\rangle$ where $\tau=\left[\begin{smallmatrix}
0 & h\\
1 & 0 
\end{smallmatrix}
\right]$.
\begin{lemma}\label{invform}
Every involution $\iota\in\mathscr H_0$ is equal to \[\iota=\left[\begin{smallmatrix}
{\bf i}\cdot p(z) & q(z)h\\
\bar q(z) & -{\bf i}\cdot p(z) 
\end{smallmatrix}
\right]\] for some $p\in\R[z]$ and $q\in\C[z]$ with no common real roots and $p^2-q\bar qh\in\R[z]_+$. 
\end{lemma}
\begin{proof}
All such elements are indeed involutions, as one easily calculates.
From the proof of the first statement of Lemma \ref{ConjPGL1}, we see that the
trace of any involution in $\PGL(2,\C(z))$ vanishes.
Since in $\mathscr{H}_0$ the diagonal entries are conjugate, they are strictly imaginary,
from which the claim follows.
\end{proof}

Fibrewise, the maps in $\mathscr{H}_0$ look like rotations, the maps in $\mathscr{H} \setminus \mathscr{H}_0$ like reflections:
\begin{lemma}\label{Lem:LocRot}
The restriction of an involution $\iota\in\mathscr H_0$ to a fibre is conjugate, inside the
group of automorphisms of the circle, to a rotation by $\pi$. For an element in $\mathscr{H} \setminus \mathscr{H}_0$, the restriction is conjugate to a reflection.
\end{lemma}
\begin{proof}
A fibre is a subvariety of the real points of $S$ and isomorphic to a circle $S^1$,
which in turn is isomorphic to $\Pp^1(\R)$.
Therefore $\iota$ restricts on each fibre to an automorphism of $\Pp^1(\R)$, that is, an
element of $\PGL(2,\R)$.
The first statement of Lemma \ref{ConjPGL1} applies equally when the field $\R$ instead of $\C(z)$
is used, which tells us that the automorphism is conjugate to an element of the form
$\left[
	\begin{smallmatrix}
	0 & \pm p\\
	1 & \ 0
	\end{smallmatrix}
\right]$,
with $p>0$ in $\R$.
The sign is negative for $\mathscr{H}_0$ and positive for $\mathscr{H} \setminus \mathscr{H}_0$, and depends on whether the element is orientation-preserving or -reversing.
With $q = \sqrt{p}$, the element is equal to
$\left[
	\begin{smallmatrix}
	0 & \pm q\\
	q^{-1} & \ 0
	\end{smallmatrix}
\right]$,
which is conjugate to
$\left[
	\begin{smallmatrix}
	0 & \pm 1\\
	1 & \ 0
	\end{smallmatrix}
\right]$
via
$\left[
	\begin{smallmatrix}
	1 & 0\\
	0 & p
	\end{smallmatrix}
\right]$.
These elements describe a rotation and a reflection, as claimed.
\end{proof}

Recall that $\R[z]_+:=\{f\in\R[z]\ |\ f(z_0)>0 \text{\ for each\ } z_0\in\R\}$. We will need the following description.

\begin{lemma}\label{positiv}
$\R[z]_+=\{p\bar p\ |\ p\in\C[z],\ p \text{\ has no real root}\}$
\end{lemma}
\begin{proof}
Since $f(z)>0$ for every $z\in\R$, $f$ has complex roots which can be sorted as pairs of complex conjugate roots. Then $f$ can be factorised in $\C(z)$ as factors of the form $(z-\alpha)(z-\bar\alpha)$ which already have the form $p_\alpha\bar p_\alpha$ with $p_\alpha=z-\alpha$ for every complex root $\alpha$ of $f$. We then construct $p'$ as the product $p'=p_{\alpha_1}\cdot p_{\alpha_2}\cdots p_{\alpha_k}$ where $k$ is the number of pairs of complex conjugate roots and in this way, $f=\lambda\cdot p'\cdot\overline{p'}$ for some real positive constant $\lambda$. Thus we define $p=\sqrt{\lambda}p'$ and the result follows.
\end{proof}

\begin{prop}\label{diff}
Let $A\in\mathscr H$ be an element of order $2$. Then the curve $\Fix(A)$, which is a double covering of $\Pp^1$, has the following properties:
\begin{enumerate}[$(a)$]
\item If $A\in\mathscr H_0$, then $\Fix(A)$ has no real point $(0$ oval$)$;
\item if $A\in\mathscr H\setminus\mathscr H_0$, then $\Fix(A)$ has one oval and $\pi(\Fix(A)(\R))=[-1,1]$.
\end{enumerate}
\end{prop}

\begin{proof}
Let $A\in\mathscr H$ be an element of order two. By Lemma~\ref{invform}, $A$ is of the form $\left[\begin{smallmatrix}
{\bf i}\cdot p(z) & q(z)h\\
\bar q(z) & -{\bf i}\cdot p(z) 
\end{smallmatrix}
\right]$ where $p\in\R[z]$, $q\in\C[z]$ and $p,q$ have no common real roots. The curve of fixed points is given by $\bar q(z)t^2-2{\bf i}p(z)t+q(z)h=0$ whose discriminant (with respect to $t$) is $-4(p^2+q\bar qh)$ and corresponds to minus the determinant of the matrix.

If $A\in\mathscr H_0$, then the determinant is positive, so $\Fix(A)$ does not have any real point.

If $A\in\mathscr H\setminus\mathscr H_0$, then the determinant is negative (because it is $(1-z^2)$ times the positive determinant). Hence, we get $2$ real points for each $z_0\in (-1,1)$.
\end{proof}

According to Proposition \ref{diff}, for an involution which is also a diffeomorphism its curve of fixed points is birational to a smooth real hyperelliptic curve with no oval or just one. In the first case, there is no real point on the fixed curve and $1$ and $-1$ are not ramification points. This involution is an orientation preserving diffeomorphism with two isolated fixed points. In the second case, the only two ramification points are $1$ and $-1$, the oval is sent by $\pi\colon S\to \A^1$ onto the real interval $[-1,1]$ and this involution is an orientation reversing diffeomorphism. Both possible cases for the curve of fixed points are illustrated in Figure \ref{fixcur}.
Now, we would like to prove the converse, i.e. for any hyperelliptic curve with one or no oval (equation of the form $t^2=(1-z^2)p$ or $t^2=-p$ for some $p\in\R[z]_+$ with no real roots) we want to associate an element $\gamma$ of $\mathscr H$ which realises the curve as $\Fix(\gamma)$. We need first to prove the following lemmas.
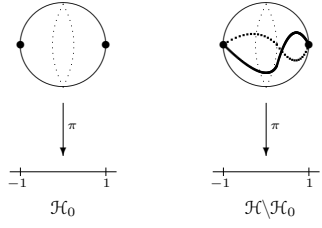
\begin{figure}[h]
\begin{center}
\begin{picture}(120,69)
\scalebox{0.8}{
\multiput(25,70)(95,0){2}{\circle{40}}

\put(2.5,67.5){\unboldmath$\bullet$}
\put(42.5,67.5){\unboldmath$\bullet$}
\put(97.5,67.5){\unboldmath$\bullet$}
\put(137.5,67.5){\unboldmath$\bullet$}

\multiput(25,42.5)(95,0){2}{\vector(0,-1){25}}
\multiput(27,30)(95,0){2}{\unboldmath{\tiny $\pi$}}

\multiput(0,10)(95,0){2}{\line(1,0){50}}
\multiput(5,12)(95,0){2}{\line(0,-1){4}}
\multiput(45,12)(95,0){2}{\line(0,-1){4}}
\multiput(-1,3)(95,0){2}{\unboldmath{\tiny $-1$}}
\multiput(43.5,3)(95,0){2}{\unboldmath{\tiny $1$}}

\qbezier[15](25,50)(35,70)(25,89)
\qbezier[15](25,50)(15,70)(25,89)
\qbezier[15](120,50)(130,70)(120,89)
\qbezier[15](120,50)(110,70)(120,89)

\put(19,-10){\unboldmath{\small $\mathscr H_0$}}
\put(110,-10){\unboldmath{\small $\mathscr H\setminus\mathscr H_0$}}

\linethickness{1pt}\qbezier(100,70)(120,50)(125,60)
\linethickness{1pt}\qbezier(125,60)(132,85)(140,70)
\linethickness{1pt}\qbezier[15](100,70)(115,80)(125,70)
\linethickness{1pt}\qbezier[15](125,70)(135,55)(140,70)}
\end{picture}
\end{center}
\caption{Possible appearances of the fixed curve of elements in $\Aut(S(\R)/\pi)$.}
\label{fixcur}
\end{figure}
\begin{lemma}\label{quadratic}
Let $f\in\R[z]$ be a polynomial of degree two such that $f\in \R[z]_+$ then there exist $a\in\R[z]$ and a positive real number $c$  such that $f(z)=a(z)^2+c(z^2-1)$.
\end{lemma}
\begin{proof}
Since $f\in\R[z]_+$, then $f$ is factorised as $f(z)=(z-\alpha)(z-\bar\alpha)=z^2-(\alpha+\bar\alpha)z+\alpha\bar\alpha$ for $\alpha$ a complex number and making $\alpha=b+{\bf i}d$, we rewrite $f$ as $f(z)=z^2-2bz+(b^2+d^2)$. Then if we write $a(z)^2=f(z)-c(z^2-1)=(1-c)z^2-2bz+b^2+d^2+c$, we want to show that there exist some value of $c>0$ such that the right side is indeed a square with respect to $z$. So we want the discriminant of such an expression to be zero. This is $4b^2-4(1-c)(b^2+d^2+c)=4(c^2+(b^2+d^2-1)c-d^2)=0$ which implies that $c$ is a positive solution of $p(c):=c^2+(b^2+d^2-1)c-d^2$ so we compute the discriminant of this quadratic expression with respect to $c$ and want it to be larger than zero i.e. $\Delta_c:=(b^2+d^2-1)^2+4d^2>0$ but this is always the case. Now, since the leading coefficient of $a(z)^2$ has to be larger than zero, implies that $c<1$ so we just check that the discriminant which depends on $c$ has a root between $0$ and $1$ which is true because $p(0)=-d^2<0$ and $p(1)=b^2>0$. What remains is to check the case $b=0$ i.e. $\alpha={\bf i}d$. In this case, $f(z)=z^2+d^2$ so we just take $c=1$ and $a=\sqrt{d^2+1}$.
\end{proof}

\begin{lemma}\label{setV}
Let $V$ be the set \[V=\{a^2+P\cdot(z^2-1)\ |\ a\in\R[z], P\in\R[z]_+\}.\]
\begin{enumerate}[$(a)$]
\item If $f,g\in V\cap \R[z]_+$, then $f\cdot g\in V\cap \R[z]_+$,
\item $\R[z]_+\subset V$.
\end{enumerate}
\end{lemma}
\begin{proof}

\begin{enumerate}[$(a)$]
\item Let $f,g\in V\cap \R[z]_+$ then $f=a^2+P\cdot (z^2-1)$ and $g=b^2+Q\cdot(z^2-1)$ for $a,b\in\R[z]$ and $P,Q\in\R[z]_+$. We have then
\begin{align*}
f\cdot g=& (a^2+P\cdot(z^2-1))(b^2+Q\cdot(z^2-1))\\
=&(ab)^2+(z^2-1)[a^2Q+P(b^2+Q\cdot(z^2-1))]
\end{align*}
and $a^2Q+P(b^2+Q\cdot(z^2-1))\in\R[z]_+$ because $a^2$, $Q$, $P$, and $b^2+Q\cdot(z^2-1)$ are all in $\R[z]_+$. Therefore, $f\cdot g\in V\cap\R[z]_+.$
\item Let $f\in\R[z]_+$ then $f$ can be presented as a product of quadratic polynomials. Since every quadratic factor is also in $\R[z]_+$, thus it suffices to prove the Lemma in the case where $f$ is quadratic and this was already proved in Lemma \ref{quadratic}.\qedhere
\end{enumerate}
\end{proof}

\begin{lemma}\label{reciprocalProp3.14}
The elements in $\Aut(S(\R)/\pi)$ realise all smooth real hyperelliptic curves with at most one oval. More precisely, 
\begin{enumerate}[$($a$)$]
\item for a real smooth hyperelliptic curve with one oval of the form $t^2=(1-z^2)\beta\bar\beta$ for some $\beta\in\C(z)$ with no real roots there is an orientation reversing birational diffeomorphism whose fixed curve is this curve,
\item for a real smooth hyperelliptic curve with no oval of the form $t^2=-\beta\bar\beta$ for some $\beta\in\C(z)$ with no real roots there is an orientation preserving birational diffeomorphism whose fixed curve is this one.
 \end{enumerate}
\end{lemma}

\begin{proof}
Given the hyperelliptic curve $t^2=(1-z^2)\beta\bar\beta$ for some $\beta\in\C(z)$ with no real roots, the element $\alpha=\left[\begin{smallmatrix}
0 & \beta(z)h\\
\bar \beta(z) & 0 
\end{smallmatrix}
\right]$ is an involution in $\mathscr H\setminus \mathscr H_0$ whose fixed curve is $t^2=(1-z^2)\beta\bar\beta$. In the other case, when $t^2=-\beta\bar\beta$ where $\beta$ has no real roots, we have $\beta\bar\beta\in\R[z]_+\subset V$ by Lemma \ref{setV} and then there are $a\in\R[z]$ and $P\in\R[z]_+$ such that $\beta\bar\beta=a^2+P(z^2-1)$. Lemma \ref{positiv} implies that $P=b\bar b$ for some $b\in\C[z]$ then the element $\alpha=\left[\begin{smallmatrix}
{\bf i}a(z) & b(z)h\\
\bar b(z) & -{\bf i}a(z) 
\end{smallmatrix}
\right]$ is an involution in $\mathscr H_0$ whose fixed curve is $t^2=-\beta\bar\beta$.
\end{proof}

\begin{lemma}\label{1basepoint}
Let $a,b,c,d\in \C[z]$ and let $A(z)=\left[\begin{smallmatrix} a(z)& b(z) \\ c(z)& d(z)\end{smallmatrix}\right]\in \GL(2,\C(z))$. Let $z_0\in \C$ be a simple root of $ad-bc\in \C[z]$, such that $A(z_0)$ has rank $1$.

Then, the birational map of \ $\mathbb{P}^1\times \A^1$ given by 
\[([t:u],z)\dasharrow \left([a(z)t+b(z)u:c(z)t+d(z)u],z\right)\]
has exactly one base-point on the line $z=z_0$, and no infinitely near base-point to this one.
\end{lemma}
\begin{proof}
Making the change of variable $z\mapsto z-z_0$, we can assume that $z_0=0$. Replacing $A(z)$ with $\alpha A(z)\beta$, where $\alpha,\beta\in \GL(2,\C)$, we can moreover assume that $A(0)=\left[\begin{smallmatrix} 0&0 \\ 1& 0\end{smallmatrix}\right]$, so we can write $A(z)=\left[\begin{smallmatrix} za(z)& zb(z) \\ 1+zc(z)& zd(z)\end{smallmatrix}\right]$, for some $a,b,c,d\in \C[z]$ (which are not the same as before but we keep the same letters to simplify the notation). Since $z_0$ is a simple root of the determinant, we have $b(0)\neq0$. The corresponding birational map of $\mathbb{P}^1\times \A^1$ is then 
\[([t:u],z)\dasharrow \left([z(a(z)t+b(z)u):t+z(c(z)t+d(z)u)],z\right)\]
and has a unique proper base-point on the line $z=0$, which is the point $([0:1],0)$. 

 The blow-up of this point is locally given by 
\[\begin{array}{rccc}
\pi\colon& \A^2&\to & \mathbb{P}^1\times \A^1\\
&(t,v) & \mapsto & ([t:1],tv)\end{array}\]
And the lift of our birational is then locally given by
\[(t,v)\dasharrow \left(\frac{v(a(vt)t+b(vt))}{c(vt)tv+d(vt)v+1}, \frac{t(c(vt)tv+d(vt)v+1)}{a(vt)t+b(vt)}\right).\]
The curves $E,E'$ corresponding respectively to the exceptional divisor and the fibre $z=z_0$ are now given by $t=0$ and $v=0$ respectively, and exchanged by the lift:
\[\begin{array}{rcl}
(0,v)&\dasharrow& \left(\frac{v b(0)}{1+d(0)v},0\right)\\
(t,0)&\dasharrow& \left(0,\frac{t}{a(0)t+b(0)}\right)\end{array}\]
This implies that both, our map and its inverse, have a simple base-point at $(0,0)$.
\end{proof}

\begin{thm}\label{thm4}
Let $g$, $g'\in\Aut(S(\R)/\pi)$ of order $2$. Then $g$ and $g'$ are conjugate in $\Bir(S/\pi)$ if and only if they are conjugate in $\Aut(S(\R)/\pi)$.
\end{thm}
\begin{proof}
Let $g$ and $g'$ be conjugate in $\Bir(S/\pi)$, then there is $\alpha\in\Bir(S/\pi)$ such that $\alpha g\alpha^{-1}=g'$. We want to show that $g$ and $g'$ are conjugate in $\Aut(S(\R)/\pi)$.
By Proposition~\ref{diff}, the curve of fixed points of an element in $\Aut(S(\R)/\pi)$ either contains no real point or only one oval.

 If $\alpha\in\Bir(S/\pi)\setminus\Aut(S(\R)/\pi)$, there is a real point $r\in S(\R)$ where $\alpha$ is not defined, and this point is not $P_S$ or $P_N$ (Lemma \ref{ExchangeOrNotBir}). The element $\alpha$ blows up this point and contracts the conic $\Gamma_{z_r}$ passing through $r$ which is a fibre of the conic bundle structure of $S$. Then $\alpha(\Gamma_{z_r})=q$ for some $q\in S(\R)$.
 
 Note that $q$ is fixed by $g$. Indeed, otherwise $g(q)=q'\neq q$ and as $g$ preserves the fibration, $g(\Gamma_{z_r})=\Gamma_{z_r}$, then $\alpha(g(\Gamma_{z_r}))\neq g'(\alpha(\Gamma_{z_r}))$. Since $q$ is a real point fixed by $g$ and distinct from $P_S$ and $P_N$, the curve $\Fix(g)$ contains real points. We may then assume that $g$ is equal to $\left[\begin{smallmatrix}0& b(z)h\\
\bar b(z)& 0 
\end{smallmatrix}\right]$ (Lemma~\ref{reciprocalProp3.14}). The centraliser of $g$ contains the following subgroup
\[\Cc(g)=\left\{\left[\begin{smallmatrix}a(z) & \lambda b(z)h\\
\lambda\bar b(z) & a(z)\end{smallmatrix}\right]\ ;\ a,\lambda\in\R[z]\text{ and } a^2-\lambda^2b\bar bh\neq0\right\}\subset\mathscr G.\]
We want to prove now that $\Cc(g)$ contains, in particular, an element $\beta=\left[\begin{smallmatrix}a(z) & b(z)h\\
\bar b(z) & a(z)\end{smallmatrix}\right]$ such that $D(z)=a(z)^2-b(z)\bar b(z)(1-z^2)$ has only one zero exactly at $z=z_r$ on the interval $(-1,1)$. The reason of the existence of such a $\beta$ is that it is possible to find a polynomial $a(z)$ with values $a(-1)=0$ and $a(z_r)=\sqrt{b(z_r)\bar b(z_r)(1-z_r^2)}$ and satisfying that $D(z)>0$ on $(-\infty,-1)\cup(z_r,\infty)$ and $D(z)<0$ on the interval $(-1,z_r)$. Notice that $b(z)\bar b(z)(1-z^2)>0$ for $z\in(-1,1)$ and the condition $D(z)>0$ for $z^2>1$ is already fulfilled (see Lemma~\ref{G}). We use the function $f(z)=z^m$ with $m$ sufficiently large and apply a suitable linear change of coordinates, namely $a(z)=\sqrt{b(z_r)\bar b(z_r)(1-z_r^2)}\cdot f\left(\frac{z+1}{z_r+1}\right)$ to get the polynomial $a(z)$ with the required conditions. See Figure~\ref{Fig:a(z)}.
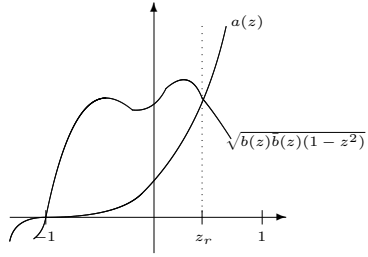
\begin{figure}[h]
\begin{center}
\begin{picture}(90,69.3)
\scalebox{0.9}{
\put(-10,0){\vector(1,0){115}} %axis
\put(50,-15){\vector(0,1){105}} % axis y
\qbezier[30](70,80)(70,35)(70,0) % recta x=z_r

\put(5,3){\line(0,-1){6}}
\put(0,-10){\unboldmath{\tiny $-1$}}
\put(95,3){\line(0,-1){6}}
\put(93,-10){\unboldmath{\tiny $1$}}
\put(70,3){\line(0,-1){6}}
\put(67,-10){\unboldmath{\tiny $z_r$}}

\qbezier(5,0)(35,0)(45,10) % a(z)
\qbezier(45,10)(70,35)(80,80) 
\qbezier(-10,-10)(-8,0)(5,0) 
\put(82,80){\unboldmath{\tiny $a(z)$}}

\qbezier(0,-9)(4,-6)(5,0) % \sqrt(b\bar b*h)
\qbezier(5,0)(15,50)(30,50) 
\qbezier(30,50)(35,50)(41,45) 
\qbezier(41,45)(50,44)(55,54) 
\qbezier(55,54)(65,63)(70,50) 
\qbezier(70,50)(75,44)(82,33) 
\put(80,30){\unboldmath{\tiny $\sqrt{b(z)\bar b(z)(1-z^2)}$}}}
\end{picture}
\end{center}
\caption{Conditions for the polynomial $a(z)$}
\label{Fig:a(z)}
\end{figure}

With $\beta\in\Cc(g)$ as before, i.e. the element with the only root of its determinant at $z=z_r$, Lemma~\ref{1basepoint} implies that the birational map that $\beta$ defines has exactly one real base-point and no infinitely near base-point to this one. Then replacing $\alpha$ by $\beta^{-1}\alpha$, one gets one base-point less. Then the claim follows by induction. 
\end{proof}

\begin{prop}\label{correspAut(S(R))}
There are bijective correspondences
 \[
     \left\{\!
         \begin{array}{cc}
              \mbox{conjugacy classes of involutions} \\
              \mbox{in $\Aut^+(S(\R)/\pi)$}
         \end{array}
     \!\right\}
     \xleftrightarrow{\ 1:1 \ } 
     \left\{\!\!
         \begin{array}{cc}
              \mbox{smooth real projective} \\
              \mbox{curves $\Gamma$ with no real point}\\
              \mbox{with $\pi\colon\Gamma\to\Pp^1$ a $2:1$-covering,}\\
              \mbox{up to $\pi$-isomorphism}
         \end{array}
     \!\!\right\}
\]

 \[
     \left\{\!
         \begin{array}{cc}
              \mbox{conjugacy classes of involutions} \\
              \mbox{in $\Aut(S(\R)/\pi)\setminus\Aut^+(S(\R)/\pi)$}
         \end{array}
     \!\right\}
     \xleftrightarrow{\ 1:1 \ } 
     \left\{\!\!
         \begin{array}{cc}
              \mbox{smooth real projective} \\
              \mbox{curves $\Gamma$ with one oval}\\
              \mbox{with $\pi\colon\Gamma\to\Pp^1$ a $2:1$-covering,}\\
              \mbox{up to $\pi$-isomorphism}
         \end{array}
     \!\!\right\}
\]
\end{prop}
\begin{remark}By a $\pi$-isomorphism we mean an isomorphism $\gamma\colon \Gamma\to \Gamma'$ such that $\pi\gamma=\pi$.\end{remark}

\begin{proof} 
Let $g,g'\in\Aut(S(\R)/\pi)$ be of order $2$.
If $g$ and $g'$ are conjugate in $\Aut(S(\R)/\pi)$ then by Theorem~\ref{thm}, $\Fix(g)$ and $\Fix(g')$ are birational over $\R$ by some $\pi$-isomorphism.
Proposition~\ref{diff} tell us that $\Fix(g)$ and $\Fix(g')$ are a double covering of $\Pp^1$ with no real point (when $g,g'$ are orientation-preserving birational diffeomorphisms) or with one oval (when $g,g'$ are orientation-reversing birational diffeomorphisms), and Lemma~\ref{reciprocalProp3.14} shows that all such curves are obtained.
Given a $\pi$-isomorphism between two smooth real hyperelliptic curves with no oval (respectively one), Theorem~\ref{thm} implies that $g$ and $g'$ are conjugate in $\Bir(S/\pi)$ and Theorem~\ref{thm4} that $g$ and $g'$ are indeed conjugate in $\Aut(S(\R)/\pi)$.
\end{proof}

\subsection{Elements in $\Bir(S/\pi)$ of finite order larger than two}\label{Sec:Bir(S/pi)Order>2}
\label{SecRotation}
The goal of this subsection is to show that any element in $\Bir(S/\pi)$ of finite order larger than two which preserves the fibration is conjugate to a rotation. We start by observing that any rotation $\rho_\theta\in\Bir(S/\pi)$ is given by the map 
\[\begin{array}{rcrl}
\rho_\theta\colon& S & \longrightarrow &\ \ \ \ \  \  S\\
&(x,y,z)&\longmapsto&(x\cos\theta-y\sin\theta, x\sin\theta+y\cos\theta,z)\end{array}\]
which via $\psi$ (Lemma \ref{LemAction}) corresponds in $\A^2$ to the map $(t,z)\mapsto(te^{-{\bf i}\theta},z)$ and is equivalent to the action of the element $\left[\begin{smallmatrix}
e^{-{\bf i}\theta} & 0\\
0\phantom{\bar b} & 1
\end{smallmatrix}
\right]=\left[\begin{smallmatrix}
e^{-{\bf i}(\theta/2)} & 0\\
0 & e^{{\bf i}(\theta/2)}
\end{smallmatrix}
\right]=\left[\begin{smallmatrix}
1\phantom{\bar b} & 0\\
0\phantom{\bar b} & e^{{\bf i}\theta}
\end{smallmatrix}
\right]=R_\theta\in\mathscr G$.  With this observation and the following remark, the result is presented in Lemma \ref{Rot}.

\begin{remark}\label{diagonal}\
\begin{enumerate}[(i)]
\item Let $A\in\PGL(2,\C(z))$ an element of finite order larger than $2$. Then $A$ is diagonalisable.
\item Two diagonal elements
$\left[
	\begin{smallmatrix}
		1 & 0\\
		0 & a
	\end{smallmatrix}
\right]$
and
$\left[
	\begin{smallmatrix}
		1 & 0\\
		0 & b
	\end{smallmatrix}
\right]$
are conjugate in $\PGL(2,\C(z))$ if and only if $a=b^{\pm1}$.
\end{enumerate}
\end{remark}

\begin{lemma}\label{Rot}
Let $A\in\mathscr G$ of order $n\neq2$. Then $A$ is conjugate to a rotation
\[R_\theta=\left[\begin{array}{cc}
1 & 0\\
0 & e^{{\bf i}\theta}
\end{array}
\right]\]
in $\mathscr G$ for some angle $\theta$.
\end{lemma}

\begin{proof}
Since $A$ is an element of finite order $n\neq2$ then by Remark~\ref{diagonal}, A is diagonalisable in $\PGL(2,\C(z))$ so there is an element $\alpha\in\PGL(2,\C(z))$ so that $A=\alpha\left[\begin{smallmatrix}
1 & 0\\
0 & \mu
\end{smallmatrix}
\right]\alpha^{-1}$ for some $\mu\in\C(z)^*$ an element of order $n$, i.e. $\mu$ is a root of unity that we can write as $\mu=e^{i\theta}$ for some angle $\theta$. 

We define $J:=\left[\begin{smallmatrix}
1 & 0\\
0 & s
\end{smallmatrix}
\right]\alpha^{-1}$ and we want to find $s\in\C(z)$ such that $J\in\mathscr G$ and $JAJ^{-1}=R_\theta$. This latter condition is fulfilled by the form of $J$.  To ask for $J\in\mathscr G$ is the same as $J$ satisfies the relation $\tau J\tau=\bar J$ which is equivalent to $\tau \left[\begin{smallmatrix}
1 & 0\\
0 & s
\end{smallmatrix}
\right]\alpha^{-1}\tau=\left[\begin{smallmatrix}
1 & 0\\
0 & \bar s
\end{smallmatrix}
\right]\overline{\alpha}^{-1}$. Multiplying to the right by $\bar\alpha$ we get $\tau \left[\begin{smallmatrix}
1 & 0\\
0 & s
\end{smallmatrix}
\right]\alpha^{-1}\tau\overline{\alpha}^{-1}=\left[\begin{smallmatrix}
1 & 0\\
0 & \bar s
\end{smallmatrix}
\right]$. We call $\rho:=\overline{\alpha}^{-1}\tau\alpha$ and we rewrite the last equation in terms of $\rho$ obtaining:
\begin{equation}\label{s}
\tau \left[\begin{smallmatrix}
1 & 0\\
0 & s
\end{smallmatrix}
\right]\bar\rho=\left[\begin{smallmatrix}
1 & 0\\
0 & \bar s
\end{smallmatrix}
\right]
\end{equation}
where $\bar\rho=\rho^{-1}$ because $\rho\bar\rho=(\bar\alpha^{-1}\tau\alpha)(\alpha^{-1}\tau\bar\alpha)=1$. 

On the other hand, the fact that $A\in\mathscr G$ i.e. $\tau A\tau =\bar A$ which is the same as $\tau \alpha\left[\begin{smallmatrix}
1 & 0\\
0 & \mu
\end{smallmatrix}
\right]\alpha^{-1}\tau=\bar\alpha\left[\begin{smallmatrix}
1 & 0\\
0 & \bar \mu
\end{smallmatrix}
\right]\bar\alpha^{-1}$ is equivalent to $\rho\left[\begin{smallmatrix}
1 & 0\\
0 & \mu
\end{smallmatrix}
\right]=\left[\begin{smallmatrix}
\mu & 0\\
0 & 1
\end{smallmatrix}
\right]\rho$ and gives the condition on $\rho$ to be of the form $\rho=\left[\begin{smallmatrix}
0 & \lambda\\
1 & 0
\end{smallmatrix}
\right]$ for some $\lambda \in\C(z)^*$. Moreover, $\rho\bar\rho=1$ implies that $\lambda\in\R(z)^*$ because $\left[\begin{smallmatrix}
0 & \lambda\\
1 & 0
\end{smallmatrix}
\right]\left[\begin{smallmatrix}
0 & \bar\lambda\\
1 & 0
\end{smallmatrix}
\right]=\left[\begin{smallmatrix}
\lambda & 0\\
0 & \bar\lambda
\end{smallmatrix}
\right]=\left[\begin{smallmatrix}
1 & 0\\
0 & 1
\end{smallmatrix}
\right]$. With this information about $\rho$, finding $s\in\C(z)^*$ satisfying the equation (\ref{s}) is equivalent to find $s$ satisfying the equation
\begin{equation}\label{lambda}
\lambda=(1-z^2)s\bar s
\end{equation}

Note that we already know that $\frac{\lambda}{1-z^2}\in\R(z)^*$, but not every element of $\R(z)^*$ can be written as $s\bar s$.
What follows is to describe $\rho$ in terms of entries of $\alpha$ and $\tau$ in order to find candidates for the value of $s$ satisfying the previous equation.
Let us present $\alpha=\left[\begin{smallmatrix}
a & b\\
c & d
\end{smallmatrix}
\right]$, then the relation $\rho=\bar\alpha^{-1}\tau\alpha$ explicitly will be
\begin{equation*}
\left[\begin{smallmatrix}
0\phantom{\bar b} & \lambda\\
1\phantom{\bar b} & 0
\end{smallmatrix}
\right]=\left[\begin{smallmatrix}
\phantom{-} \bar d & -\bar b\\
-\bar c & \phantom{-} \bar a
\end{smallmatrix}
\right]\left[\begin{smallmatrix}
0\phantom{\bar b} & 1-z^2\\
1\phantom{\bar b} & 0
\end{smallmatrix}
\right]\left[\begin{smallmatrix}
a\phantom{\bar b} & b\\
c\phantom{\bar b} & d
\end{smallmatrix}
\right]
\end{equation*}
and this gives two equations
\begin{equation}\label{alpharho}
-a\bar b+(1-z^2)c\bar d=0\text{\ and \ } (a\bar a-(1-z^2)c\bar c)\lambda=-b\bar b+(1-z^2)d\bar d
\end{equation}
When $a\neq0$, $\bar b=(1-z^2)\frac{c\bar d}{a}$ and plugging it in the second equation in (\ref{alpharho}) we get 
\[\lambda[a\bar a-(1-z^2)c\bar c]=(1-z^2)[a\bar a-(1-z^2)c\bar c](d\bar d/a\bar a)
\] hence $\lambda=(1-z^2)\frac{d\bar d}{a\bar a}$.
In the case $a=0$, equations (\ref{alpharho}) imply that $d=0$ and that 
\begin{equation*}
\lambda=\frac{1}{1-z^2}\frac{b\bar b}{c\bar c}
\end{equation*}
Then, we may choose $s=\frac{d}{a}$ when $a\neq0$ or $s=\frac{1}{1-z^2}\frac{b}{c}$ otherwise and in this way there exist $J\in\mathscr G$ such that $JAJ^{-1}=R_\theta$.
\end{proof}

\subsection{Elements in $\Aut(S(\R)/\pi)$ of finite order larger than two}\label{Sec:Aut(S(R)/pi)Order>2}
We can check that Lemma~\ref{Rot} also holds in the subgroup $\Aut(S(\R)/\pi)$, via $\psi$:
\begin{lemma}\label{RotationsinH}
Let $A\in\mathscr H$ of order $n\neq2$. Then $A$ is conjugate to a rotation
\[
R_\theta=\left[\begin{array}{cc}
1 & 0\\
0 & e^{{\bf i}\theta}
\end{array}
\right]
\]
in $\mathscr H$ for some angle $\theta$.
\end{lemma}
\begin{proof}
Let $A\in\mathscr H$ of finite order different from $2$, then by Lemma~\ref{Rot}, there is $\alpha\in\mathscr G$ such that $\alpha A\alpha^{-1}=R_\theta$. Let $\mathcal A=\psi^{-1}A\psi$. By abuse of notation, the element $\psi^{-1}\alpha\psi\in\Bir(S/\pi)$ will be called $\alpha$ as well. If $\alpha\in\Bir(S/\pi)\setminus\Aut(S(\R)/\pi)$, there is a real point $r\in S(\R)$ where $\alpha$ is not defined. The element $\alpha$ blows up this point and contracts the conic $\Gamma_{z_r}$ passing through $r$ which is a fibre of the conic bundle structure of $S$. Then $\alpha(\Gamma_{z_r})=q$ for some $q\in S(\R)$, which is sent by $R_\theta$ to a different real point ($R_\theta$ only fixes $P_N$ and $P_S$). As $\mathcal A$ preserves the fibration, $\mathcal A(\Gamma_{z_r})=\Gamma_{z_r}$, then $\alpha(\mathcal A(\Gamma_{z_r}))\neq R_\theta(\alpha(\Gamma_{z_r}))$
\end{proof}

\subsection{Involutions in $\Bir(S,\pi)\setminus\Bir(S/\pi)$}\label{Sec:InvNonTrivial}

Since we want now to study conjugacy classes of elements in $\Bir(S,\pi)\setminus\Bir(S/\pi)$ whose square is the identity, we observe that thanks to Lemma \ref{ImFinOrd}, we can think about elements of finite order in $\Bir(S,\pi)$ as the semi-direct product between elements of finite order in $\Bir(S/\pi)$ and $\Z/2\Z$ where $\Z/2\Z$ is generated by $\eta\colon S_\C\to S_\C$ sending $z$ to $-z$. The action of $\eta$ on $\Bir(S/\pi)$ is given by the map:
\begin{equation}\label{etaaction}
\begin{array}{rcrc}
\eta\colon&\PGL(2,\C(z)) & \longrightarrow & \PGL(2,\C(z))\\
&\left[\begin{smallmatrix}
a(z) & b(z)\\
c(z) & d(z) 
\end{smallmatrix}
\right] &\longmapsto& \left[\begin{smallmatrix}
a(-z) & b(-z)\\
c(-z) & d(-z) 
\end{smallmatrix}
\right]\end{array}
\end{equation}

Let $\alpha=(\alpha_0,\eta)\in\Bir(S,\pi)$ then $\alpha^2=(\alpha_0\eta(\alpha_0),1)\in\Bir(S/\pi)$ and $\eta(\alpha_0)=\alpha_0(-z)$ which means that all entries of $\alpha_0$ in $\C(z)$ are changed by the $\C$-field automorphism of $\C(z)$ sending $z$ to $-z$. We are then interested in the case $\alpha_0\eta(\alpha_0)$ is the identity.

Recall that in Lemma \ref{LemAction}(c), we identified $\Bir(S/\pi)$ with the group \[\mathscr G=\{A\in\PGL(2,\C(z))\ |\ \tau A\tau=\bar{A}\}\] where $\tau= \left[\begin{smallmatrix}
0 & 1-z^2\\
1 & 0 
\end{smallmatrix}
\right]$. We denote by $T$ the following group, \[T:=\{A\in\GL(2,\C(z))\ |\ A=\tau\bar{A}\tau^{-1}\}\subset\GL(2,\C(z))\] whose image under the canonical projection corresponds to $\mathscr G$. We have the following exact sequence where $\mathfrak p$ denotes the canonical projection:
\[1\rightarrow\R(z)^*\rightarrow T \xrightarrow\mfp \mathscr G\rightarrow 1\] Hence we obtain the cohomology exact sequence
\begin{equation}\label{ExSeqEta}
H^1(\langle\eta\rangle,T)\xrightarrow \mfp H^1(\langle\eta\rangle,\mathscr G)\xrightarrow\delta H^2(\langle\eta\rangle,\R(z)^*)
\end{equation} where $\langle\eta\rangle\simeq\Z/2\Z$ and the action of $\eta$ is described in ($\ref{etaaction}$).

The next lemma tells us that $H^1(\langle\eta\rangle,T)$ is trivial. Once that is done, the study of the map $\delta$ will show that conjugacy classes of $\alpha\in\Bir(S,\pi)\setminus\Bir(S/\pi)$ with $\alpha^2=id$ are parametrised by particular elements in $\R(z^2)$.

\begin{lemma}\label{groupT}
Let $T:=
\left\{A\in\GL(2,\C(z))\ ;\ A=\tau\overline{A}\tau^{-1}\right\}$ with $\tau$ as before. Then the group $T$ can be presented more precisely as \[T=\left\{\left[\begin{smallmatrix}
a & hb\\
\overline{b} & \overline{a}
\end{smallmatrix}
\right]\ ;\ a,b\in\C(z), a\bar{a}-hb\bar{b}\neq0\right\}\] and $H^1(\langle\eta\rangle,T)=\{1\}$.
\end{lemma}

\begin{proof}
The group $T$ is isomorphic to the multiplicative group of the non-com\-mu\-ta\-tive field $K:=\C(z)+\C(z)\xi$ where $\xi^2=h$ and $a(z)\xi=\xi\overline{a}(z)$ for any $a\in\C(z)$. The isomorphism is defined by sending an element $A=\left[\begin{smallmatrix}
a(z) & hb(z)\\
\overline{b}(z) & \overline{a}(z)
\end{smallmatrix}
\right]\in T$ to the element $a(z)+b(z)\xi\in\C(z)+\C(z)\xi$. Indeed, we have that the product in $K$, 
\[(a+b\xi)(c+d\xi)=ac+b\xi d\xi + ad\xi+b\xi c=ac+b\bar{d}h+(ad+b\bar{c})\xi\] corresponds in $T$ to the product $\left[\begin{smallmatrix}
a & hb\\
\overline{b} & \overline{a} 
\end{smallmatrix}
\right]\left[\begin{smallmatrix}
c & hd\\
\overline{d} & \overline{c}
\end{smallmatrix}
\right]=\left[\begin{smallmatrix}
ac+ b\bar{d}h & h(ad+b\bar{c})\\
\bar{a}\bar{d}+\bar{b}c & \bar{a}\bar{c}+\bar{b}dh
\end{smallmatrix}
\right].$

The corresponding action of $\langle\eta\rangle\simeq\Z/2\Z$ on $\C(z)+\C(z)\xi$ is given by the extension of the field automorphism $z\mapsto -z$ of $\C(z)^*$ to $K^*$, to be more precise, $a(z)+b(z)\xi\mapsto a(-z)+b(-z)\xi.$

Let $g\colon\langle\eta\rangle\rightarrow K^*$ be a cocycle such that $g(1)=1$ and $g(\eta)=A$ for some $A\in K^*$ such that $A\eta(A)=1$. Let $C\in K$ such that  $B=C+A\eta(C)\neq0$, such a $C$ exists because we may choose $C=A$ when $A\neq -1$, otherwise there are many choices of $C$ satisfying $C-\eta(C)\neq0$, e.g. $C=z$. We have thus $\eta(B)=\eta(C)+\eta(A)C$ and hence $A\eta(B)=A\eta(C)+A\eta(A)C=A\eta(C)+C=B$ i.e. $A=B\eta(B)^{-1}$ and this means that $A$ is a coboundary.
\end{proof}

The following Lemma will be useful to compute $H^2(\langle\eta\rangle,\R(z)^*)$.
\begin{lemma}\label{H2norm}
Let $G$ be a group with two elements acting on an abelian group $M$ and let $\xi$ be the non trivial element of $G$. 
\begin{enumerate}[$(a)$]
\item Any class $[c]\in H^2(G,M)$ admits a normalised $2$-cocycle $c'$ i.e. it is the class of $c\colon G^2\rightarrow M$ such that $c(g,1)=c(1,g)=1$ for every $g\in G$. 
\item Let $c\colon G^2\rightarrow M$ is a normalised $2$-cocycle and define $\rho(c)=c(\xi,\xi)\in M$.
Then $\rho$ induces an isomorphism of groups \[H^2(G,M)\xrightarrow\cong M^G/\{m\xi(m)\ |\ m\in M\}.\]
\end{enumerate}
\end{lemma}

\begin{lemma}
For the exact cohomology sequence $(\ref{ExSeqEta})$, 
\begin{align*}
H^2(\langle\eta\rangle,\R(z)^*)&\simeq\R(z^2)^*/ \{f\eta(f)\ |\ f\in\R(z)^*\}\\
&=\langle[-1],\{[z^2+b]:b>0\}\rangle\simeq\{\pm1\}\oplus\left(\bigoplus_{b\in\R_{>0}}\Z/2\Z\right).
\end{align*}
\end{lemma}
\begin{proof}
Let $(\R(z)^*)^{\eta}$ denote the elements of $\R(z)^*$ which are invariant with respect to the action of $\eta$ described above. We call $\mathscr{N}$ the map $\mathscr N\colon \R(z)^*\to (\R(z)^*)^{\eta}$ given by $\mathscr N(p(z))=p(z)\eta(p(z))=p(z)p(-z)$. Then by Lemma \ref{H2norm}, $H^2(\Z/2\Z,\R(z)^*)$ is isomorphic to $\coker (\mathscr N)$ that we need to compute. First, we prove that $(\R(z)^*)^{\eta}=\R(z^2)^*$. The inclusion $\R(z^2)^*\subset(\R(z)^*)^{\eta}$ is clear. Reciprocally, if $g(z)\in(\R(z)^*)^{\eta}$, $g(z)=\frac{p(z)}{q(z)}$ with $p,q\in \R[z]$ that we can assume having non common factors. Thus from $\frac{p(z)}{q(z)}=\frac{p(-z)}{q(-z)}$ follows that $p(z)q(-z)=p(-z)q(z)$ and then roots of both sides need to coincide. This implies that if $a$ is a real root of $p(z)$, it has to be a root of $p(-z)$ and therefore $z^2-a^2$ divides $p(z)$. For a complex root $\alpha$ of $p(z)$, using the same argument we obtain that $(z-\alpha)(z-\bar\alpha)(z+\alpha)(z+\bar\alpha)$ divides $p(z)$. By induction on the number of roots of $p$ and $q$, we obtain $R(z)^\eta=\R(z^2)$.

In order to compute $\coker(\mathscr N)$ we look at the image by $\mathscr N$ of generators of $\R(z)^*$ and compare with generators of $\R(z^2)^*$. Generators of $\R(z)^*$ are $a\in\R^*$, $(z-b)$ with $b\in\R$, and $(z-\alpha)(z-\bar\alpha)$ with $\alpha\in\C\setminus\R$ and they are mapped by $\mathscr N$ to $a^2$, $b^2-z^2$, and $(z^2-\alpha^2)(z^2-\bar\alpha^2)$ while generators of $\R(z^2)^*$ are $c\in\R$, $(z^2-d)$ with $d\in\R$, and $(z^2-\beta)(z^2-\bar\beta)$ with $\beta\in\C\setminus\R$ (notice that $\beta$ is always a square). Hence, $\coker(\mathscr N)\simeq\R(z^2)^*/\Ima(\mathscr N)=\langle[-1],\{[z^2+b]:b>0\}\rangle\subset\R(z^2)^*/\Ima(\mathscr N)$.

To see the structure of $H^2(\langle\eta\rangle,\R(z)^*)$, we see that $[-1]\cdot[-1]=1$ and for any $b>0$, $[z^2+b][z^2+b]=1$ because $(z^2+b)(z^2+b)=(z^2+b)\eta(z^2+b)=1$ in $\R(z^2)^*/ \{f\eta(f)\ |\ f\in\R(z)^*\}$. However, $[z^2+b][z^2+c]\neq1$ for $b,c>0$ and $b\neq c$ and $[-1][z^2+b]=-(z^2+b)\neq1$ for $b >0$.
\end{proof}

\begin{prop}\label{deltaeta}
The connecting map $H^1(\langle\eta\rangle,\mathscr G)\xrightarrow\delta H^2(\langle\eta\rangle,\R(z)^*)$ for the exact cohomology sequence $(\ref{ExSeqEta})$ corresponds to the map \[\begin{array}{rcrl}
\delta\colon & H^1(\langle\eta\rangle,\mathscr G) & \longrightarrow & \langle[-1],\{[z^2+b]:b>0\}\rangle\simeq H^2(\langle\eta\rangle,\R(z)^*)\\
& \left\{\begin{array}{c}
\text{class of } \tilde A\in \mathscr G;\\
\tilde A\eta(\tilde A)=1 
\end{array}
\right\} &\longmapsto& \left\{\begin{array}{c}
\ \text{class of } \mu\in\R(z^2);\\
 A\eta(A)=\left[\begin{smallmatrix}
\mu & 0\\
0 & \mu
\end{smallmatrix}
\right]
\end{array}
\text{and }\ \mfp(A)=\tilde A
\right\}\end{array}\]
and it is bijective.
\end{prop}

\begin{proof}
In order to study how the connecting map $\delta$ is defined, we use the Snake Lemma
(see e.\,g.\@ \cite{Neukirch2000}, Lemma 1.3.1) that in our case works as follows.
Consider the following diagram, in which $\Z_2$ stands for $\langle\eta\rangle$:
\[\xymatrix@R-0.8pc{
& {\scriptstyle C^1(\Z_2,\R(z)^*)/ B^1(\Z_2,\R(z)^*)}\  \ar[d]_{\partial_\R} \ar[r]^{i_1} & {\scriptstyle C^1(\Z_2,T)/ B^1(\Z_2,T)}  \ar[d]_{\partial_T} \ar[r]^{\mfp_1} & {\scriptstyle C^1(\Z_2,\mathscr G)/ B^1(\Z_2,\mathscr G)}\ \ar[d]_{\partial_{\mathscr G}} \ar[r] & {\scriptstyle 1}\\
{\scriptstyle 1}\ \ar[r] & {\scriptstyle Z^2(\Z_2,\R(z)^*)}\ \ar[r]^{i_2} & {\scriptstyle Z^2(\Z_2,T)}\ \ar[r]^{\mfp_2} & {\scriptstyle Z^2(\Z_2,\mathscr G)}}\]
Notice that $\delta$ is the same as the map $\ker(\partial_{\mathscr G})\xrightarrow\delta \coker(\partial_\R)$. Let $[p]\in H^1(\langle\eta\rangle,\mathscr G)$, then $p$ is a map $p\colon\langle\eta\rangle\rightarrow \mathscr G$ defined by sending $1$ to $1$ and $\eta$ to $\tilde A$ for some $\tilde A\in \mathscr G$ satisfying $\tilde A\eta(\tilde A)=1$. Since $\mfp_1$ is surjective, there is $[r]\in C^1(\langle\eta\rangle,T)/B^1(\langle\eta\rangle,T)$, this is $r\colon\langle\eta\rangle\rightarrow T$ so that $1\mapsto 1$ and $\eta\mapsto A$ where $A\in T$ is a representative of the element $\tilde A$. There is $q\in Z^2(\langle\eta\rangle,\R(z)^*)$ such that $i_2(q)=\partial_T([r])$ because $\mfp_2(\partial_T([r]))=\partial_{\mathscr G}(\mfp_1([r]))$ and $\mfp_2(\partial_T([r]))=\partial_{\mathscr G}([p])=1$ since $[p]\in\ker\partial_{\mathscr G}$ then $\partial_T([r])\in\ker\mfp_2=\Ima\ i_2$. Then $\delta$ is defined by sending $[p]$ to $[q]$ satisfying $i_2([q])=\partial_T([r])$. More explicitly, $\partial_T([r])$ is the normalised cocycle 
\[\begin{array}{rcrl}
\partial_T([r])\colon & \langle\eta\rangle\times\langle\eta\rangle & \longrightarrow & T\\
& (g_1,g_2) &\longmapsto& r(g_1)g_1(r(g_2))(r(g_1g_2))^{-1}\\
& (1,1) &\longmapsto& 1\\
& (1,\eta) &\longmapsto& 1\\
& (\eta,1) &\longmapsto& 1\\
& (\eta,\eta) &\longmapsto& A\eta(A)
\end{array}\]
Thus, $A\eta(A)=\left[\begin{smallmatrix}
\mu & 0\\
0 & \mu
\end{smallmatrix}
\right]$ with $i_2([q])(\eta,\eta)=\mu\in\R(z^2)^*$. Summing up, $\delta$ corresponds to the map
\[\begin{array}{rcrl}
\delta\colon & H^1(\langle\eta\rangle,\mathscr G) & \longrightarrow &  H^2(\langle\eta\rangle,\R(z)^*)\\
& \left\{\begin{array}{c}
\tilde A\in \mathscr G;\\
\tilde A\eta(\tilde A)=1 
\end{array}
\right\} &\longmapsto& \left\{\begin{array}{c}
\ \ \ \ \ \ \ \ \ \ \mu\in\R(z^2);\\
 A\eta(A)={\scriptsize \left[\begin{smallmatrix}
\mu & 0\\
0 & \mu
\end{smallmatrix}
\right]}
\end{array}
\text{and }\ \mfp(A)=\tilde A
\right\}.\end{array}\]

Let us see that the map $\delta$ is surjective: the element $\left[\begin{smallmatrix}
{\bf i} & 0\\
0 & -{\bf i}
\end{smallmatrix}
\right]$ is mapped by $\delta$ to the class $[-1]$. When $c\in\R_{>0}$, the element $\left[\begin{smallmatrix}
{\bf i}(z-{\bf i}\sqrt{c})&0\\
0&-{\bf i}(z+{\bf i}\sqrt{c})
\end{smallmatrix}\right]$ is sent by $\delta$ to the class $[z^2+c]$. Given any finite product of classes $\gamma=(z^2+c_1)\cdots (z^2+c_k)$ in $H^2(\langle\eta\rangle,\R(z)^*)$ with $c_i>0$ for $1\leq i\leq k$, the diagonal elements of the form $\left[\begin{smallmatrix}
a(z) & 0\\
0 & \bar a(z)
\end{smallmatrix}
\right]$ where 
\[a(z)=\begin{cases}
{\bf i}(z-{\bf i}\sqrt{c_1})(z-{\bf i}\sqrt{c_2})\cdots (z-{\bf i}\sqrt{c_k}), & \text{if } k \text{ is odd} \\
(z-{\bf i}\sqrt{c_1})(z-{\bf i}\sqrt{c_2})\cdots (z-{\bf i}\sqrt{c_k}), & \text{if } k \text{ is even}
\end{cases} 
\] is mapped to $\gamma$. This proves the surjectivity of the application $\delta$.

In order to prove injectivity, we will show that any class $\tilde A=\left[\begin{smallmatrix}
a(z) & hb(z)\\
\bar b(z) & \bar a(z)
\end{smallmatrix}
\right]$  in $H^1(\Z/2\Z,\mathscr G)$ is equivalent to a diagonal element $D$ of the form $\left[\begin{smallmatrix}
x(z) & 0\\
0 & \bar x(z)
\end{smallmatrix}
\right]$. In other words, we want to show that we can find an element $\alpha=\left[\begin{smallmatrix}
c(z) & hd(z)\\
\bar d(z) & \bar c(z)
\end{smallmatrix}
\right]$ in $\mathscr G$ such that $\eta(\alpha)A\alpha^{-1}=D$ where $A$ is the representative of $\tilde A$ in $T$. This leads to the following equation
\[\bar c(z)(\bar d(-z)a(z)+\bar c(-z)\bar b(z))-\bar d(z)(h\bar d(-z)b(z)+\bar c(-z)\bar a(z))=0\] which is equivalent to 
\begin{equation}\label{alph}
\frac{\bar c(z)}{\bar d(z)}=\frac{\bar a(z)\frac{\bar c(-z)}{\bar d(-z)}+hb(z)}{\bar b(z)\frac{\bar c(-z)}{\bar d(-z)}+a(z)}
\end{equation}
We call $\Psi$ the following automorphism of $\Pp^1_{\C(z)}$ defined by
\[\begin{array}{rcrc}
\Psi\colon & \Pp^1_{\C(z)} & \longrightarrow &  \Pp^1_{\C(z)}\\
& (r(z)\colon s(z)) &\longmapsto& (\bar a(z)r(z)+hb(z)s(z)\colon \bar b(z)r(z)+a(z)s(z))
\end{array}.\] The equation (\ref{alph}) can be seen as $f(z)=\Psi(f(-z))$ for $f(z)=\frac{\bar c(z)}{\bar d(z)}$. In this way,
finding $c(z)$ and $d(z)$ satisfying the equation (\ref{alph}) is equivalent to find fixed points of $\widetilde\Psi$ where $\widetilde\Psi(f(z)):=\Psi(f(-z))$. First we notice that the automorphism $\Psi$ is a linear automorphism given by the element $\left[\begin{smallmatrix}
\bar a(z) & hb(z)\\
\bar b(z) & a(z)
\end{smallmatrix}
\right]$ in $\PGL(2,\C(z))$ that we denote by $\Hat A$ since it comes from $A$ by interchanging the elements of the mean diagonal, this implies that $\widetilde{\Psi}$ has order two because $\widetilde{\Psi}\circ\widetilde{\Psi}=id$ is equivalent to $\Hat A\eta(\Hat A)=1$ which is satisfied because $A$ is a class in $H^1(\langle\eta\rangle,\mathscr G)$. On the other hand, the element $\Hat A$ is equivalent to $\Hat{\Hat A}=\left[
\begin{smallmatrix}
0&-\det\hat A\\
\bar b(z)^2&0
\end{smallmatrix}\right]$ since $B^{-1}\Hat A\eta(B)=\Hat{\Hat A}$ for $B=\left[\begin{smallmatrix}
1&\bar a(z)/\bar b(z)\\
0&1
\end{smallmatrix}\right]$. Hence, the existence of fixed points for the automorphism associated to $\Hat{\Hat A}$ gives the existence of fixed points for the automorphism $\widetilde{\Psi}$. Then we look explicitly for elements $u,v\in\C(z)$ such that $(u(z)\colon v(z))=\Hat{\Hat A}(u(-z)\colon v(-z))=(-\det\Hat A v(-z)\colon \bar b(z)^2u(-z))$ in $\Pp^1_{\C(z)}$ i.e. $u(z)u(-z)\bar b(z)^2=-v(z)v(-z)\det\Hat A$ and then $\frac{u(z)}{v(z)}\frac{u(-z)}{v(-z)}=-\frac{\det\Hat A}{\bar b(z)^2}$. The right side of this last equation belongs to $\C(z^2)$ because $\det \Hat A=\det A$ which belongs to $\R(z^2)$ and $\bar b(z)^2\in\C(z^2)$ condition imposed by the fact that $A$ is a class in $H^1(\langle\eta\rangle,\mathscr G)$. Existence of $u$ and $v$ comes from the next Lemma.
\end{proof}

\begin{lemma}\label{C(z^2)guys}
Any element $f\in\C(z^2)$ can be written as the product $g(z)g(-z)$ for some element $g\in\C(z)$. In other words, \[\C(z^2)=\{g(z)g(-z)\ :\ g(z)\in\C(z)\}.\]
\end{lemma}
\begin{proof}
Clearly, for $g\in\C(z)$ it follows that $g(z)g(-z)\in\C(z^2)$. Reciprocally, let $f\in\C(z^2)$. Thus $f=\frac{p(z)}{q(z)}$ with $p,q\in\C[z^2]$. We can write $p$ in terms of roots as $p(z)=\alpha(z^2-\alpha_1)\cdots(z^2-\alpha_s)$ where $\alpha,\alpha_i\in\C$, $1\leq i\leq s$. Any factor of $p$ can be decomposed as a product of the form $-(z-\sqrt{\alpha_i})(-z-\sqrt{\alpha_i})$ for any root $\alpha_i$. We can then write $p$ as the product $g_1(z)g_1(-z)$ where 
\[g_1(z)=\begin{cases}
\sqrt{\alpha}(z-\sqrt{a_1})\cdots (z-\sqrt{a_r})(z-\sqrt{\alpha_1})\cdots(z-\sqrt{\alpha_s}), & \text{\!\!if } s \text{ is even} \\
{\bf i}\sqrt{\alpha}(z-\sqrt{a_1})\cdots (z-\sqrt{a_r})(z-\sqrt{\alpha_1})\cdots(z-\sqrt{\alpha_s}), & \text{\!\!if } s \text{ is odd}.
\end{cases} 
\]
In the same way, $q(z)=g_2(z)g_2(z)$ and therefore, $f$ can be presented as the product $\frac{g_1(z)}{g_2(z)}\cdot\frac{g_1(-z)}{g_2(-z)}$.
\end{proof}

\begin{corollary}[from Proposition \ref{deltaeta}]\label{invactonbasis}
The conjugacy classes of elements $\alpha=(\alpha_0,\eta)\in\Bir(S,\pi)\setminus\Bir(S/\pi)$ such that $\alpha_0\eta(\alpha_0)$ is the identity are parametrised by the classes of polynomials $\langle[-1],\{[z^2+b]:b>0\}\rangle\simeq H^2(\langle\eta\rangle,\R(z)^*)$.
\end{corollary}
\begin{proof}
The cohomology group  $H^1(\langle\eta\rangle,\mathscr G)$ corresponds precisely to the
set of conjugacy classes of involutions in $\Bir(S,\pi) \setminus \Bir(S/\pi)$,
that is, classes of elements $(\alpha_0,\eta)$ as in the statement.
Therefore Proposition \ref{deltaeta} directly implies the corollary.
\end{proof}

\begin{corollary}\label{ColPropInv}
The set of conjugacy classes of involutions in
$\Aut(S(\R),\pi) \setminus \allowbreak \Aut(S(\R)/\pi)$
surjects naturally to the set of conjugacy classes of involutions in
$\Bir(S,\pi) \setminus \Bir(S/\pi)$.
\end{corollary}
\begin{proof}
Let $(A,\eta)$ be an involution in $\Bir(S,\pi) \setminus \Bir(S/\pi)$.
The proof of Proposition \ref{deltaeta} shows that $(A,\eta)$ is conjugate
to an element $(\tilde A,\eta)$ where $\tilde A$ is, via $\psi$, an element
of the form 
$\left[
   \begin{smallmatrix}
      a(z) & 0\\
      0 & \bar a(z)
   \end{smallmatrix}
\right]$,
and $a \in \C[z]$ has no real roots.
Since in that case
$a\bar{a} \in \R[z]_+$,
Proposition~\ref{diffG} tells us that such an element corresponds to one of $\Aut(S(\R)/\pi)$.
Hence the birational diffeomorphism
$(\tilde A, \eta) \in \Aut(S(\R),\pi) \setminus \allowbreak \Aut(S(\R)/\pi)$
is conjugate in $\Bir(S,\pi)$ to $(A,\eta)$, and therefore every conjugation class of
$\Bir(S,\pi) \setminus \allowbreak \Bir(S/\pi)$
contains a conjugation class of
$\Aut(S(\R),\pi) \setminus \allowbreak \Aut(S(\R)/\pi)$.
\end{proof}

\section{Connection between families}\label{Ch:connection}
In this section, we collect all our results, and use the fixed points and the classification of the possible Sarkisov links given by Iskovskikh in \cite{Isk96} to give the proofs of Theorem~\ref{MainThm} and Theorem~\ref{MainThm2} (Section~\ref{Ch:results}).

We start with some definitions, which come from the equivariant Sarkisov program.
\begin{definition}\label{MoriFib}
Let $X$ be a smooth projective real rational surface with $X(\R)\simeq S(\R)$, let $g\in \Aut(X)$ be an automorphism of finite order and let $\mu\colon X\to Y$ be a morphism. 

The triple $(X,g,\mu)$ is said to be a \emph{Mori fibration} when one of the following holds 
\begin{enumerate}[$(i)$]
\item $\rk(\Pic(X)^{g})=1$, $Y$ is a point and $X$ is a Del Pezzo surface;
\item $\rk(\Pic(X)^{g})=2$, $Y=\Pp^1$ and the map $\mu$ is a conic bundle.
\end{enumerate}
\end{definition}
\begin{remark}\label{ConicBAutSpi}
In the second case, we can do as in Proposition~\ref{MinMod} and find a birational morphism $\varepsilon\colon X\to S$ that restricts to a diffeomorphism $X(\R)\to S(\R)$, such that $\pi\varepsilon=\alpha\mu$, for some $\alpha\in \Aut(\Pp^1_\R)$. This conjugates $g$ to an element $\varepsilon g \varepsilon^{-1}\in \Aut(S(\R),\pi)$. The possible choices for $\varepsilon$ just replace $\varepsilon g \varepsilon^{-1}$ with a conjugate in the group $\Aut(S(\R),\pi)$. 
\end{remark}
\begin{definition}
Let $\mu\colon X\to Y$ and $\mu'\colon X'\to Y'$, $g\in \Aut(X), g'\in \Aut(X')$ be two Mori-fibrations. An \emph{isomorphism of Mori fibrations} is an isomorphism $\rho \colon X\to X'$, such that $g'\rho=\rho g$ and $\mu'\rho=\alpha\mu$ for some isomorphism $\alpha\colon Y\to Y'$.
\end{definition}

\begin{definition}
A \emph{Sarkisov link} between two Mori fibrations $\mu\colon X\to Y$ and $\mu'\colon X'\to Y'$, $g\in \Aut(X), g'\in \Aut(X')$ is a birational map $\zeta\colon X\dasharrow X'$ such that $g'\zeta=\zeta g$ and is of one of the following four types,
\begin{enumerate}[$(i)$]
\item \emph{Links of type $\mathrm{I}$}. These are commutative diagrams of the form
\[\xymatrix@R-0.8pc{
X \ar[d]_\mu \ar@{-->}[r]^\zeta & X' \ar[d]^{\mu'}  \\
Y=\{p\} & Y'=\Pp^1 \ar[l]_{\rho}}\]
where $\zeta^{-1}\colon X'\to X$ is a birational morphism, which is the blow-up of either a $g$-orbit of real points or imaginary conjugate points of $X$, and where $\rho$ is the contraction of $Y'=\Pp^1$ to the point $p$.
\item \emph{Links of type $\mathrm{II}$}. These are commutative diagrams of the form
\[\xymatrix@R-0.8pc{
X \ar@/^1pc/ @{-->}[rr]^{\zeta} \ar[d]_\mu & Z \ar[r]_{\beta'} \ar[l]^{\beta} & X' \ar[d]^{\mu'} \\
Y \ar[rr]^{\simeq}_{\rho}  &  & Y'}\]
where $\beta\colon Z\to X$ (respectively $\beta'\colon Z\to X'$) is a birational morphism, which is the blow-up of either a $g$-orbit (respectively $g'$-orbit) of real points or imaginary conjugate points of $X$  (respectively of $X'$), and where $\rho$ is an isomorphism between $Y$ and $Y'$.
\item \emph{Links of type $\mathrm{III}$}. (These are the inverse of the links of type I). These are commutative diagrams of the form
\[\xymatrix@R-0.8pc{
X \ar[d]_\mu \ar[r]^\zeta & X' \ar[d]^{\mu'}  \\
Y=\Pp^1 \ar[r]^{\rho}& Y'=\{p\}}\]
where $\zeta\colon X\to X'$ is a birational morphism, which is the blow-up of either a $g'$-orbit  of real points or imaginary conjugate points of $X'$, and where $\rho$ is the contraction of $Y=\Pp^1$ to the point $p$.
\item \emph{Links of type $\mathrm{IV}$}. These are commutative diagrams of the form
\[\xymatrix@R-0.8pc{
X \ar[d]_\mu \ar[r]^\zeta_{\simeq} & X' \ar[d]^{\mu'}  \\
Y=\Pp^1 & Y'=\Pp^1}\]
where $\zeta\colon X\to X'$ is an isomorphism and $\mu$, $\mu'\circ\zeta$ are conic bundles on $X'$ with distinct fibres.
\end{enumerate}
\end{definition}
The following result is given in \cite[Theorem 2.5]{Isk96}
\begin{thm}\label{thmIsk}
Let $(X,g,\mu)$ and $(X',g',\mu')$ be two Mori-fibrations. Every birational map $\rho\colon X\dasharrow X'$ such that $g'\rho=\rho g$ decomposes into elementary links and isomorphisms of conic bundles. 
\end{thm}
Looking at the classification of links of \cite{Isk96}, we obtain the following lemma with the links that could be possible to have in our classification problem.

\begin{lemma}\label{links}
Let $(X,g,\mu)$ and $(X',g',\mu')$ be two Mori-fibrations, and let $\rho\colon X\dasharrow Y$ be a birational map which restricts to a diffeomorphism $X(\R)\to Y(\R)$. Then, $\rho$ decomposes into elementary links that blow up only imaginary points and contract only imaginary curves, and are of the following type:
\begin{enumerate}[$(a)$]
\item Links of type $\mathrm{II}$ between conic bundles, which correspond therefore to a conjugation in $\Aut(S(\R),\pi)$.
\item Links of type $\mathrm{II}$ of the form $ X\dasharrow X$, where $X$ is either the sphere $S$ or a Del Pezzo surface of degree $4$. Moreover, the two elements of $\Aut(X)$ corresponding to this link are conjugate in $\Aut(X)$.
\item Link of type $\mathrm{I}$ and $\mathrm{III}$ between the sphere $S$ and the Del Pezzo surface of degree $6$ obtained by blowing up two conjugate points on $S$. These are possible for only a few of elements, given in Lemma~$\ref{Aut(dP6)}$.
\item Links of type $\mathrm{IV}$ on Del Pezzo surfaces of degree $2$ or $4$, obtained by blowing up pairs of conjugate points in $S$. 

If the two elements of $\Aut(S(\R),\pi)$ corresponding to the link are not conjugate, then $X$ is a Del Pezzo surface of degree $4$ and the two automorphisms are $g_1,g_2\in \Aut(X)$ described in Lemma~$\ref{Lem:g1g2DP4}$.
\end{enumerate}
\end{lemma}
\begin{proof}
It follows from Proposition~\ref{p} that $X$, $X'$ do not contain any real $(-1)$-curve. Moreover, the map $\rho$ has no real base-points implying that the first Sarkisov link obtained in the decomposition does not have real base-points (the base-points of the link are taken among the base-points of the map, see the proof of \cite[Theorem 2.5]{Isk96}). Proceeding by induction on the number of links provided by Theorem~\ref{thmIsk}, we obtain that $\rho$ decomposes into Sarkisov links that do not blow up any real point or contract any real curve. In particular, the surfaces obtained are all diffeomorphic to the sphere and with $K^2_X\in 2\mathbb{Z}$.

It remains to study links $X\dasharrow X'$, between two Mori-fibrations $\mu\colon X\to Y$ and $\mu'\colon X'\to Y'$, $g\in \Aut(X)$, $g'\in \Aut(X')$, such that $X(\R)\simeq X'(\R) \simeq S(\R)$, with $(K_X)^2$, $(K_{X'})^2\in 2\mathbb{Z}$, and which do not blow up any point. In the case where $Y$ is a point, we can moreover assume that $(K_X)^2\not=6$, by Proposition~\ref{dP6nominimal} (and similarly $(K_{X'})^2\not=6$ if $Y'$ is a point). Looking at the list of \cite[Theorem $2.6$]{Isk96}, we get the following possibilities.
\begin{enumerate}
\item Links of type $\mathrm{I}$ and $\mathrm{III}$ ($Y$ is a point and $Y'=\Pp^1$ or vice versa). Looking at \cite[Theorem $2.6$, case $(i)$]{Isk96}, one gets only one possibility, which is the blow-up of two imaginary conjugate points on the sphere $S$. Up to automorphism, these points can be taken to be the two base-points of $\pi\colon S\dasharrow \Pp^1$, and the automorphisms that preserve the union of these two points are described in  Lemma~$\ref{Aut(dP6)}$.
\item Links of type $\mathrm{II}$ ($Y=Y'=\Pp^1$ or $Y=Y'$ is a point).

In the first case, when $Y=Y'=\Pp^1$, the link corresponds to conjugation in the group $\Aut(S(\R),\pi)$ (see Remark~\ref{ConicBAutSpi}).

 In the second case, the list of \cite[Theorem $2.6$, case $(ii)$]{Isk96} yields the following three possibilities:
 
 \begin{enumerate}[$(i)$]
 \item
(Case $(K_X)^2=8,(b)$)
 A birational map $S(\R)\dasharrow S(\R)$ that blows up $3$ pairs of conjugate points and contract $3$ pairs of conjugate curves. It corresponds to the Geiser involution on the blow-up of the $6$ points. 
 \item
(Case $(K_X)^2=8,(d)$) A birational map $S(\R)\dasharrow S(\R)$ that blows up $2$ pairs of conjugate points and contract $2$ pairs of conjugate curves. 
 \item
(Case $(K_X)^2=4,(b)$)
 A birational map $X(\R)\dasharrow X(\R)$ that blows up $2$ pairs of conjugate points on a Del Pezzo surface $X$ of degree $4$ and contract $2$ pairs of conjugate curves. It corresponds to the Geiser involution on the blow-up of the $4$ points. 
 \end{enumerate} In each case we get a link $X\dasharrow X$, where $X$ is either the sphere $S$ or a Del Pezzo surface of degree $4$. It remains to see that the two automorphisms of prime order of $\Aut(X)$ produced by this link are conjugate by an element of $\Aut(X)$. If the link corresponds to a Geiser involution, this is because the Geiser involution commutes with all automorphism of the surface (see Proposition~\ref{Prop:Geiser}). In the other case, the orbit blown up consists of two pairs of conjugate points on $S(\C)$, so the automorphism is an element of order $2$ in $\Aut(S)$, so conjugate to a rotation, a reflection or the antipodal involution (Proposition~\ref{ConjAut(S)}). By looking at the fixed points, we observe that two elements of order $2$ in $\Aut(S)$ are conjugate in $\Aut(S)$ if and only if they are conjugate in $\Aut(S(\R))$.
\item Links of type $\mathrm{IV}$.  ($X\simeq X'$ is a surface which admits two different conic bundle structures, and the link consists of changing the structure). It follows from \cite[Theorem $2.6$, case $(iv)$]{Isk96} that $(K_X)^2\in \{2,4,8\}$. The case $8$ is not possible since $\Pic(S)\cong \mathbb{Z}$. If $(K_X)^2=2$, the link is given by the Geiser involution (by \cite[Theorem $2.6$]{Isk96}), which commutes with all automorphisms. Hence, the two automorphisms of $\Aut(S(\R),\pi)$ provided by the links are conjugate. This is the same if $(K_X)^2=4$ and if there is an element of $\Aut(S)$ which commutes with the automorphism. By Lemma~\ref{Lem:DP4Links}, the only remaining case is when the two automorphisms are $g_1,g_2$ given in Lemma~$\ref{Lem:g1g2DP4}$.\qedhere
\end{enumerate}
\end{proof}
Lemma~\ref{links} shows that the automorphisms $g_1,g_2$ given in Lemma~\ref{Lem:g1g2DP4} are quite special. The following result describes the situation.
\begin{lemma}\label{Lem:SpecialMapsg1g2}
\begin{enumerate}[$(1)$]
\item
Let $X$ be a Del Pezzo surface of degree $4$ with $\mu\in \C\setminus\{\pm1\}$, $\lvert\mu\rvert =1$ $($see Lemma~$\ref{points})$, and $g_1,g_2\in \Aut(X)$ be the automorphisms given in Lemma~\ref{Lem:g1g2DP4}. The action on the two conic bundles invariant yields two involutions \[g_1'(\mu)\in \Aut(S(\R)/\pi),\ g_2'(\mu)\in \Aut(S(\R),\pi)\setminus \Aut(S(\R)/\pi)\] given by 
\[\begin{array}{rccccc}
g_1'(\mu)&\colon& (t,z)&\dasharrow& \left(\frac{-2{\bf i}\mu t+(1+\mu)(1-z^2)}{\mu(2{\bf i}+(1+\mu)t)},z\right)\\
g_2'(\mu)&\colon &(t,z)&\dasharrow &\left(\frac{(1-z^2)({\bf i}t(1+\mu)-2)}{-2\mu t-{\bf i}(1+\mu)(1-z^2)},-z\right)\end{array}\]
$($using the map $\psi\colon S_\C\dasharrow \A^2_\C$ of Lemma~$\ref{LemAction})$
\item
Taking another surface given by $\mu'\in \C\setminus\{\pm1\}, \lvert \mu'\rvert=1$, the following are equivalent:
\begin{enumerate}
\item
$g_1'(\mu)$ and $g_1'(\mu')$ are conjugate in $\Aut(S(\R),\pi)$;
\item
$g_2'(\mu)$ and $g_2'(\mu')$ are conjugate in $\Aut(S(\R),\pi)$;
\item
$\mu'=\mu^{\pm 1}$.
\end{enumerate}
\item
Let $g\in \Aut(S(\R)/\pi)$ be an element of order $2$, such that $\Fix(g)$ is a rational curve with no real point. Then, $g$ is conjugate in $\Aut(S(\R),\pi)$ to $g_1'(\mu)$ for some $\mu\in \C\setminus\{\pm1\}$, $\lvert \mu\rvert=1$.
\end{enumerate}
\end{lemma}
\begin{proof}Let $g\in \Aut(S(\R)/\pi)$ be an element of order $2$, such that $\Fix(g)$ is a rational curve with no real point.
The element $g$ belongs to $\Aut^{+}(S(\R)/\pi)$, and the map $\pi$ restricts to a double covering $\pi_g\colon \Fix(g)\to \Pp^1$ (Proposition~\ref{diff}). Since the curve is rational, by the Riemann-Hurwitz formula the double covering is ramified over two points $q,\bar{q}\in \mathbb{P}^1(\C)$. These two points determine the curve $\Fix(g)$, up to isomorphisms above $\Pp^1(\C)$, i.e. isomorphisms $\rho\colon \Fix(g)\to \Fix(g')$ with $\pi_{g'}\rho=\pi_g$. Hence, by Theorems~\ref{thm} and~\ref{thm4}, the conjugacy class of $g$ in $\Aut(S(\R)/\pi)$ is given by the set $\{q,\bar{q}\}$.

We will use this observation to show that $g$ is conjugate to one of the automorphisms $g_1$, $g_2\in \Aut(X)$, where $X$ is a Del Pezzo surface of degree $4$, given in Lemma~\ref{Lem:g1g2DP4}.

We use the map $\psi\colon S_\C \dasharrow \A^2_\C$, $(x,y,z)\dasharrow (x-{\bf i} y,z)$ given in Lemma~\ref{LemAction} to compute the action of $g_1$, $g_2$ on $\A^2_\C$. 
Note that $\psi\varphi^{-1}\colon \Pp^1_\C\times\Pp^1_\C\dasharrow \A^2_\C$ is locally given by 
\[((1:s),(1:v))\dasharrow \left(\frac{-2{\bf i}s}{sv+1}, \frac{1-sv}{1+sv}\right),\]
and its inverse is $(t,z)\dasharrow ((z+1:{\bf i}t),(t:{\bf i}(z-1)))$. Using the explicit description of  Lemma~\ref{Lem:g1g2DP4}, the actions of $g_1$, $g_2$ are then respectively given by 
\[g_1'(\mu)\colon (t,z)\dasharrow \left(\frac{-2{\bf i}\mu t+(1+\mu)(1-z^2)}{\mu(2{\bf i}+(1+\mu)t)},z\right)\]
\[g_2'(\mu)\colon (t,z)\dasharrow \left(\frac{(1-z^2)({\bf i}t(1+\mu)-2)}{-2\mu t-{\bf i}(1+\mu)(1-z^2)},-z\right)\]
These correspond to involutions $g_1'(\mu)\in \Aut(S(\R)/\pi)$ and $g_2'(\mu)\in \Aut(S(\R),\pi)\setminus\allowbreak \Aut(S(\R)/\pi)$, which are conjugate by an element which is in the group $\Aut(S(\R))\setminus\allowbreak \Aut(S(\R),\pi)$ (see Lemma~\ref{Lem:g1g2DP4}).

In order to show that there exists $\mu$ such that $g$ is conjugate to $g_1'(\mu)$ in $\Aut(S(\R),\pi)$, we need to compute the ramification points of $\Fix(g_1'(\mu))$. The curve of fixed points of $g_1'(\mu)$ is given by
\[\mu(1+\mu)t^2+4{\bf i}\mu t-(1-z^2)(1+\mu)=0\]
so its discriminant with respect to $t$ is equal to
\[-4\mu (\mu+1)^2 \cdot \left(z^2-\left(\frac{\mu-1}{\mu+1}\right)^2\right),\]
and the two points correspond then to $z=\pm \frac{\mu-1}{\mu+1}$. We conjugate $g$ with an automorphism of the form 
\[g_b\!:\!(x,y,z)\mapsto\left(x\frac{\sqrt{1-b^2}}{bz+1},y\frac{\sqrt{1-b^2}}{bz+1},\frac{z+b}{bz+1}\right)\]
for some $b\in (-1,1)$ (see Lemma~\ref{imphigrande}), and claim that we can send the points $q,\bar{q}$ onto $\pm \frac{\mu-1}{\mu+1}$ for some $\mu\in \C\setminus \{\pm 1\}$ with $\lvert \mu\rvert=1$. To see this, we make the change of coordinates $z=\frac{1-z'}{1+z'}$, $z'=\frac{1-z}{1+z}$, so that the map $g_b$ acts as $z'\mapsto z'\frac{1-b}{1+b}$ and the points $z=\pm \frac{\mu-1}{\mu+1}$ correspond to $z'=\mu^{\pm 1}$. The claim follows then from the fact that the map $b\mapsto \frac{1-b}{1+b}$ yields a bijection $(-1,1)\to \R_{>0}$. Hence $g$ is conjugate to $g_1'(\mu)$ for some $\mu$.

Let us show that $g_1'(\mu)$ is conjugate to $g_1'(\mu')$ in $\Aut(S(\R),\pi)$ if and only if $\mu'=\mu^{\pm 1}$. First, observe that $\frac{1/\mu-1}{1/\mu+1}=\frac{1-\mu}{1+\mu}$, so the pair of points are the same for $\mu$ and $\mu^{-1}$. Hence, $g_1'(\mu)$ is conjugate to $g_1'(\mu')$ in $\Aut(S(\R),\pi)$. Second, if $g_1'(\mu')$ is conjugate to $g_1'(\mu)$, there exists an element of $\Aut(S(\R),\pi)$ whose action on $\Pp^1$ sends  $\left\{\pm \frac{\mu-1}{\mu+1}\right\}$ onto $\left\{\pm \frac{\mu'-1}{\mu'+1}\right\}$. But the action is generated by the maps $z\mapsto \frac{z+b}{bz+1}$, $b\in (-1,1)$ and by $z\to -z$ (Lemma~\ref{imphigrande}). Making the same change of coordinates as before, we obtain that $\mu'=\mu^{\pm 1}$.

To finish the proof, it remains to see that two elements $g_2'(\mu)$ and $g_2'(\mu')$ are conjugate in $\Aut(S(\R),\pi)$ if and only if $\mu'={\mu}^{\pm 1}$. The element $g_2'(\mu)$ corresponds to an element of $H^2(\langle\eta\rangle,\R(z)^*)$ that we can compute using Proposition~\ref{deltaeta}. To do this, we need to write the corresponding element of $H^1(\langle \eta\rangle,\mathscr G)$. Composing $g_2'(\mu)$ with $(t,z)\to (t,-z)$, we obtain the element of $\tilde{A}=\mathscr G$ given by
\[\left[\begin{smallmatrix}
-{\bf i}(1+\mu)(1-z^2) & 2(1-z^2)\\
2\mu & {\bf i}(1+\mu)(1-z^2)
\end{smallmatrix}
\right].\]
In order to get an element of $T\subset \GL(2,\C(z))$ (see Lemma~\ref{groupT}), we divide each element of the matrix with $\nu$, with $\nu\in \C$, $\lvert \nu\rvert=1$, $\nu^2=\mu$, and get
\[\left[\begin{smallmatrix}
a & hb\\
\overline{b} & \overline{a}
\end{smallmatrix}
\right]\in T\subset\GL(2,\C(z)),\]
with $a=-\frac{{\bf i}{(1+\mu)(1-z^2)}}{\nu}$, $b=\frac{2}{\nu}$ (indeed, $\overline{a}=\frac{{\bf i}{(1+1/\mu)(1-z^2)}}{1/\nu}=\frac{{\bf i}(1+\mu)(1-z^2)}{\nu}$). Obser-ving that $\overline{a}=-a$ and that $a,b$ are invariant by $z\mapsto -z$, the corresponding element of $H^2(\langle \eta\rangle,\R(z)^*)$ can be computed (using Proposition~\ref{deltaeta}) by
\[\left[\begin{smallmatrix}
a & hb\\
\overline{b} & \overline{a}
\end{smallmatrix}\right]^2=\left[\begin{smallmatrix}
a^2+b\overline{b}h & 0\\
0 & a^2+b\overline{b}h
\end{smallmatrix}\right]\]
and corresponds therefore to 
{\small \[a^2+b\overline{b}h=(1-z^2)\left(z^2-\left(\frac{1-\mu}{1+\mu}\right)^2\right)\frac{(1+\mu)^2}{\mu}.\]}
Writing $\mu=\cos(\theta)+{\bf i} \sin(\theta)$ we obtain $\frac{(1+\mu)^2}{\mu}=2(\cos(\theta)+1),\left(\frac{1-\mu}{1+\mu}\right)^2=\frac{\cos(\theta)-1}{\cos(\theta)+1}=\frac{\cos^2(\theta)-1}{(\cos(\theta)+1)^2}\in \R_{<0}$, so the corresponding element of $H^2(\langle \eta\rangle,\R(z)^*)$ is the class of $z^2+\frac{1-\cos(\theta)}{\cos(\theta)+1}$. Denoting by $s\colon (0,\pi)\cup (\pi,2\pi)\to \R_{>0}$ the map $s(\theta)= \frac{1-\cos(\theta)}{\cos(\theta)+1}$, we observe that $s(\theta)=s(\theta')$ if and only if $\theta'\in \{\theta, 2\pi-\theta\}$. This gives the result.
\end{proof}
\subsection{Proof of theorems \ref{MainThm} and \ref{MainThm2}}\label{Sec:Proofs}
We can now finish by giving the proof of the main theorems.
\begin{proof}[Proof of Theorem $\ref{MainThm}$]
Let $g\in\Aut(S(\R))$ be of prime order. By Proposition~\ref{MinMod}, one of the two following possibilities holds
\begin{enumerate}[$(a)$]
\item There exists a birational morphism $\varepsilon\colon X\to S$ which is the blow-up of $0$, $1$, $2$, or $3$ pairs of conjugate imaginary points in $S$, such that $\hat g=\varepsilon^{-1}g\varepsilon\in\Aut(X)$, $\Pic(X)^{\hat g}\cong \Z$, and $X$ is a Del Pezzo surface.
\item There exists $\alpha\in \Aut(\Pp^1)$ such that $\alpha\pi=\pi g$. Moreover, there exists a birational morphism $\varepsilon\colon X\to S$ that restricts to a diffeomorphism $X(\R)\to S(\R)$ such that $\hat g=\varepsilon^{-1}g\varepsilon\in\Aut(X)$, $\pi\varepsilon\colon X\to\Pp^1$ is a conic bundle on $X$, and $\Pic(X)^{\hat g}\cong\Z^2$.
\end{enumerate}
In particular, we have a Mori fibration in the sense of Definition~\ref{MoriFib}.

In the case $(a)$, $X$ is a Del Pezzo surface with possible degree $8,6,4$, or $2$. If $(K_X)^2=8$, $X\simeq S$ and $g\in\Aut(S)$. By Proposition~\ref{ConjAut(S)}, $g$ is conjugate to one of the cases (3), (4), or (5) of the statement.
If $X$ is a Del Pezzo surface of degree $6$, $X$ comes from $S$ by blowing up a pair of conjugate imaginary points and Proposition~\ref{dP6nominimal} tell us that $\hat g$ comes from an automorphism of $S$, having the same cases as before.
If $X$ is a Del Pezzo surface of degree $4$, $X$ comes from $S$ by blowing up two pairs of conjugate imaginary points and by Proposition~\ref{dP4rank1} $g$ is conjugate to $\alpha_1$ or $\alpha_2$ giving in case (2).
If $X$ is a Del Pezzo surface of degree $2$, $X$ comes from $S$ by blowing up three pairs of conjugate imaginary points and Lemma~\ref{numinimal} asserts that the Geiser involution $\nu$ is such that $\Pic(X)^\nu$ has rank $1$ and Lemma~\ref{dP2nonrank1} that there is no other such automorphism of $X$. We get then case (1).

We look now at case $(b)$, where $\rk(\Pic(X)^{\hat g})=2$. In this case, $g$ is conjugate to an element of $\Aut(S(\R),\pi)$ by some birational morphism $\varepsilon\colon X\to S$ that restricts to a diffeomorphism $X(\R)\to S(\R)$ (see Remark~\ref{ConicBAutSpi}) that we call $g$ again for simplicity.
Since the order of $g$ is finite, by Lemma~\ref{ImFinOrd} the image of $g$ under the map $\Phi\colon\Bir(S,\pi)\to\Aut(\Pp^1)$ is the identity or $\eta\colon z\mapsto -z$, after conjugation by an element of $\Aut(S(\R),\pi)$. 
\begin{itemize}
\item If $\Phi(g)$ is the identity, then $g\in\Aut(S(\R)/\pi)$. When $g$ has order larger than $2$, by Lemma~\ref{RotationsinH} $g$ is conjugate to a rotation, case (3).

\noindent If $g$ has order $2$, then $g$ is an element in $\Aut^+(S(\R)/\pi)$ when $g$ is an orientation-preserving birational diffeomorphism or an element that belongs to $\Aut(S(\R)/\pi)\setminus \Aut^+(S(\R)/\pi)$ otherwise.  Proposition~\ref{diff} implies in the first case, that $\Fix(g)$ is a double covering of $\Pp^1$ with no real points and in the second case, that $\Fix(g)$ is a double covering of $\Pp^1$ with real points one oval and ramification points $P_N$ and $P_S$. Lemma~\ref{ExchangeOrNotBir} implies that $P_N$ and $P_S$ are fixed in both cases. By Lemma~\ref{Lem:LocRot}, the action of $g$ on the fibres of $\pi$ is either by rotations of order $2$ when $g$ is in $\Aut^+(S(\R)/\pi)$ or by reflections when $g$ is in $\Aut(S(\R)/\pi)\setminus\Aut^+(S(\R)/\pi)$. We get thus cases (6) and (7) in the statement, except if the curve $\Fix(g)$ is rational. It remains to see that if $\Fix(g)$ is rational, $g$ is conjugate to another case. If $g\in\Aut(S(\R)/\pi)\setminus\allowbreak\Aut^+(S(\R)/\pi)$, then the curve $\Fix(g)$ is isomorphic to $\Pp^1_\R$ and $g$ is conjugate to the reflection $\upsilon\colon(w:x:y:z)\mapsto(w:-x:y:z)$ by Theorems~\ref{thm} and~\ref{thm4}. If $g\in\Aut(S(\R)/\pi)$, then $g$ is conjugate to an automorphism of the last family by Lemma~\ref{Lem:SpecialMapsg1g2}.
\item If $\Phi(g)=\eta$, then $g=g'\tilde\eta$ with $\Phi(\tilde\eta)=\eta$ (Lemma~\ref{ImFinOrd}) and $g'\in\Aut(S(\R)/\pi)$. Since the order of $g$ is prime, $g$ is of order $2$ in $\Aut(S(\R),\pi)\setminus\Aut(S(\R)/\pi)$ giving the case (8) in the statement, or one of the automorphisms $(w:x:y:z)\mapsto(w:\pm x:\pm y:-z)$.\qedhere
\end{itemize}
\end{proof}

\begin{proof}[Proof of Theorem $\ref{MainThm2}$]

All the cases are disjoint because of the fixed curves and order, except maybe in case (2) where the curve of fixed points of $\alpha_i$ has genus $1$ because elements in cases (6) and (7) may a have curve of fixed points of the same genus. However, $\alpha_i$ is not conjugate to an automorphism of a conic bundle since there is no sequence of links coming from it to a Mori fibration preserving a conic bundle (Lemma~\ref{links}). On the other hand, $\alpha_i$ is conjugate to another element if and only if the conjugation is by an isomorphism of the surface $X$; this is again a consequence of Lemma~\ref{links}. We proved that conjugacy classes in (2) are disjoint and parametrised by isomorphism classes of pairs $(X,g)$, where $X$ is a Del Pezzo surface of degree $4$ with $X(\R)\simeq S(\R)$ and $g$ is an automorphism of order $2$ that does not preserve any real conic bundle (Proposition~\ref{dP4rank1}).

It remains to show the parametrisation of the families $(1)$ and $(3)-(8)$. 

For (1), the curves of fixed points in $S(\C)$ are not rational and invariant under conjugation in $\Bir(S)$ and then in $\Aut(S(\R))$. We obtain a map from the set of conjugacy classes associated to each family to the set of isomorphism classes of the set of fixed curves.
The surjectivity is given by the correspondence 
\[\left\{
   \begin{array}{c}
      \mbox{Smooth real quartics}\\
      \mbox{with one oval}
   \end{array}
\right\}
\leftrightarrow
\left\{
   \begin{array}{c}
      \mbox{Del Pezzo surfaces of degree 2}\\
      \mbox{diffeomorphic to the sphere}
   \end{array}
\right\}\] 
Concerning injectivity, if two quartics are isomorphic, then the surfaces are isomorphic. This is because the canonical divisor of the quartic is the class of a line (see proof Proposition~\ref{Prop:Geiser}). Then every isomorphism extends to $\Pp^2$ and then, it yields an isomorphism of Del Pezzo surfaces of degree 2.

For (6) and (7), the elements are conjugate in $\Aut(S(\R))$ if and only if they are conjugate in $\Aut(S(\R))$, because it is not possible to use other links that links of type $\mathrm{II}$ (see the description of links given in Lemma~\ref{links}). We can thus consider the fixed locus, which is not only a non-rational curve, but also a curve endow with a $2\colon 1$-covering. Moreover, the elements of $\Aut(S(\R),\pi)$ preserve the interval. Conversely, let $\Gamma\to\Pp^1$, $\Gamma'\to\Pp^1$ be $2:1$-coverings of $\Pp^1$ and assume that there exists an isomorphism $\alpha\colon\Pp^1\to\Pp^1$ such that the following diagram commutes:
\[\xymatrix@R-0.8pc{
\Gamma\  \ar[d]_\pi \ar[r]^{\rho}_{\sim} & \Gamma' \ar[d]_\pi  \\
\Pp^1\ \ar[r]^\alpha_\simeq & \ \Pp^1
}\]
and that $\alpha$ preserves $[-1,1]$ then $\alpha$ is in the group given in Lemma~\ref{imphigrande}, then there exist $\xi\in\Aut(S(\R),\pi)$ such that we replace $\rho$ with $\xi \rho\xi^{-1}$ and may assume that $\alpha=id$. Then the corresponding elements are conjugate by Proposition~\ref{correspAut(S(R))}.

For (4) and (5), the parametrisation is trivial since there is only one element in each family.

For (3), if two rotations are equal up to sign, they are conjugate by $\upsilon$ or the identity. It remains to see that if $r_\theta$ is conjugate to $r_{\theta'}$ by $\rho\in \Aut(S(\R))$ then $\theta=\pm\theta'$ $\pmod{2\pi}$. We may assume that the order is $\geq 5$, (since otherwise we always have $\theta=\pm \theta$ $\pmod{2\pi}$). We decompose $\rho$ into elementary links and use Lemma~\ref{links} to see that $\rho$ is a product of maps of the following type:
\[\xymatrix@R-0.8pc{
dP_6\  \ar[d] \ar@{-->}[r]^{\mathrm{II}} & \ dP_6 \ar[d]  \\
S\ar[r]^{\simeq}  & S
}\]
where the vertical arrows are blow-ups of two imaginary fixed points, fixed by $g$ and the image. Hence, we may assume that the points are  $(0:\pm{\bf i}: 1:0)$ and then we stay in $\Aut(S(\R),\pi)$ (Lemma~\ref{Aut(dP6)}). In $\Aut(S(\R)/\pi)\rtimes\langle\tau\rangle\subset\PGL(2,\C)\rtimes\langle\tau\rangle$ the elements are $\left(\left[\begin{smallmatrix}1& 0\\ 0& e^{{\bf i}\theta}\end{smallmatrix}\right],1\right)$ (see Subsection~\ref{SecRotation}), and two are conjugate only if $\theta=\pm\theta'.$

For (8), by Corollary~\ref{ColPropInv}, conjugacy classes of elements in $\Aut(S(\R),\pi)\setminus\allowbreak\Aut(S(\R)/\pi)$ surjects naturally to the set of conjugacy classes of elements in $\Bir(S,\pi)\setminus\allowbreak\Bir(s/\pi)$ which is uncountable. These correspond to the conjugacy classes of $\Bir(S,\pi)$, we may then have a priori more conjugacy classes in $\Aut(S(\R),\pi)$.
It remains to prove that two such elements are conjugate in $\Aut(S(\R),\pi)$ if and only if they are conjugate in $\Aut(S(\R))$. For this, we write $\rho\in \Aut(S(\R))$ an element that conjugates one involution to another, and decompose it into elementary links. If all links are of type $\mathrm{II}$, then $\rho\in \Aut(S(\R),\pi)$. If some links of type $\mathrm{I}$ or $\mathrm{III}$ are used, then by Lemma~\ref{links} these pass through the sphere and the Del Pezzo of degree $6$, which is impossible here, since elements of the last family are not conjugate to $(w:x:y:z)\mapsto(w:\pm x:\pm y:-z)$ by hypothesis. The last part is when $\rho$ decomposes into links of type $\mathrm{II}$ and $\mathrm{IV}$. The links of type $\mathrm{IV}$ provide two fibrations of the same surface, which lead to two different elements of $\Aut(S(\R),\pi)$. If the two elements are conjugate in this latter group, the result is clear. The only case where this is not true is by Lemma~\ref{links} the case given by the automorphisms $g_1$, $g_2$ on special Del Pezzo surfaces of degree $4$ given by $\lvert \mu\vert=1$ (Lemma~$\ref{Lem:g1g2DP4}$). But in this case, we conjugate an element of $\Aut(S(\R),\pi)\setminus \Aut(S(\R)/\pi)$ to an element of $\Aut(S(\R)/\pi)$, and when we come back we did not change the conjugacy class in $\Aut(S(\R),\pi)$ (Lemma~\ref{Lem:SpecialMapsg1g2}). This ends the proof of the Theorem~\ref{MainThm2}.
\end{proof}
\printbibliography
\end{document}